\documentclass[10pt,reqno,english]{amsart}
\usepackage{amsfonts,amsmath,latexsym,verbatim,amscd,mathrsfs,color,array}
\usepackage[colorlinks=true, citecolor=green!50!black]{hyperref}

\usepackage{amssymb}
\usepackage{pdfsync}
\usepackage{enumerate}
\usepackage{todonotes}

\usepackage{graphicx}
\usepackage{subcaption}
\graphicspath{{pictures/}}
\usepackage{float}

\newcommand{\Sd}{\mathbb{S}^2}
\newcommand{\Sn}{\mathbb{S}^{N-1}}
\newcommand{\Su}{\mathbb{S}^1}
\newcommand{\rn}{{\mathfrak R}_n  }

\newcommand{\N}{\mathbb{N}}

\newcommand{\R} {\mathbb R}

\newcommand{\cuad}{{\sqcap\kern-.68em\sqcup}}

\newcommand{\ve}{\varepsilon}

\newcommand{\be}{\begin{equation}}
	\newcommand{\ee}{\end{equation}}

\newcommand{\la}{\lambda}

\newtheorem{theorem}{Theorem}

\newtheorem{lemma}{Lemma}[section]
\newtheorem{prop}[lemma]{Proposition}

\theoremstyle{definition}
\newtheorem{remark}[lemma]{Remark}

\newcommand{\bremark}{\begin{remark} \em}
	\newcommand{\eremark}{\end{remark} }

\numberwithin{equation}{section}

\begin{document}
		
	\title{Bifurcations in Isoperimetric Problems with Nonlocal Interactions}
		
	\author[F. De Regibus]{Fabio De Regibus}
	\address{\noindent F.D.R.:  Department of Mathematical Sciences University of Bath, Bath BA2 7AY, United Kingdom.}
    \email{fdr28@bath.ac.uk}
	
	\author[M. Grossi]{Massimo Grossi}
	\address{\noindent M.G.: Dipartimento di Scienze di Base Applicate per l'Ingengeria Sapienza Università di Roma, via Antonio Scarpa 16, 00161 Roma, Italy}
	\email{massimo.grossi@uniroma1.it}
	
	\author[M. Musso]{Monica Musso}
	\address{\noindent M.M.:  Department of Mathematical Sciences University of Bath, Bath BA2 7AY, United Kingdom.}
	\email{mm2683@bath.ac.uk}

\thanks{ Work partially supported by IES/R3/233057 - Royal Society International Exchanges 2023 Round 3. FDR and MG are partially supported by INDAM}

\begin{abstract}
We study isoperimetric problems modeled on the liquid drop model, with nonlocal interactions under a volume constraint. While balls are natural critical points, we show that, for an unbounded sequence of radii, non-spherical solutions bifurcate from the family of balls. These new solutions lie arbitrarily close to balls and can have arbitrarily large volume. Conversely, at radii outside this sequence, no bifurcation occurs, and nearby solutions are trivial, arising only from rigid motions. 
\end{abstract}

\maketitle

\section{Introduction and main results}
\label{sec_intro}

Let $\Omega \subset \R^N$ be a smooth set, with $N\geq 2$, and consider the energy functional
\begin{equation}\label{1}
E(\Omega)=\mathrm{Per}(\Omega)+D(\Omega).
\end{equation}
Here the perimeter is defined by
\[
\mathrm{Per}(\Omega)=\sup\left\{\int_\Omega \mathrm{div}\,\Psi\,dx:\,\Psi\in C^1_c(\R^N,\R^N),\ \|\Psi\|_{L^\infty(\R^N)}\le 1\right\},
\]
and it coincides with the surface area of $\partial\Omega$ when $\Omega$ is smooth.
 The second term in the energy functional represents a nonlocal interaction. In dimension $N\geq2$ we consider 
\begin{equation}\label{defD}
D(\Omega)=\frac12\int_\Omega\int_\Omega \frac{1}{|x-y|^\lambda}\,dx\,dy,
\end{equation}
for  $0<\lambda <N-1$. Note that this  coincides with Coulomb interaction when $\lambda = N-2$.
In dimension $N=2$ we also consider
\begin{equation}\label{defD2}
D(\Omega)=\frac{1}{2\pi}\int_\Omega\int_\Omega \log\!\left(\frac{1}{|x-y|}\right)\,dx\,dy.
\end{equation}
We study critical points of the energy functional $E$ under a volume constraint. While every ball of radius $R$ naturally constitutes a critical point, our focus is on the existence of new, non-spherical solutions. Specifically, we show that for certain distinguished values of $R$, branches of solutions bifurcate from the family of balls, revealing a richer structure of equilibria beyond the classical spherical configurations.\\

In dimension $N=3$ and considering a Coulomb-type interaction corresponding to $\lambda = 1$ in~\eqref{defD}, the energy functional arises in the classical liquid drop model introduced in~\cite{bohr,Gamow1930} to describe atomic nuclei. In this model the nucleus is represented by a region $\Omega$ of incompressible fluid with nucleons (protons and neutrons) uniformly distributed.
Under this assumption, the number of nucleons is proportional to the volume $|\Omega|$. The perimeter term in~\eqref{1} models surface tension, which tends to hold the nucleus together, while the nonlocal term $D(\Omega)$ accounts for the Coulomb repulsion between protons.
	
The basic variational problem for equilibria of this system is the following: given $m>0$, find regions $\Omega\subset\R^3$ with smooth boundary that are critical points of $E$ under the volume constraint $|\Omega|=m$. This problem, introduced by Bohr and Wheeler~\cite{BohrWheeler1939} in connection with nuclear fission, leads to the Euler-Lagrange condition that for some constant $\mu$
\begin{equation}
H_{\partial\Omega}(x)+\int_\Omega \frac{1}{|x-y|}\,dy=\mu,
\qquad \text{for all } x\in\partial\Omega,
\label{eq:equation}
\end{equation}
where $H_{\partial\Omega}$ denotes the mean curvature of the boundary.

As already observed, balls of volume $m$ are always solutions: balls minimize the perimeter under the constraint $|\Omega|=m$ by the classical isoperimetric inequality (see~\cite[Theorem 14.1]{Maggi2012}), while they maximize the Coulomb interaction term~\cite[Theorem 3.7]{LiebLoss1997}. The competition between these two effects makes problem~\eqref{eq:equation} delicate; see~\cite{ChoksiMuratovTopaloglu2017} for a survey of results.

In the small-volume regime $m\to0$, the perimeter term dominates, while the Coulomb interaction becomes dominant for large volume $m\gg1$. A simple scaling argument clarifies the competition between the two terms in the energy.
Indeed, let $\Omega$ be a region with $|\Omega|=m$ and write $\Omega=m^{1/3}\mathcal{M}$, so that $|\mathcal{M}|=1$. Then~\eqref{1} becomes
\[
E(\Omega)=m^{2/3}\big(\mathrm{Per}(\mathcal{M})+m\,D(\mathcal{M})\big).
\]

Kn\"upfer and Muratov~\cite{KnupferMuratov2014} proved that for sufficiently small $m>0$ the minimizer of~\eqref{1} is the ball $B$ with $|B|=m$; see also Julin~\cite{Julin2014} and Bonacini–Cristoferi~\cite{BonaciniCristoferi2014}. More recently, Chodosh and Ruohoniemi~\cite{ChodoshRuohoniemi2025} showed that this holds for all $m\le1$. 

For larger volumes the situation changes: balls cease to be global minimizers when $m>m_*$, where (see~\cite{ChoksiPeletier2011,FL15})
\[
m_*=5\,\frac{2^{1/3}-1}{1-2^{-2/3}}\approx3.51.
\]
The threshold $m_*$ is defined as the volume for which a single ball and two equally sized balls at infinite separation have the same energy. 
Frank and Nam~\cite{FrankNam2021} showed that no global minimizer exists for $m\ge 8$, improving earlier results by Lu and Otto~\cite{LuOtto2014}, see also~\cite{FrankKillipNam2016}. Choksi and Peletier~\cite{ChoksiPeletier2011} conjectured that global minimizers exist precisely for $m\in(0,m_*]$, and they are balls. This was confirmed in~\cite{FrankNam2021} for $0<m\le m_*$, assuming no minimizer exists for $m>m_*$.

Very few results address the existence of non-spherical solutions to Problem~\eqref{eq:equation} for $m>m_*$. Beyond balls, the only known explicit constructions for large $m$ are the rotational tori and double tori in~\cite{RenWei1,RenWei2}.  More recently,~\cite{delPinoMussoZuniga2025} constructed a family of compact, embedded solutions with large volume, forming a striking “pearl necklace” along a large circle, with cross-sections closely approximating Delaunay unduloids. These solutions exhibit dihedral symmetry in the plane of the circle, remaining invariant under rotations by angles of ${2\pi \over k}$, for all $k$ large.

Balls remain local minimizers well beyond $m_*$: they are linearly stable for $0<m\le 10$ and lose stability for $m>10$. At $m=10$  a bifurcation occurs, giving rise to non-spherical local minimizers for $m>10$ and saddle points for $m<10$. In~\cite{Fr19}, Frank established the existence of a smooth family of nonspherical, volume-constrained critical points near $m=10$, bifurcating from the ball at which it loses stability.  These sets are cylindrically symmetric, changing from prolate shapes below $m=10$ to oblate shapes above, with energy above that of balls for smaller volumes and below for larger volumes. At $m=10$, an exchange of stability takes place, with balls losing stability while the nonspherical sets become stable.
These configurations were anticipated by Bohr and Wheeler~\cite{BohrWheeler1939} as intermediate states in the decay of a ball into two. They correspond to volume-preserving deformations of a ball with volume between a critical value and $10$, where the energy along such paths first increases and then decreases. 
From a numerical perspective, we would like to mention the work~\cite{XuDu2023} where different bifurcation branches are traced, uncovering toroidal and Delaunay-like shapes for large volumes and providing illustrative figures.

We also refer to the works~\cite{FigalliFuscoMaggiMillotMorini2015,KM14}, which address the cases of dimension $N\not= 3$ and the Coulomb interaction depends on a parameter $\lambda\not= 1$.

\medskip

In the case of dimension $N=2$, the energy functional $E$ does not admit global minimizers, as the logarithmic Coulomb interaction~\eqref{defD2} is unbounded. Nevertheless
one can still consider critical points of $E$ under the volume constraint. This type of 
 logarithmic Coulomb interaction arises, for instance, in the study of pattern formation in some growth and inhibition process in biological systems. We refer to~\cite{RZ25} for several interesting phenomena occurring in this framework.
See also~\cite{RWW25} for a related functional in which the nonlocal term involves a perturbed logarithmic interaction.

\vspace{20pt}
 
The aim of this paper is to establish the existence of an unbounded sequence of radii at which bifurcations from the ball with those radii occur in the liquid drop model, thereby extending Frank's result~\cite{Fr19}. We prove that the same phenomenon occurs in dimension $N=2$ for the Coulomb interaction~\eqref{defD2}. More generally, for the nonlocal interaction~\eqref{defD}, we show that in every dimension $N\ge 2$ there is an unbounded sequence of radii at which bifurcations from the sphere take place.

More precisely, we prove the existence of a diverging sequence $(\rn)_n \subset \mathbb{R}^+$  such that nontrivial branches of critical points bifurcate from the balls centered at the origin and of radius $\rn$, denoted by  $B_{\rn}$.  As a consequence, for any given $m>0$, there exist solutions with volume exceeding $m$ that are not balls but lie arbitrarily close to balls. Conversely, we show that if $R \neq \rn$, no bifurcation occurs, and therefore no solutions exist in a neighbourhood of the ball of radius $R$, except for the trivial ones obtained by rigid motions of the ball.

\medskip

The domains we obtain are star-shaped with respect to the origin and have the form
\begin{equation}\label{Omegavarphi}
\Omega_\varphi  = \left\{ x \in \mathbb{R}^N : |x| < \varphi \!\left( \frac{x}{|x|} \right) \right\}, \quad \varphi (z) = R + u \!\left( z \right)\ge0,
\end{equation}
for some $R>0$ and small function $u : \mathbb{S}^{N-1} \to \mathbb{R}$ with $u \in C^{2,\alpha}(\mathbb{S}^{N-1})$ for some $\alpha \in (0,1)$. Observe that when $\varphi (z) \equiv \rn $, then $\Omega_\varphi$ is the ball of radius $B_{\rn}$. We consider non-constant functions $u$ whose symmetry properties depend on both the dimension $N$ and the radius $\rn$.

\medskip

We present our results in the following order: first the liquid drop model ($N=3$ and $\lambda = 1$); next the case of the logarithmic Coulomb interaction $D$ as in~\eqref{defD2} in dimension $N=2$; and finally the general case with $N \ge 2$ and $\lambda \in (0, N-1)$ in~\eqref{defD}.
We treat these cases separately for two reasons. First, the bifurcation radii take different forms depending on the dimension and on the type of interaction. Second, the nature of the bifurcation results differs across dimensions.

\medskip

\subsection{The liquid drop model case: \texorpdfstring{$N=3$}{N=3} and \texorpdfstring{$\lambda =1$}{lambda=1}}

In this case, the sequence of radii where bifurcation occurs is given by
\begin{equation}\label{Rn3}
\rn=\left(\frac{3(2n+1)(n+2)}{8\pi}\right)^{1/3}, \quad \text{for all }  n \in \N, \quad n \geq 2.
\end{equation}
Observe that for $n=2$
\[
{\mathfrak R}_2= \left({15 \over 2 \pi}\right)^{1/ 3}, 
\]
so that we recover the value  $m=10$ for the volume at which  bifurcation has been established in~\cite{Fr19}.

To state our result, we introduce 
 a base of the spherical harmonics on $\Sd$ of order $n$, which can be written in spherical coordinates - $\phi\in[0,2\pi)$ and $\theta\in[0,\pi]$ - as follows
\begin{equation}
\label{harm-3d}
\begin{gathered}
\mathfrak I_n(\theta)=c_n P_n(\cos(\theta)),\\
\mathfrak L^1_{n,m}(\phi,\theta)=c_{n,m}\sin(m\phi)\sin^m(\theta)Q_{n-m}(\cos(\theta)),\quad m=1,\dots,n,\\
\mathfrak L^2_{n,m}(\phi,\theta)=c_{n,m}\cos(m\phi)\sin^m(\theta)Q_{n-m}(\cos(\theta)),\quad m=1,\dots,n,
\end{gathered}
\end{equation}
where $c_n,c_{n,m}\in\R$, $P_{n}$ is the Legendre polynomial of degree $n$ and $Q_{n-m}$ is a polynomial of degree $n-m$ given by
\[
Q_{n-m}(t)=\frac{d^m}{dt^m}P_n(t).
\]
We refer to Lemma~\ref{lemma:sphar} below for more details.

The following result describes   our bifurcation result in dimension $N=3$ and for $\lambda =1$, for which critical points of the energy functional $E$ under volume constraint correspond to solutions of the Euler-Lagrange equation as in~\eqref{eq:equation}.

\begin{theorem}
\label{thm_N=3}
For all $n \in \mathbb{N}$ with $n \ge 2$, let $\rn$ be defined as in~\eqref{Rn3}. Then there exist solutions to Problem~\eqref{eq:equation} of the form~\eqref{Omegavarphi} bifurcating from the ball $B_{\rn}$. More precisely:

\begin{enumerate}[$\!\!$(1)$\!\!$]
\item For any $n$ even, there exists $\ve=\ve(n)>0$ and  two continuously differentiable curves
\[
t\in(-\ve,\ve)\mapsto R_n^t\in(0,+\infty), \quad t\in(-\ve,\ve)\mapsto u_n^t\in\mathcal C^{2,\alpha}(\Sd),
\]
such that $R_n^0= \rn$, $u_n^0=0$ and $\Omega_{R_n^t+u_n^t}$ solves~\eqref{eq:equation}. Moreover,
\begin{enumerate}[i)]
\item $u_n^t$ is invariant under the action of the group $\mathcal O(2)\times \mathcal O(1)$, 
\item As $t \to 0$
\[
R_n^t= \rn-\alpha_n t+ O(t^2),\qquad u_n^t=t \mathfrak I_n+O(t^2), \quad {\mbox {in}} \quad \mathcal C^{2,\alpha}(\Sd), 
\]
where $\mathfrak I_n$ is given by~\eqref{harm-3d} and $\alpha_n>0$ is given by~\eqref{alphan}.
\end{enumerate}
\item For any $n$  odd, there exists $\ve=\ve(n)>0$, such that for all $m\in\{(n+1)/2,\dots,n\}$ odd,  there are continuously differentiable curves 
\[
t\in(-\ve,\ve)\mapsto R_{n,m}^t\in(0,+\infty), \quad t\in(-\ve,\ve)\mapsto u_{n,m}^t\in\mathcal C^{2,\alpha}(\Sd),
\]
such that $R_n^0= \rn$, $u_{n,m}^0=0$ and $\Omega_{R_n^t+u_{n,m}^t}$ solves~\eqref{eq:equation}. 
Moreover,
\begin{enumerate}[i)]
\item $u_{n,m}^t(\phi,\theta)=u_{n,m}^t(\phi,\pi-\theta)=u_{n,m}^t(\pi/m-\phi,\theta)$,
\item as $t \to 0$
\[
R_{n,m}^t= \rn+ O(t^2),\qquad u_{n,m}^t=t \mathfrak L^1_{n,m} +O(t^2), \quad {\mbox {in}} \quad \mathcal C^{2,\alpha}(\Sd),
\]
where $\mathfrak L^1_{n,m}$ is given by~\eqref{harm-3d},
\end{enumerate}
\item For all $R\not=\rn$, there are no solutions to~\eqref{eq:equation} arbitrarily close to $B_R$ in $\mathcal C^{2,\alpha}(\Sd)$.
\end{enumerate}
\end{theorem}

\begin{figure}[H]
\centering

\begin{subfigure}{0.24\textwidth}
\centering
\includegraphics[width=\linewidth]{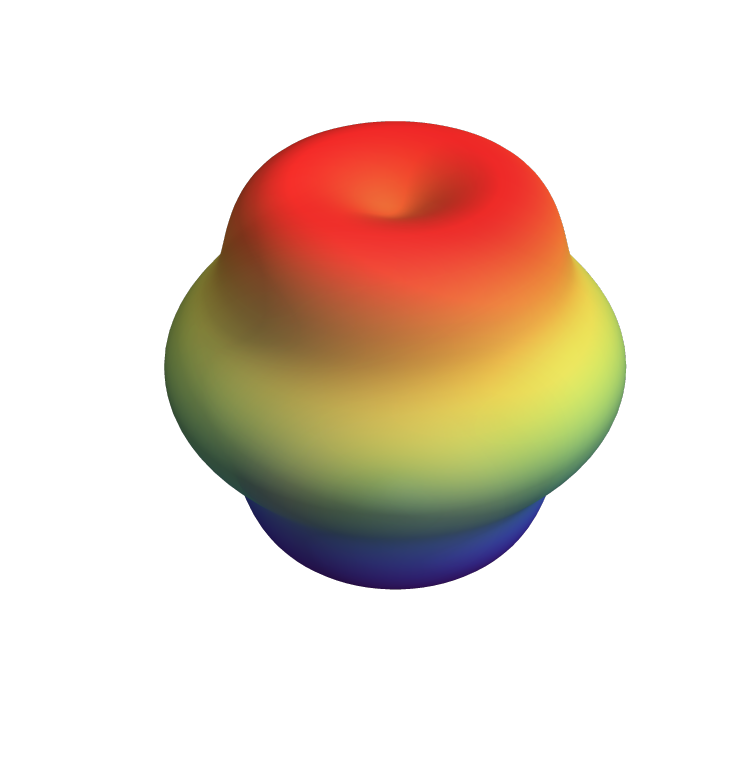}\caption*{$t=-0.03$}
\end{subfigure}
\begin{subfigure}{0.24\textwidth}
\centering
\includegraphics[width=\linewidth]{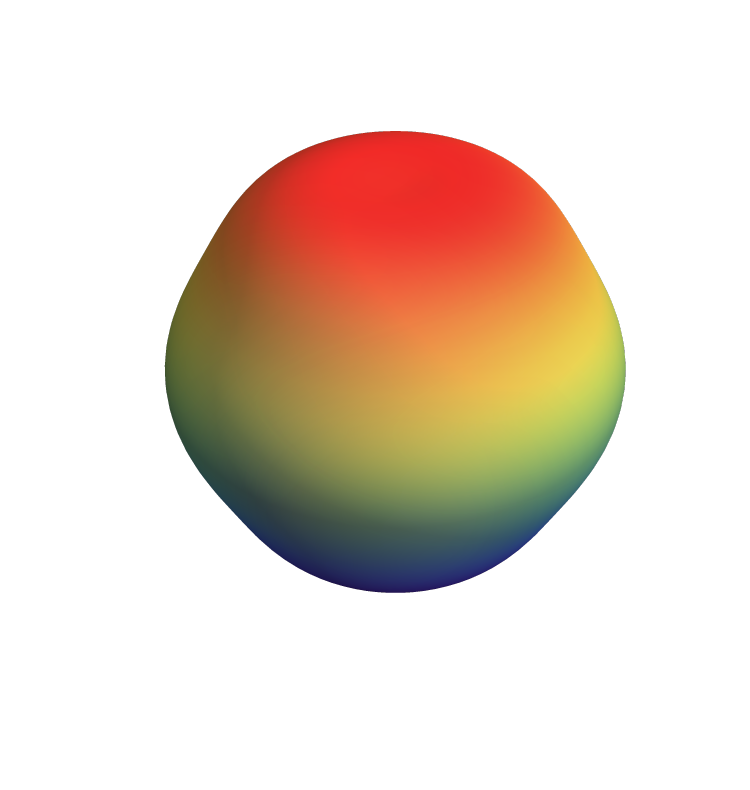}\caption*{$t=-0.01$}
\end{subfigure}
\begin{subfigure}{0.24\textwidth}
\centering
\includegraphics[width=\linewidth]{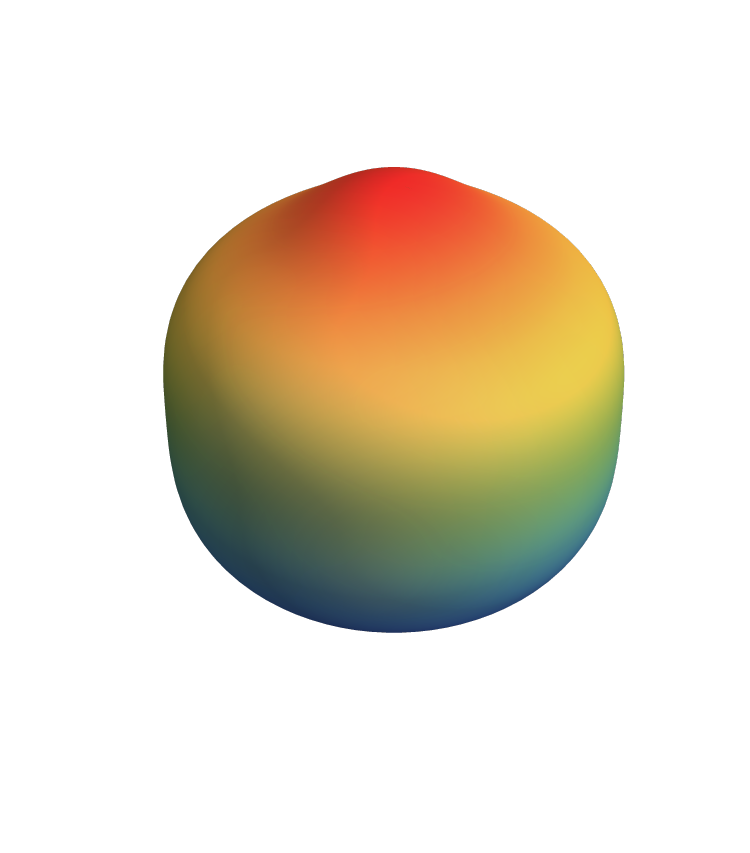}\caption*{$t=0.01$}
\end{subfigure}
\begin{subfigure}{0.24\textwidth}
\centering
\includegraphics[width=\linewidth]{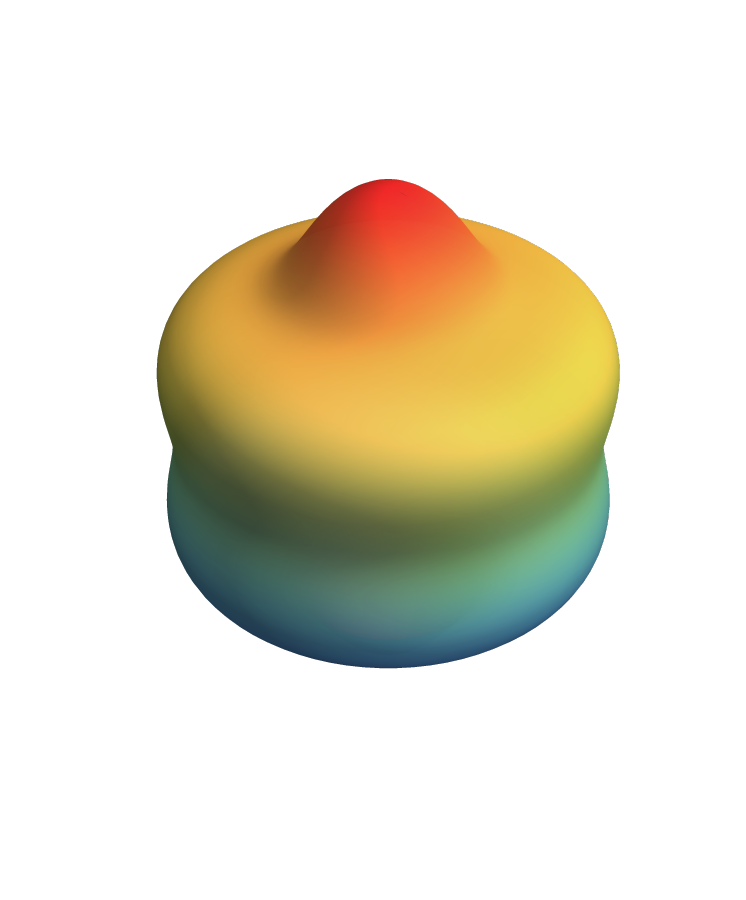}\caption*{$t=0.03$}
\end{subfigure}

\caption{The solution bifurcating at $n=6$, for positive and negative $t$.}
\end{figure}

\begin{figure}[H]
\centering

\begin{subfigure}{0.24\textwidth}
\centering
\includegraphics[width=\linewidth]{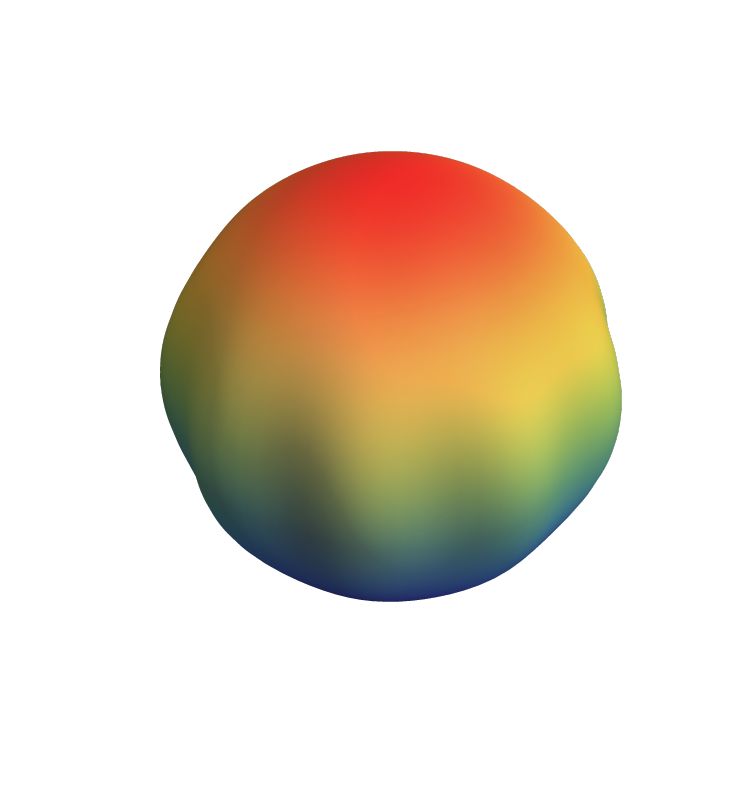}
\end{subfigure}
\begin{subfigure}{0.24\textwidth}
\centering
\includegraphics[width=\linewidth]{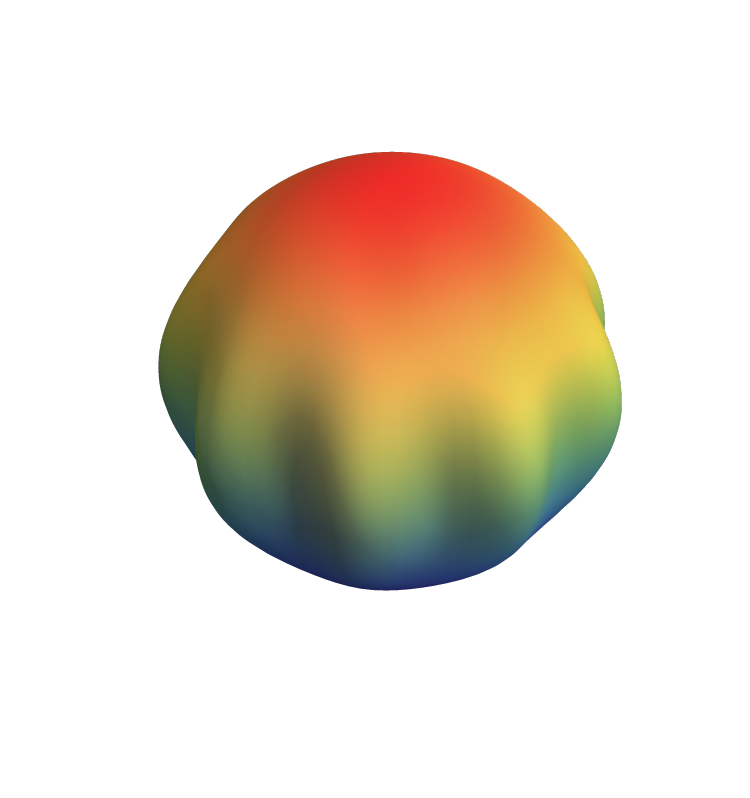}
\end{subfigure}
\begin{subfigure}{0.24\textwidth}
\centering
\includegraphics[width=\linewidth]{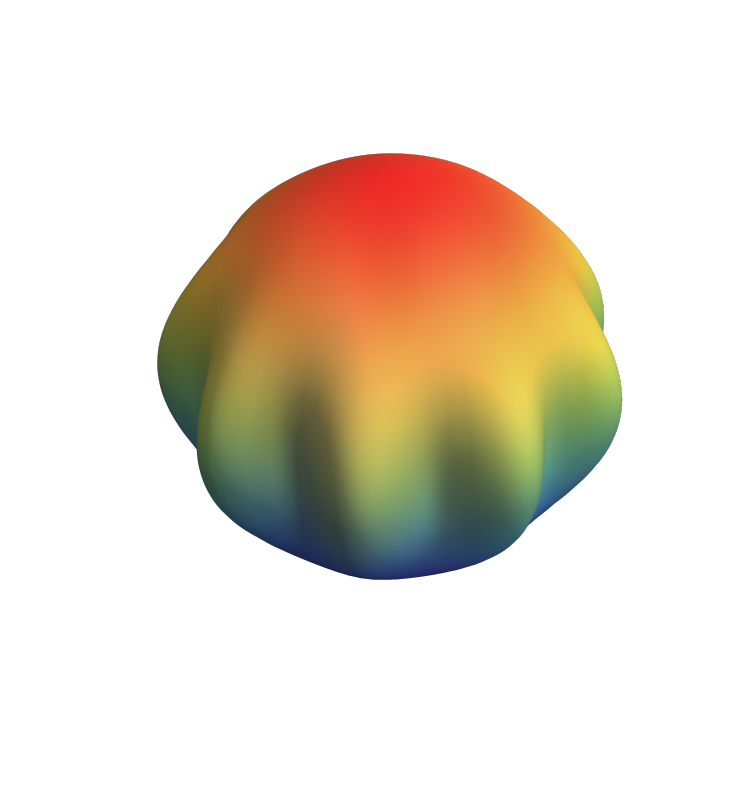}
\end{subfigure}
\begin{subfigure}{0.24\textwidth}
\centering
\includegraphics[width=\linewidth]{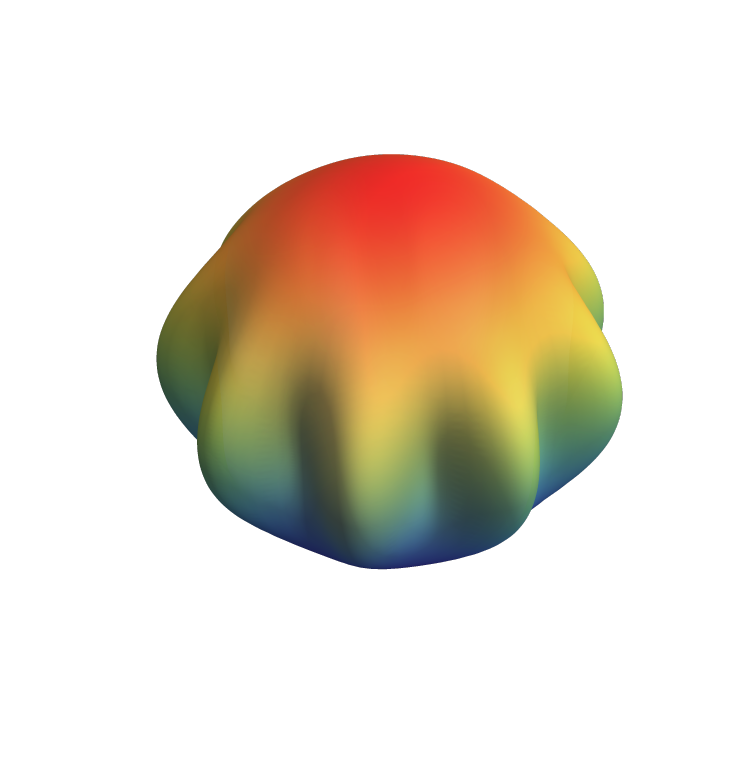}
\end{subfigure}\\

\begin{subfigure}{0.24\textwidth}
\centering
\includegraphics[width=\linewidth]{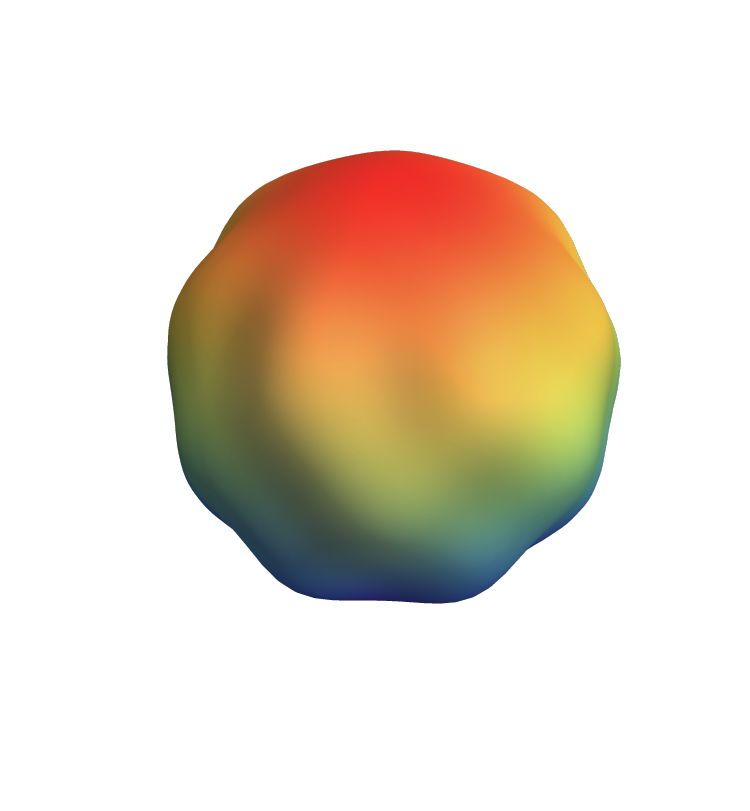}
\end{subfigure}
\begin{subfigure}{0.24\textwidth}
\centering
\includegraphics[width=\linewidth]{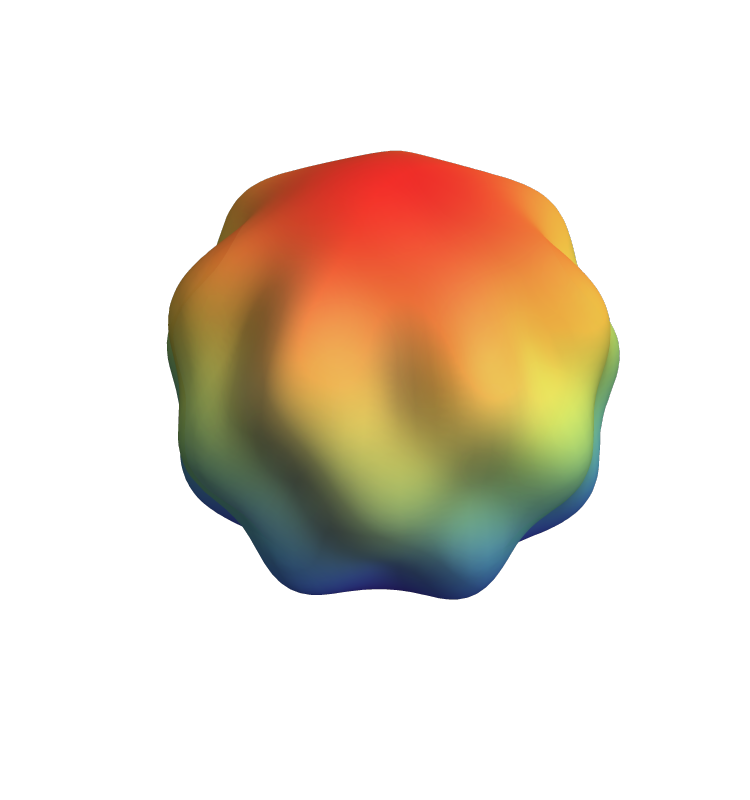}
\end{subfigure}
\begin{subfigure}{0.24\textwidth}
\centering
\includegraphics[width=\linewidth]{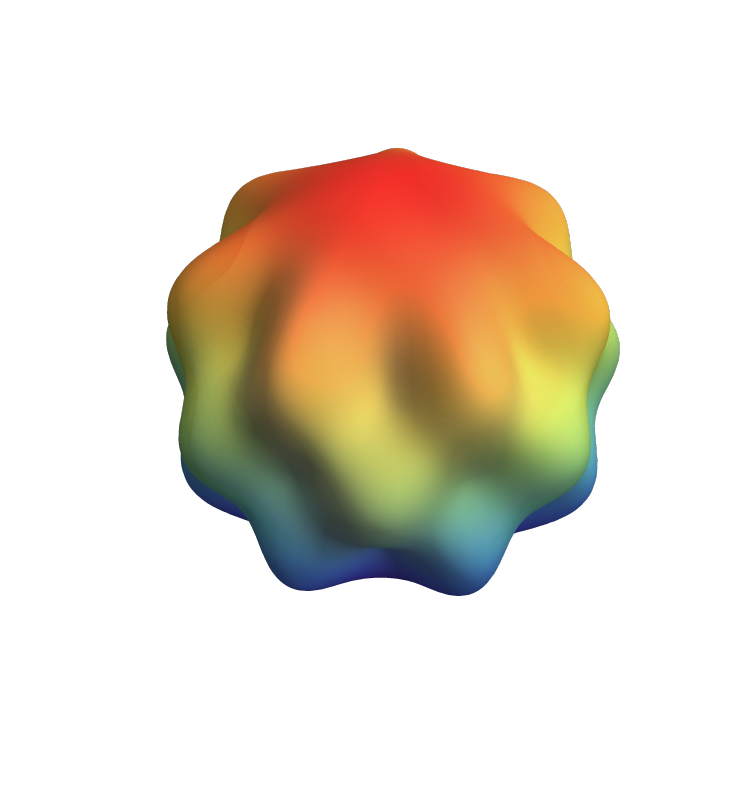}
\end{subfigure}
\begin{subfigure}{0.24\textwidth}
\centering
\includegraphics[width=\linewidth]{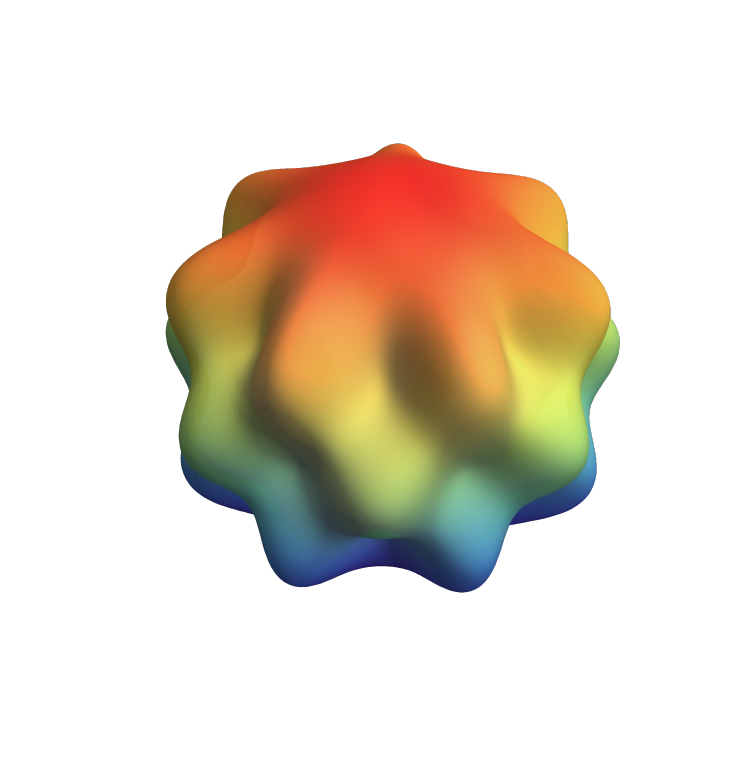}
\end{subfigure}\\

\begin{subfigure}{0.24\textwidth}
\centering
\includegraphics[width=\linewidth]{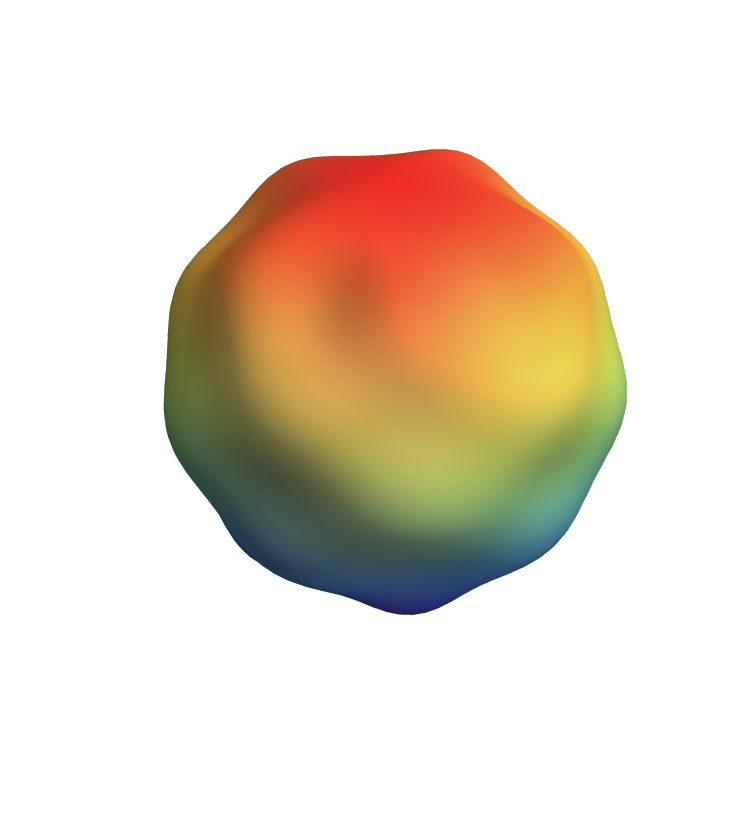}\caption*{$t=0.1$}
\end{subfigure}
\begin{subfigure}{0.24\textwidth}
\centering
\includegraphics[width=\linewidth]{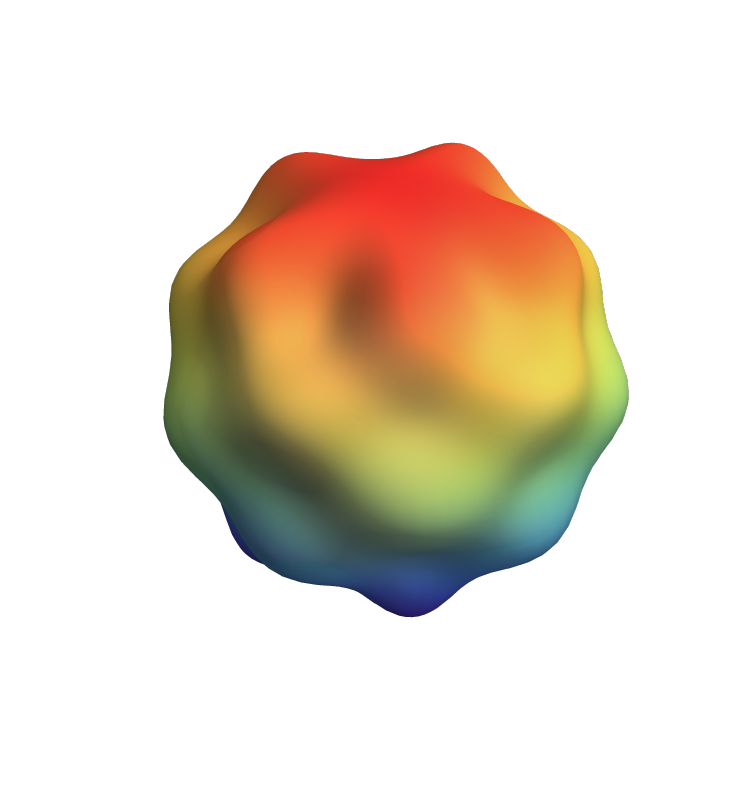}\caption*{$t=0.2$}
\end{subfigure}
\begin{subfigure}{0.24\textwidth}
\centering
\includegraphics[width=\linewidth]{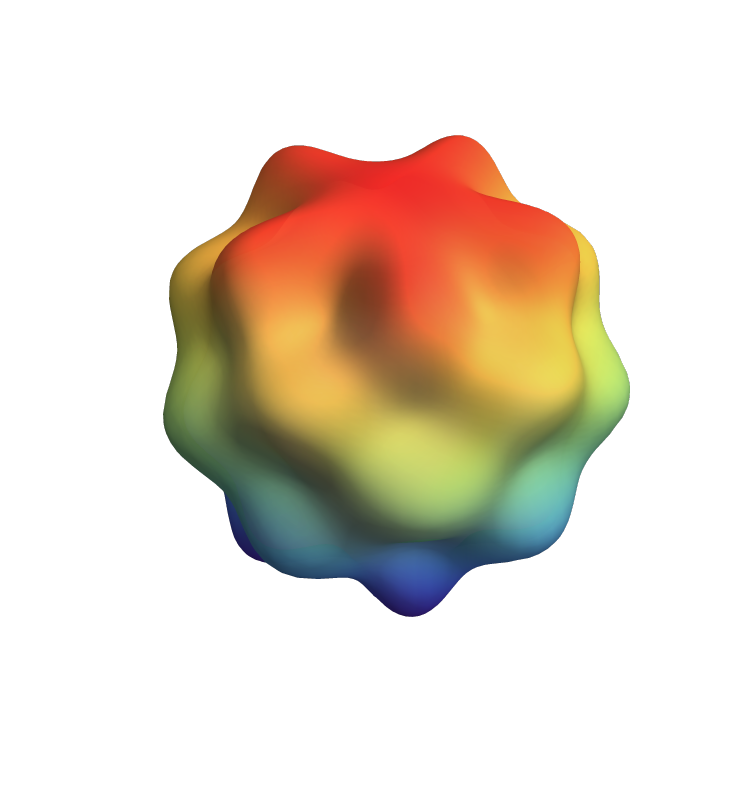}\caption*{$t=0.3$}
\end{subfigure}
\begin{subfigure}{0.24\textwidth}
\centering
\includegraphics[width=\linewidth]{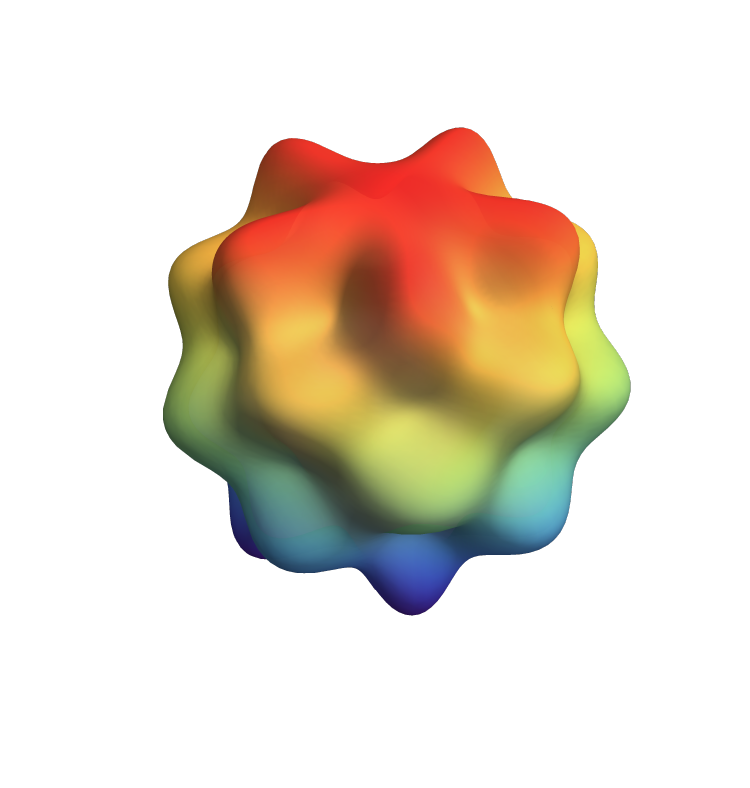}\caption*{$t=0.4$}
\end{subfigure}

\caption{The three family of solutions at $n=9$, for positive $t$.}
\end{figure}

\begin{remark}
As already mentioned, the case $n=2$ reproduces the bifurcation found in~\cite{Fr19}. Our result shows that similar bifurcations occur at all radii $\rn$ for $n$ even, with the domain shape transitioning from prolate to oblate depending on whether $R<\rn$ or $R>\rn$. Our results further extend to odd $n \ge 3$, where the solutions display distinct symmetries from the even-$n$ case. Moreover, in this case the number of solutions grows with $n$, revealing an increasingly rich bifurcation structure.

Furthermore we observe that, when $n$ is odd and for all $m$ as in the statement of the theorem, the domain 
\[
\Omega_{R_{n,m}^t+v_{n,m}^t} = \left\{ x \in \mathbb{R}^3 : |x| < R_n^t + v_{n,m}^t \left( {x \over |x|} \right) \right\},
\]
with
\[
v_{n,m}^t (\phi, \theta) = u_{n,m}^t (\phi+\pi/2m, \theta),
\]
is also a solution to~\eqref{eq:equation}. In fact, $\Omega_{R_n^t+u_n^t}$ and $\Omega_{R_n^t+v_n^t}$ differ from each other by a simple rotation. As will be clear from the proof, it is possible to find such a curve of solutions by bifurcating in the appropriate space. In this case, it is easy to see that the first-order expansion is given by $v_{n,m}^t = t \mathfrak L^2_{n,m} + O(t^2)$ in $\mathcal C^{2,\alpha}(\Sd)$. We refer to Remark~\ref{rmk-bif3} for the details.

An intriguing open question is whether the solutions constructed in~\cite{delPinoMussoZuniga2025}, which exhibit dihedral symmetry in the $(x_1, x_2)$-plane, lie on the bifurcation branches originating from 
$\rn$ for odd $n$, as suggested by Theorem~\ref{thm_N=3}. Addressing this appears to be a highly challenging problem.
\end{remark}

\begin{remark}
A different behavior of the volume of the bifurcating solutions between even $n$ and odd $n$ can be observed. Indeed, if $n$ is even, we have $\frac{d}{dt}{R_n^t}_{|t=0}<0$, and this implies that both values above and below $|B_{\rn}|$ are attained by $|\Omega_{R_n^t+u_n^t}|$ already at the linear level. In particular, the bifurcation is transcritical. On the other hand, if $n$ is odd, $\frac{d}{dt}{R_{n,m}^t}_{|t=0}=0$, and thus this change in volume cannot be observed at the linear level. It seems very plausible that in this case $\frac{d^2}{dt^2}{R_{n,m}^t}_{|t=0}\neq 0$, leading to a subcritical or supercritical bifurcation.
\end{remark}

\medskip

\subsection{The case of dimension \texorpdfstring{$N=2$}{N=2} with logarithmic Coulomb interaction}

In this case, we consider critical points of $E$ under the volume constraint. These are given by solutions to  the Euler-Lagrange equation that for some constant $\mu$
\begin{equation}
\label{eq:2}
H_{\partial\Omega}(x)+{1\over 2\pi} \int_\Omega \log \left( \frac{1}{|x-y|} \right) \, dy =\mu,
\qquad \text{for all } x\in\partial\Omega,
\end{equation}
where $H_{\partial\Omega}$ denotes the curvature of the boundary.

The sequence
of radii where bifurcation occurs is given by 
\begin{equation}\label{Rn2}
\rn=(2n(n+1))^{1/3}, \quad n \in \N, \quad n \geq 2.
\end{equation}
Then, our main result reads as follows.

\begin{theorem}
\label{thm_N=2}
For all $n\in\N$, $n\ge2$, 
let $\rn$ be defined as in~\eqref{Rn2}. Then there exist solutions to Problem~\eqref{eq:2} of the form~\eqref{Omegavarphi} bifurcating from the ball $B_{\rn}$.\\
More precisely:
 there exist $\ve=\ve(n)>0$ and two continuously diﬀerentiable curves 
\[
 t\in(-\ve,\ve)\mapsto R_n^t\in(0,+\infty), \quad t\in(-\ve,\ve)\mapsto u_n^t\in\mathcal C^{2,\alpha}(\Su),
 \]
 such that $R_n^0= \rn$, $u_n^0=0$ and $\Omega_{R_n^t+u_n^t}$ solves~\eqref{eq:2}. Moreover,
\begin{enumerate}[i)]
\item $u_n^t(\phi)=u_n^t(\phi+\pi/n)=u_n^t(2\pi-\phi)$,
\item as $t \to 0$
\[
R_n^t= \rn+ O(t^2),\qquad u_n^t(\phi)=t \sin(n\phi)+O(t^2), \quad {\mbox {in}} \quad \mathcal C^{2,\alpha}(\Su).
\]
\end{enumerate}
\medskip
Finally, for all $R\not=\rn$, there are no solutions to~\eqref{eq:2} arbitrarily close to $B_R$ in $\mathcal C^{2,\alpha}(\Su)$.
\end{theorem}

\begin{figure}[htbp]
\centering

\begin{subfigure}{0.32\textwidth}
\centering
\includegraphics[width=\linewidth]{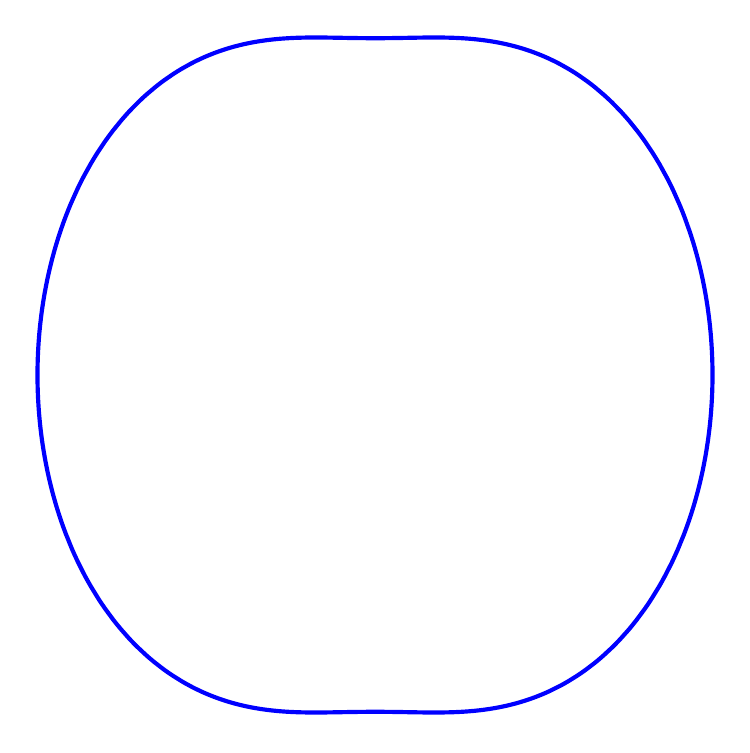}\caption*{$n=2$, $t=0.5$}
\end{subfigure}
\begin{subfigure}{0.32\textwidth}
\centering
\includegraphics[width=\linewidth]{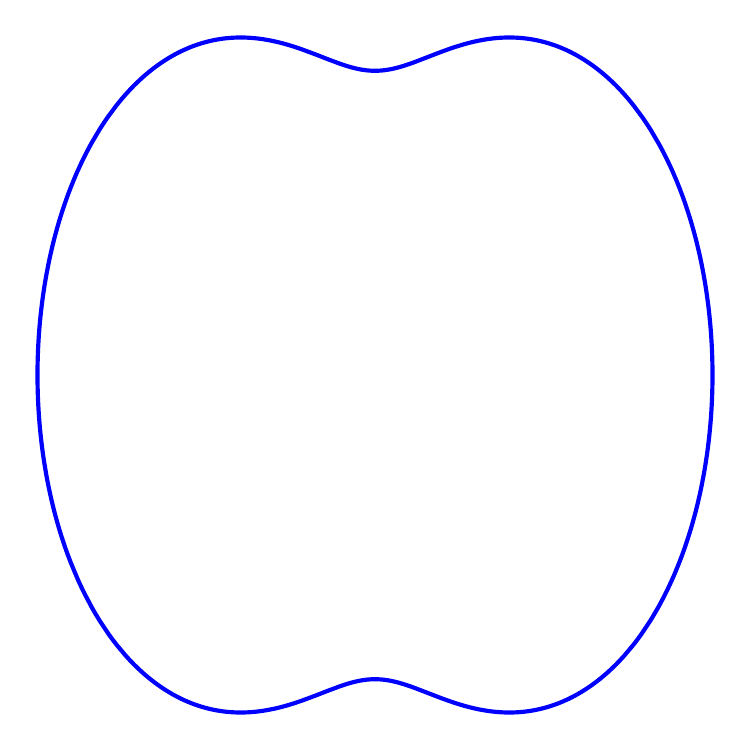}\caption*{$n=2$, $t=0.8$}
\end{subfigure}
\begin{subfigure}{0.32\textwidth}
\centering
\includegraphics[width=\linewidth]{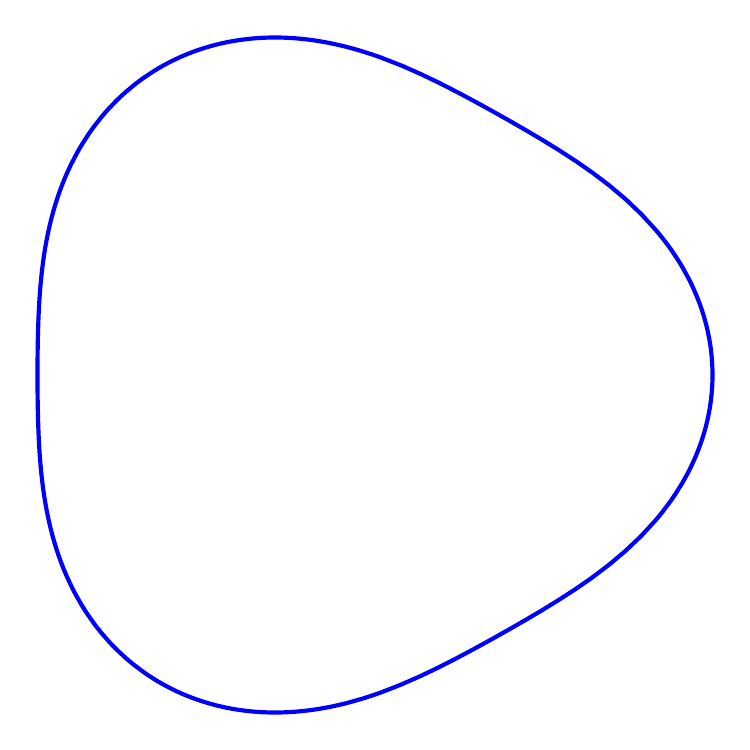}\caption*{$n=3$, $t=0.25$}
\end{subfigure}

\begin{subfigure}{0.32\textwidth}
\centering
\includegraphics[width=\linewidth]{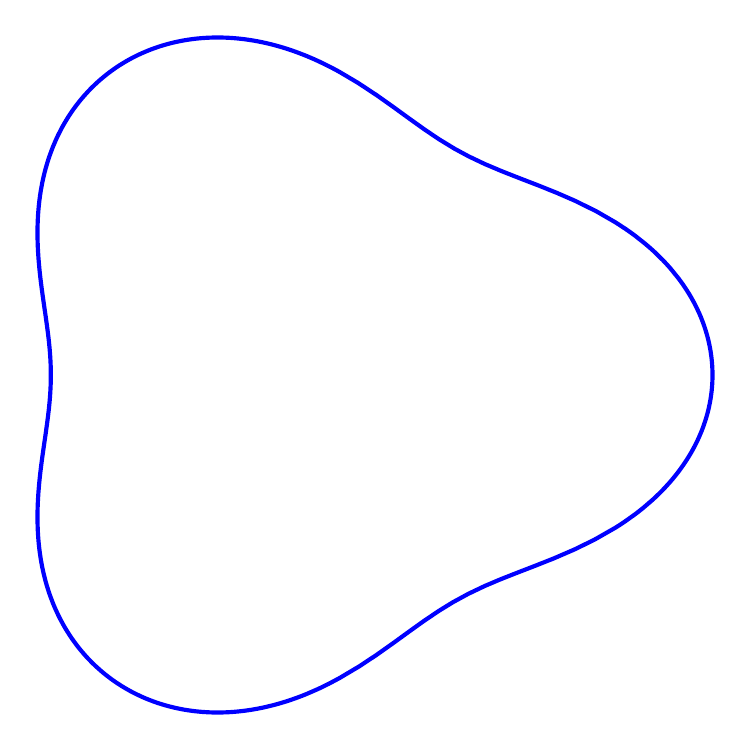}\caption*{$n=3$, $t=0.5$}
\end{subfigure}
\begin{subfigure}{0.32\textwidth}
\centering
\includegraphics[width=\linewidth]{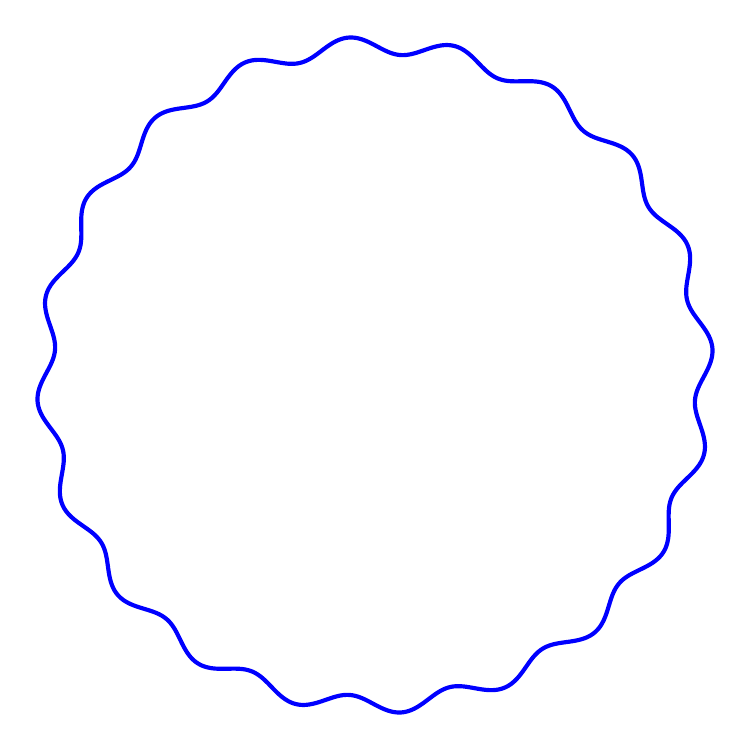}\caption*{$n=20$, $t=0.25$}
\end{subfigure}
\begin{subfigure}{0.32\textwidth}
\centering
\includegraphics[width=\linewidth]{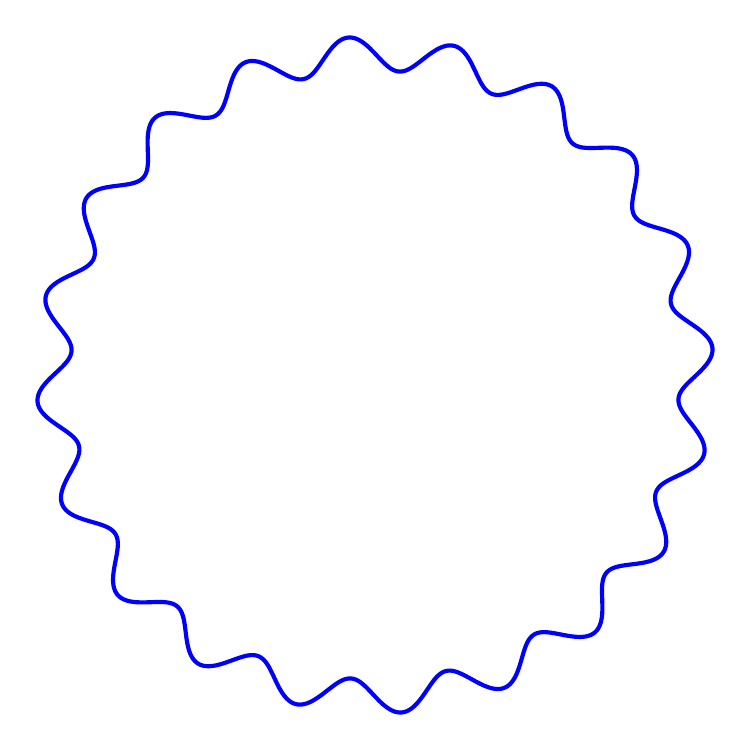}\caption*{$n=20$, $t=0.5$}
\end{subfigure}

\caption{Some examples in the $2D$ case.}
\end{figure}
\begin{remark}
In contrast to the three-dimensional case, the number of solutions at each bifurcation point $\rn$ remains the same. Indeed, the dimension of the kernel does not increase with $n$, unlike in the case of $\Sd$.

Moreover, note that, as in the odd case in dimension $N=3$, we observe that for all $n\ge 2$ the domain $\Omega_{R_n^t+v_n^t}$, with $v_n^t(\phi)=u_n^t(\phi+\pi/2n)$, is also a solution to~\eqref{eq:2}.
Again, it is possible to obtain such a curve of solutions by bifurcating in the appropriate space, and the first-order expansion of the solution is given by $v_{n,m}^t = t \cos(n\phi) + O(t^2)$ in $\mathcal C^{2,\alpha}(\Su)$, see Remark~\ref{rmk-bif2}.
\end{remark}

\begin{remark}
It is interesting to note that for $N=2$, the bifurcating solution has the same volume as the ball $B_{\rn}$ at the linear level, both for $n$ even and odd. This is in contrast with the case $N=3$, where for $n$ even the bifurcation is transcritical and the volume already changes at the linear level.

In the two-dimensional case, we believe that the volume is also preserved at second order.
\end{remark}

\medskip

\subsection{The case of dimension \texorpdfstring{$N\geq2$}{N>=2} and Coulomb interaction with parameter \texorpdfstring{$0<\lambda <N-1$}{0<lambda<N-1}}

In this case we consider critical points of the energy functional
\[
E_\lambda ( \Omega) = {\mbox {Per}} (\Omega) +D_\lambda (\Omega),
\]
where 
\[
D_\lambda (\Omega)={1\over 2} \int_{\Omega} \int_{\Omega} {1\over |x-y|^{\lambda} } \, dx \, dy,
\]
with fixed volume $m$ and, of course, any ball of volume $m$ is still a critical point of this functional under the prescribed volume constraint. The Euler Lagrange equation is
\begin{equation}\label{eq:la}
H_{\partial\Omega}(x)+\int_\Omega \frac{1}{|x-y|^\lambda }\,dy=\mu,
\qquad \text{for all } x\in\partial\Omega,
\end{equation}
where $H_{\partial\Omega}$ still denotes the mean curvature of the boundary and $\mu\in\R$.

The explicit expression of the sequence of radii where bifurcation occurs depends on the dimension $N$ and the exponent $\lambda$. For any $n \in \N$, $n\ge2$, it is given by 
\begin{equation}\label{Rnla}
\mathfrak R_n= \rn (\lambda , N) =\left(\frac{n(n+N-2)-N+1}{-(\mu_n (\lambda , N)+C({\lambda},N))}\right)^{\frac{1}{N+1-\lambda}},
\end{equation}
where
\[
C (\lambda , N) = -\lambda\pi^{\frac N2}\frac{\Gamma(N-1-\lambda)}{ \Gamma (N -{\lambda \over 2})\Gamma(\frac{N-\lambda}{2})}<0,
\]
and
\[
\mu_1(\lambda , N)=-C(\lambda , N),\qquad\mu_{n+1}(\lambda , N)=\frac{\frac\lambda2+n}{N-1-\frac\lambda2+n}\,\mu_n(\lambda , N)<0.
\]

In Lemma~\ref{lemma:autov} we show that, for $\lambda$ and $N$ fixed, $\rn$ is an increasing sequence of positive numbers that diverges to $+\infty$.
 
\medskip

In dimensions $N=2,3$, we obtain essentially the same results as in Theorems~\ref{thm_N=3} and~\ref{thm_N=2}, whereas for $N \ge 4$ it is possible to construct new solutions exhibiting additional symmetry types. This illustrates that the situation becomes more involved as the dimension of the ambient space increases. For this reason, here we only consider the case $N\ge4$.
 Our result is not intended to provide a comprehensive description of all possible bifurcations when $N \ge 4$, but rather to offer some insight into the much richer bifurcation structure that may arise in this setting.

\begin{theorem}
\label{thm_lambda}
For all $n \in \mathbb{N}$ with $n \ge 2$, let $\rn$ be defined as in~\eqref{Rnla}. Then there exist solutions to Problem~\eqref{eq:la} of the form~\eqref{Omegavarphi} bifurcating from the ball $B_{\rn}$. More precisely:
\begin{enumerate}[$\!\!$(1)$\!\!$]
\item For any  $n$  even, there exists $\ve=\ve(n)>0$ such that for all $k=1,\dots,N-1$,
there are continuously differentiable curves 
\[
t\in(-\ve,\ve)\mapsto R_{n,k}^t\in(0,+\infty), \quad t \in(-\ve,\ve)\mapsto u_{n,k}^t\in\mathcal C^{2,\alpha}(\Sd),
\]
such that $R_{n,k}^0= \rn$, $u_{n,k}^0 =0$ and $\Omega_{R_{n,k}^t+u_{n,k}^t}$ solves~\eqref{eq:la}. Moreover,
\begin{enumerate}[i)]
\item $u_{n,k}^t$ is invariant under $\mathcal O(k)\times \mathcal O(N-k)$,
\item As $t \to 0$
\[
u_{n,k}^t = t \mathfrak L_{n,k} + O(t^2), \quad {\mbox {in}} \quad \mathcal C^{2,\alpha}(\Sn)
\]
where $\mathfrak L_{n,k}$ is the unique spherical harmonic of order $n$ invariant under $\mathcal O(k) \times \mathcal O(N-k)$.
\end{enumerate}
\item For any $n$ odd, there exist $\ve=\ve(n)>0$ and two continuously differentiable curves 
\[
 t\in(-\ve,\ve)\mapsto R_n^t\in(0,+\infty), \quad t\in(-\ve,\ve)\mapsto u_n^t\in\mathcal C^{2,\alpha}(\Su),
 \]
 such that $R_n^0= \rn$, $u_n^0=0$ and $\Omega_{R_n^t+u_n^t}$ solves~\eqref{eq:la}. Moreover,
\begin{enumerate}[i)]
\item the function $u_n^t$ satisfies the following invariance properties
\[
\begin{aligned}
u_n^t&(\phi,\theta_1,\dots,\theta_{N-2})=u_n^t(\pi/n-\phi,\theta_1,\dots,\theta_{N-2})\\
&=u_n^t(\phi,\pi-\theta_1,\dots,\theta_{N-2})=\dots=u_n^t(\phi,\theta_1,\dots,\pi-\theta_{N-2}),
\end{aligned}
\]
\item as $t \to 0$
 \[
u_n^t=t \mathfrak L^{n,\dots,n,1}_{n,n}+O(t^2), \quad {\mbox {in}} \quad \mathcal C^{2,\alpha}(\Sd),
 \]
 where\footnote{see Section~\ref{devorichiamarla} for the precise definition of $\mathfrak L^{n,\dots,n,i}_{n,n}$.} in spherical coordinates
\[
\mathfrak L^{n,\dots,n,1}_{n,n}(\phi,\theta_1,\dots,\theta_{N-2})=\sin(n\phi)(\sin(\theta_1))^n\dots(\sin(\theta_{N-2}))^n.
\]
\end{enumerate}
\item For all $R\not=\rn$, there are no solutions to~\eqref{eq:equation} arbitrarily close to $B_R$ in $\mathcal C^{2,\alpha}(\Sn)$.
\end{enumerate}
\end{theorem}

\begin{remark}
As will become clear from the proof (see Section~\ref{devorichiamarla}), when $n \ge 3$ is odd it is not possible to give a general explicit formula describing additional bifurcating branches, as we did in the case $N=3$. Nevertheless, the method can be applied to specific cases, and further bifurcations may be obtained by exploiting the symmetries of spherical harmonics. This is not our focus here, since our aim is only to show that bifurcation occurs at every $\rn$, $n \ge 3$.

Moreover,  in this case as well, the domain $\Omega_{R_n^t+v_n^t}$, with $v_n^t (\phi,\theta_1,\dots,\theta_{N-2}) = u_n^t (\phi+\pi/2n,\theta_1,\dots,\theta_{N-2})$, is also a solution to~\eqref{eq:la} and can be obtained by bifurcation after imposing the appropriate symmetries. In this case, the first-order expansion of the solution is $v_{n,m}^t=t \mathfrak L^{n,\dots,n,2}_{n,n} +O(t^2)$ in $\mathcal C^{2,\alpha}(\Sn)$.

If $n$ is even and $N \ge 4$, one can find solutions that are geometrically different from those arising in the three-dimensional case.
 \end{remark}

\medskip
\noindent\textbf{Sketch of the proof.}
The general approach to prove our bifurcation results  follows that introduced in~\cite{Fr19}. We first show that any star-shaped domain of the form~\eqref{Omegavarphi} is a critical point of the energy  $E$
 under a volume constraint if the function $\varphi$ in~\eqref{Omegavarphi} is a solution of a certain partial differential equation on $\Sn$. Trivially, $\varphi = R$ is a solution, for any value of $R>0$. The bifurcation  problem is then reformulated as finding the non-trivial zeros of a suitable operator equation
 \[
 u \to \Phi (R,u), \quad \varphi = R+u.
 \]
 around certain specific values of $R$.
 By analyzing the linearized operator $D_u \Phi (R,0)$ and exploiting the structure of spherical harmonics on $\Sn$, we can apply the Crandall–Rabinowitz theorem in class of functions satisfying  appropriate symmetry conditions. This establishes the occurrence of bifurcation and provides a precise description of the solution branches in a neighborhood of the bifurcation point.
To establish the result in Theorem~\ref{thm_lambda} for $n$ even, we directly use a result of Smoller and  Wasserman in~\cite{SM90}, on the symmetry of the spherical harmonics.

To show that for $R \neq \rn$ there are no solutions close to the ball, we implement a Lyapunov--Schmidt reduction argument and prove that any such solutions are necessarily trivial, corresponding merely to translated or rescaled balls.

\medskip

\subsection*{Organization of the paper}
The rest of the paper is organized as follows. In Section~\ref{sec_N=3} we address the three-dimensional case and prove Theorem~\ref{thm_N=3}, while in Section~\ref{sec_N=2} we consider the planar case and prove Theorem~\ref{thm_N=2}. In the final Section~\ref{sec_further}, we discuss the higher-dimensional case with the  interaction depending on~$\lambda$ and prove Theorem~\ref{thm_lambda}. We collect some technical lemmas in Appendix~\ref{appendix}.

\section{The liquid drop model case: \texorpdfstring{$N=3$}{N=3} and \texorpdfstring{$\lambda =1$}{lambda=1}}
\label{sec_N=3}

We recall that we are considering set of the following type
\[
\Omega_\varphi = \left\{ x \in \R^3 \, : \, |x| < \varphi \left( {x \over |x|} \right) \right\},
\]
for some $\varphi: \Sd \to \R$, where $\varphi \in \mathcal C^{2, \alpha } (\Sd)$, with $\varphi \geq 0$.
In~\cite[Lemma 1]{Fr19}, Frank shows that $\Omega_\varphi$ is a critical point for $E$ of constrained volume, if and only if $\varphi$ is a solution of a semilinear elliptic equation on $\Sd$.
\begin{lemma}[Lemma 1 in~\cite{Fr19}]
\label{lemma1_3D}
Let $\varphi\in\mathcal C^{2,\alpha}(\Sd)$, $\varphi\ge0$. Then
\[
{d \over dt}_{| t=0} \, E \left( |\Omega_\varphi |^{1\over 3} \, {\Omega_{\varphi +t u} \over |\Omega_{\varphi+tu}|^{1\over 3}} \right) =0,
\]
for all $u\in \mathcal C^{2,\alpha}(\Sd)$, if and only if
\be
\label{eq_main}
F(\varphi ) = \mu,
\ee
where
\[
F(\varphi )= -\mathrm{div}\left(\frac{\nabla\varphi}{\varphi\sqrt{ \varphi^2 + |\nabla \varphi|^2}}\right)+\frac{3}{\sqrt{ \varphi^2 + |\nabla \varphi|^2}}-\frac{\sqrt{ \varphi^2  |\nabla \varphi|^2}}{\varphi^2}+ \,  \int_{\Omega_\varphi}    {  d x\over |x - \varphi (z)   z|} ,
\]
and
\[
\mu= {2 \over 3} \, {{\mathrm{Per}} (\Omega_{\varphi }) \over  |\Omega_\varphi |} +{5 \over 3} \,\frac{D(\Omega_\varphi) }{|\Omega_\varphi |}.
\]
\end{lemma}

We note that for every $R>0$, the function $\varphi \equiv R$ solves~\eqref{eq_main} with $\mu = F(R) = 2R^{-1} + \frac{4\pi}{3}R^2$. Equivalently, balls of arbitrary radius are solutions of the problem.

Let 
\[
\mathcal A=\{(R,u)\in(0,+\infty)\times\mathcal C^{2,\alpha}(\Sd):u>-R\}.
\]
and define $\Phi:\mathcal A\to\mathcal C^{0,\alpha}(\Sd)$ by
\be
\label{defPhi}
\Phi(R,u)= R^2 \left( F(R+u)-F(R) \right).
\ee
The operator $\Phi$ satisfies the following properties, that are proved in~\cite[Proposition 7 and 8]{Fr19}.

\begin{prop}
The operator $\Phi$ satisfies the following properties.
\begin{enumerate}
\item The map $\Phi:\mathcal A\to\mathcal C^{0,\alpha}(\Sd)$ is of class $\mathcal C^\infty$.
\item $\Phi(R,0)=0$, for all $R>0$.
\item For all $R>0$ and $v\in\mathcal C^{2,\alpha}(\Sd)$ the linearized operator at $u=0$ is given by
\[
D_u\Phi(R,0)[v]=L_R(v),
\]
where
\[
L_R(v)[z]= -\Delta_{\Sd} v(z) - 2 v(z) +R^3 \left( \int_{\Sd} {v (z') \over |z-z'| }\,d\sigma(z') -\frac{4\pi}{3} v(z)\right).
\]
\item $L_R v=\lambda v$ has a non trivial solution if and only if $v$ is any spherical harmonics on $\Sd$ of degree $n\in\N$ and $\lambda=\lambda_n$, where
\[
\lambda_n=n \, (n+1)-2-\frac{4\pi}{3}R^3\left(1-\frac{3}{2n+1}\right).
\]
\end{enumerate}
\end{prop}

The next lemma introduces the sequence of radii at which bifurcation may occur and records some of their basic properties. The proof is straightforward and therefore omitted.

\begin{lemma}
\label{lemma:autov-3d}
One has
\begin{enumerate}
\item $\lambda_1=0$, for all $R>0$, and the corresponding eigenfunctions are the spherical harmonics of degree 1,
\item for all $n\in\N$, $n\ge2$, $\lambda_n=0$ if and only if $R=\mathfrak R_n=\left(\frac{3(2n+1)(n+2)}{8\pi}\right)^{1/3}$,
\item $\mathfrak R_{n_1}<\mathfrak R_{n_2}$, for all  $n_1,n_2\in\N$, $2\leq n_1<n_2$,
\item $\mathfrak R_n\to+\infty$, as $n\to+\infty.$
\end{enumerate}
\end{lemma}

The proof of our bifurcation result relies on explicit expressions and  well-known properties of spherical harmonics, which we recall below.

To this end let us consider spherical coordinates on $\R^3$ (and, by taking $\rho=1$, on $\Sd$), that is
\begin{equation}
\label{sphe-coord}
\begin{aligned}
x_1 &=\rho\, \cos\phi \, \sin\theta, \\
x_2 &=\rho\, \sin\phi \, \sin\theta, \\
x_3 &=\rho\, \cos\theta,
\end{aligned}
\end{equation}
with $\rho>0$, $\phi \in [0,2\pi)$ and $\theta\in [0,\pi]$.

\begin{lemma}
\label{lemma:sphar}
For all $n\in\N$, the spherical harmonics of degree $n$ on $\Sd$ are given by
\begin{gather*}
\mathfrak I_n(\theta)=c_n P_n(\cos(\theta)),\\
\mathfrak L^1_{n,m}(\phi,\theta)=c_{n,m}\sin(m\phi)\sin^m(\theta)Q_{n-m}(\cos(\theta)),\quad m=1,\dots,n,\\
\mathfrak L^2_{n,m}(\phi,\theta)=c_{n,m}\cos(m\phi)\sin^m(\theta)Q_{n-m}(\cos(\theta)),\quad m=1,\dots,n,
\end{gather*}
where $c_n,c_{n,m}\in\R$, $P_{n}$ is the Legendre polynomial of degree $n$ and $Q_{n-m}$ is given by
\[
Q_{n-m}(t)=\frac{d^m}{dt^m}P_n(t).
\]
In particular, the following holds
\[
P_n(-t)=(-1)^nP_n(t),\quad\text{and}\quad Q_{n-m}(-t)=(-1)^{n-m}Q_{n-m}(t).
\]
\end{lemma}
\begin{proof}
This is well known. Indeed, the spherical harmonics of degree $n$ on $\Sd$ can be written as
\[
e^{\pm im\phi}P^m_n(\cos\theta),\quad m=0,\dots,n,
\]
where
\[
P^m_n(t)=(1-t^2)^{m/2}\frac{d^m}{dt^m}P_n(t),
\]
is the associated Legendre polynomial and $P_n$ is the Legendre polynomial of degree $n$, that by the Rodrigues' formula can be written as
\[
P_n(t)=\frac{1}{2^nn!}\frac{d^n}{dt^n}[(t^2-1)^n].\qedhere
\]
\end{proof}

We now proceed to prove Theorem~\ref{thm_N=3}. To this end, we consider the case $n$ even and $n$ odd separately.

\subsection{Proof of Theorem\texorpdfstring{~\ref{thm_N=3}}{ 1}: the case \texorpdfstring{$n$}{n} even}

In this situation, the proof follows closely~\cite{Fr19}. Consider the spaces
\[
\mathcal X=\{u\in\mathcal C^{2,\alpha}(\Sd):\,\,u(g\,\cdot)=u(\cdot),\,\forall g\in\mathcal O(2)\times \mathcal O(1)\},
\]
and
\[
\mathcal Y=\{u\in\mathcal C^{0,\alpha}(\Sd):\,u(g\,\cdot)=u(\cdot),\,\forall g\in\mathcal O(2)\times \mathcal O(1)\}.
\]
Thus we consider functions that are invariant under rotations on the $x_1x_2$-plane and even in $x_3$.
Consider the function $\Phi$ introduced in~\eqref{defPhi} and 
 restrict it to the set $\mathcal A\cap\mathcal X$. As in~\cite[Proposition 9]{Fr19}, one can  show that
\[
\mathrm{ker}( L_{\rn})=\mathrm{span}\{\mathfrak I_n\},\quad\text{and}\quad\mathrm{ran}(L_{\rn})=\left\{u\in\mathcal Y:\int_{\Sd}\mathfrak I_n\cdot\,u\,d\sigma=0\right\}.
\]
From Lemma~\ref{lemma:sphar} we recall that, in spherical coordinates, $\mathfrak I_n$ is given by
\[
\mathfrak I_n(\phi,\theta)=c_nP_n(\cos(\theta)),
\]
and $P_n$ is a polynomial of degree $n$ such that $P_n(-t)=P_n(t)$, being $n$ even.
Note that $\mathfrak I_n$ is the unique spherical harmonic of degree $1$ or $n$ which is invariant under the action of $\mathcal O(2)\times \mathcal O(1)$.
Moreover,
\[
L_{R}\mathfrak I_n=\left(n(n+1)-2-\frac{4\pi}{3}R^3\left(1-\frac{3}{2n+1}\right)\right)\mathfrak I_n,
\]
thus
\begin{equation}
\label{der-mixt1}
\frac{d}{dR}_{|R=\rn}L_R\mathfrak I_n=-3(n^2+n-2)\rn^{-1}\mathfrak I_n\not\in\mathrm{ran}(L_{\rn}),
\end{equation}
since the $\mathrm{ran}(L_{\rn})$ is orthogonal to $\mathfrak I_n$.
Thus also the transversality conditions is satisfied, so that one can apply the Crandall-Rabinowitz Theorem~\cite{CR73}, see also~\cite[Theorem I.5.1]{KieBook}. Hence, to prove part~(1) of the theorem, it remains only to compute $\frac{d}{dt}{R_n^t}_{|t=0}$. This is done in the next lemma.

\begin{lemma}
\label{lemma:derRaggio}
We have
\[
\frac{d}{dt}{R_n^t}_{|t=0}=-\alpha_n<0,
\]
where
\begin{equation}
\label{alphan}
\alpha_n=c_n\frac{(7n^2+4n-12)(2n+1)}{12(n+2)(n-1)(3n+1)}\cdot\frac{ \left( \frac{(n-1)!!}{(n/2)!} \right)^3 }{ \frac{(3n-1)!!}{(3n/2)!} }.
\end{equation}
\end{lemma}
\begin{proof}
Taking into account~\cite[Lemmas 14, 15]{Fr19}, we can argue as in~\cite[Proposition 10]{Fr19} to see that
\[
\frac12 D^2_{uu}\Phi(R,0)[v,v]=R^2 Q_R(v),
\]
where
\begin{align*}
Q_R(v)[z]=&\frac{2}{R^3}\left(v(z)\Delta_{\Sd}v(z)+(v(z))^2\right)+\frac{\pi}{3}(v(z))^2\\
&-\frac12v(z)\int_{\Sd}\frac{v(z')}{|z-z'|}\,d\sigma(z')+\frac34\int_{\Sd}\frac{(v(z'))^2}{|z-z'|}\,d\sigma(z').
\end{align*}
By~\cite[Formula 8.915.5 page 986]{GRbook}, the square of a Legendre polynomial can be written as
\[
(P_n)^2 = \sum_{k=0}^{n} \beta_k  P_{2n-2k},
\]
where
\[
\beta_k = \frac{(b_{n-k})^2 \, b_k \, }{b_{2n - k}} \cdot \frac{4n - 4k + 1}{4n - 2k + 1},\quad\text{with}\quad b_\ell = \frac{(2\ell-1)!!}{\ell!}.
\]
Then
\[
(\mathfrak I_n)^2=(c_n P_n)^2=\sum_{k=0}^{n} c_n^2\beta_k  P_{2n-2k}=\sum_{k=0}^{n}a_{2k}\mathfrak I_{2k},
\]
with $a_{2k}=c_n^2\beta_{n-k}/c_{2k}$.
On the other hand, we know that
\[
\Delta_{\Sd}\mathfrak I_\ell=-\ell(\ell+1)\mathfrak I_\ell,\quad\text{and}\quad\int_{\Sd}\frac{\mathfrak I_\ell(z')}{|z-z'|}\,d\sigma(z')=\frac{4\pi}{2\ell+1}\mathfrak I_\ell,
\]
thus
\begin{equation}
\label{Qr}
\begin{split}
Q_{\rn}(\mathfrak I_n)&=\left[\frac{2}{\rn^3}(1-n(n+1))+\frac\pi3-\frac{2\pi}{2n+1}\right](\mathfrak I_n)^2+3\pi\sum_{k=0}^n\frac{a_{2k}}{4k+1}\mathfrak I_{2k}\\
&=\sum_{k=0}^n\left[-\frac{16\pi(n^2+n-1)}{3(2n+1)(n+2)}+\frac\pi3-\frac{2\pi}{2n+1}+\frac{3\pi}{4k+1}\right]a_{2k}\mathfrak I_{2k}.
\end{split}
\end{equation}
Taking into account~\eqref{der-mixt1},~\eqref{Qr}, and the orthogonality properties of the spherical harmonics, we can apply~\cite[formula (I.6.3)]{KieBook}, to conclude
\begin{align*}
\frac{d}{dt}{R_n^t}_{|t=0}&=-\frac12\frac{\int_{\Sd} D^2_{uu}\Phi(\rn,0)[\mathfrak I_n,\mathfrak I_n]\cdot\mathfrak I_n\,d\sigma}{\int_{\Sd} D^2_{Ru}\Phi(\rn,0)[\mathfrak I_n]\cdot\mathfrak I_n\,d\sigma}\\
&=-\frac{7n^2+4n-12}{12(n+2)(n-1)}a_n,
\end{align*}
where an easy computation show that
\[
a_n=c_n\frac{2n + 1}{3n + 1}\cdot\frac{ \left( \frac{(n-1)!!}{(n/2)!} \right)^3 }{ \frac{(3n-1)!!}{(3n/2)!} }.\qedhere
\]
\end{proof}

\subsection{Proof of Theorem\texorpdfstring{~\ref{thm_N=3}}{ 1}: the case \texorpdfstring{$n\ge3$}{n>=3} odd}

In this case we can prove the existence of many branches of solutions bifurcating from $B_{\rn}$. We again consider the spherical coordinates introduced in~\eqref{sphe-coord}.

For all $m=1,\dots,n$, we set
\[
\mathcal X_m=\{u\in\mathcal C^{2,\alpha}(\Sd):\,u(\phi,\theta)=u(\phi,\pi-\theta)=u(\pi/m-\phi,\theta)\},
\]
and
\[
\mathcal Y_m=\{u\in\mathcal C^{0,\alpha}(\Sd):\,u(\phi,\theta)=u(\phi,\pi-\theta)=u(\pi/m-\phi,\theta)\}.
\]

In the following lemma we compute the kernel and the image of our operator restricted to $\mathcal X_m$.

\begin{lemma}
\label{lemma:ker3d}
Let $m\in\{(n+1)/2,\dots,n\}$ be odd. Considering the restriction of $\Phi$ to $\mathcal A\cap\mathcal X_m$, one has
\begin{gather*}
\mathrm{ker}( L_{\rn})=\mathrm{span}\{\mathfrak L^1_{n,m}\}\footnotemark,\\
\mathrm{ran}(L_{\rn})=\left\{u\in\mathcal Y_m:\int_{\Sd}\mathfrak L^1_{n,m}\cdot\,u\,d\sigma=0\right\},
\end{gather*}
\footnotetext{ see Lemma~\ref{lemma:sphar} for the precise definition of $\mathfrak L^i_{n,m}$.}
and
\[
\frac{d}{dR}_{|R=\rn}L_R\mathfrak L^1_{n,m}=-3(n^2+n-2)\rn^{-1}\mathfrak L^1_{n,m}\not\in\mathrm{ran}(L_{\rn}).
\]
\end{lemma}
\begin{proof}
Taking into account Lemma~\ref{lemma:sphar}, it is enough to argue exactly as in~\cite[Proposition 9]{Fr19}, once we observe that the unique spherical harmonic of degree $1$ or $n$ in $\mathcal X_m$ is $\mathfrak L^1_{n,m}$.
Indeed, the only spherical harmonics of degree $1$ or $n$ that satisfy $u(\phi,\theta)=u(\pi/m-\phi,\theta)$ are
\[
\mathfrak I_1,\mathfrak I_n,\mathfrak L^1_{n,m},
\]
but among these, $\mathfrak L^1_{n,m}$ is the only one that satisfies $u(\phi,\theta)=u(\phi,\pi-\theta)$ since $n$ and $m$ are odd and $P_n(-t)=(-1)^nP_n(t)$ and $Q_{n-m}(-t)=(-1)^{n-m}Q_{n-m}(t)$.
Finally, we point out that, given $u \in \mathcal X_m$, then $\Phi(R,u) \in \mathcal Y_m$. This is obvious for the local terms in $\Phi$, while for the nonlocal one it easily follows by a suitable change of variables (with Jacobian equal to $1$).
\end{proof}

As before, we can now apply the Crandall-Rabinowitz Theorem to show that bifurcation occurs in part~(2) of the theorem.

\begin{remark}
\label{rmk-bif3}
If instead of considering the spaces $\mathcal X_m$ and $\mathcal Y_m$, we restrict our operator from
\[
\mathbf X_m=\{u\in\mathcal C^{2,\alpha}(\Sd):\,u(\phi,\theta)=u(\phi,\pi-\theta)=u(2\pi/m-\phi,\theta)\},
\]
to
\[
\mathbf Y_m=\{u\in\mathcal C^{0,\alpha}(\Sd):\,u(\phi,\theta)=u(\phi,\pi-\theta)=u(2\pi/m-\phi,\theta)\},
\]
we can argue exactly as in the previous lemma to prove that
\begin{gather*}
\mathrm{ker}( L_{\rn})=\mathrm{span}\{\mathfrak L^2_{n,m}\},\\
\mathrm{ran}(L_{\rn})=\left\{u\in\mathbf Y_m:\int_{\Sd}\mathfrak L^2_{n,m}\cdot\,u\,d\sigma=0\right\},\\
\frac{d}{dR}_{|R=\rn}L_R\mathfrak L^2_{n,m}=-3(n^2+n-2)\rn^{-1}\mathfrak L^2_{n,m}\not\in\mathrm{ran}(L_{\rn}).
\end{gather*}
By the Crandall-Rabinowitz Theorem we can deduce the existence of another branch of solutions whose first order expansion is given by $v_{n,m}^t=t \mathfrak L^2_{n,m} +O(t^2)$ in $\mathcal C^{2,\alpha}(\Sd)$. Nevertheless, if $u\in\mathcal X_m$, then $v (\phi, \theta) = u (\phi+\pi/m, \theta)\in\mathbf X_m$ and this means that the two solutions $v_{n,m}^t$ and $u_{n,m}^t$
 are geometrically the same solution of Problem~\eqref{eq:equation}.
\end{remark}

Hence, to conclude the proof of part~(2) of the theorem, it remains only to show that $\frac{d}{dt}{R_{n,m}^t}_{|t=0}=0$. This is done in the following lemma.

\begin{lemma}
We have
\[
\frac{d}{dt}{R_{n,m}^t}_{|t=0}=0.
\]
\end{lemma}
\begin{proof}
Observe that
\[
\big[\sin(m\phi)\sin^m(\theta)Q_{n-m}(\cos(\theta))\big]^2=\frac{1-\cos(2m\phi)}{2}\sin^{2m}(\theta)\big[Q_{n-m}(\cos(\theta))\big]^2,
\]
thus if we write
\[
(\mathfrak L_{n,m}^1)^2=\sum_{k\in\N}a_{k}\mathfrak I_k+\sum_{k,\ell\in\N}\sum_{i=1}^2a^i_{k,\ell}\mathfrak L_{k,\ell}^i,
\]
taking into account Lemma~\ref{lemma:sphar}, we have $a^1_{n,m}=0$, that is  there is no term proportional to $\mathfrak L_{n,m}^1$ in its square.

Arguing as in the proof of Lemma~\ref{lemma:derRaggio}, we have that
\[
Q_{\rn}(\mathfrak L_{n,m}^1)-\frac34\int_{\Sd}\frac{(\mathfrak L_{n,m}^1)^2}{|z-z'|}\,d\sigma(z')
\]
is proportional to $(\mathfrak L_{n,m}^1)^2$, which, as just shown, contains no term proportional to $\mathfrak L_{n,m}^1$. Taking into account that $\int_{\Sd}\frac{\mathfrak I_\ell(z')}{|z-z'|}\,d\sigma(z')=\frac{4\pi}{2\ell+1}\mathfrak I_\ell$ then the same happen also with the term
\[
\frac34\int_{\Sd}\frac{(\mathfrak L_{n,m}^1)^2}{|z-z'|}\,d\sigma(z').
\]
Hence, when we integrate $Q_{\rn}(\mathfrak L_{n,m}^1)$ against $\mathfrak L_{n,m}^1$ we get
\[
\int_{\Sd} D^2_{uu}\Phi(\rn,0)[\mathfrak L_{n,m}^1,\mathfrak L_{n,m}^1]\cdot\mathfrak L_{n,m}^1\,d\sigma=0,
\]
and by~\cite[formula (I.6.3)]{KieBook}, we conclude $\frac{d}{dt}{R_{n,m}^t}_{|t=0}=0$.
\end{proof}

\subsection{End of the Proof of Theorem\texorpdfstring{~\ref{thm_N=3}}{ 1}}
To conclude the proof of the Theorem~\ref{thm_N=3}, it is enough to show that there are no solutions close to $B_R$ if $R\not=\rn$. This is proved in the next proposition.

\begin{prop}
\label{prop:nobif3D}
For all $R\not=\rn$, $n\ge2$, there are no critical points of $E$ arbitrarily close to $B_R$ in $\mathcal C^{2,\alpha}(\Sd)$.
\end{prop}
\begin{proof}
Fix $\bar R\not=\rn$, then
\[
\mathrm{ker}(L_{\bar R})=\{z_1,z_2,z_3\},
\]
where $z=(z_1,z_2,z_3)\in\Sd$ is given by $z_i=x_i/|x|$, for $x=(x_1,x_2,x_3)\in\R^3$.
Denote by
\begin{gather*}
\mathcal X=\{u\in H^2(\Sd):\,u(z_1,z_2,z_3)=u(z_1,-z_2,z_3)=u(z_1,z_2,-z_3)\},\\
\mathcal Y=\{u\in L^2(\Sd):\,u(z_1,z_2,z_3)=u(z_1,-z_2,z_3)=u(z_1,z_2,-z_3)\},
\end{gather*}
and let $\Psi$ be our operator acting on these spaces, i.e. $\Psi:(0,+\infty)\times\mathcal X\to\mathcal Y$,
\[
\Psi(R,u)=R^2(F(R+u)-F(R)).
\]
Note that if we prove that all the solutions of $\Psi$ close to $(R,0)$ are trivial, then in particular the same happens in $\mathcal C^{2,\alpha}(\Su)$ and the claim is proved.

As in Lemma~\ref{lemma:ker3d}, we have
\[
\mathrm{ker}(D_u\Psi(\bar R,0))=\langle{z_1}\rangle_{\mathcal X},
\]
and
\[
\mathrm{ran}(D_u\Psi(\bar R,0))=(\langle{z_1}\rangle_{\mathcal Y})^\perp,
\]
where we denote by $\langle{\,\cdot\,}\rangle_{\mathcal X}$, or $\langle{\,\cdot\,}\rangle_{\mathcal Y}$, the span in $\mathcal X$, or $\mathcal Y$ respectively.

We can now split
\begin{gather*}
\mathcal X= \langle{z_1}\rangle_{\mathcal X}\oplus(\langle{z_1}\rangle_{\mathcal X})^\perp,\\
\mathcal Y= (\langle{z_1}\rangle_{\mathcal Y})^\perp\oplus\langle{z_1}\rangle_{\mathcal Y},
\end{gather*}
and we introduce the projections
\begin{gather*}
P:\mathcal X\to  \langle{z_1}\rangle_{\mathcal X},\\
Q:\mathcal Y\to \langle{z_1}\rangle_{\mathcal Y}.
\end{gather*}
Hence, the problem $\Psi(R,u)=0$ is equivalent to the system
\[
\begin{cases}
Q \Psi(R,Pu+(1-P)u)=0 \\
(1-Q)\Psi(R,Pu+(1-P)u)=0.
\end{cases}
\]

If we write $v=Pu\in \langle z_1\rangle_{\mathcal{X}}$ and $w=(1-P)u\in(\langle z_1\rangle_{\mathcal{X}})^\perp$, by the Implicit Function Theorem, there exists a (locally) unique function $\xi:B_\delta(\bar R,0)\subset(0,\infty)\times\langle z_1\rangle_{\mathcal{X}}\to(\langle z_1\rangle_{\mathcal{X}})^\perp$ such that
\[
(1-Q)\Psi(R,v+w)=0,
\]
if and only if
\[
w=\xi(R,v).
\]
Thus the problem $\Psi(R,u)=0$ reduces to 
\begin{equation}
\label{eqridotta1}
Q\Psi(R,v+\xi(R,v))=0,
\end{equation}
where $(R,v)\in B_\delta(\bar R,0)\subset(0,\infty)\times\langle z_1\rangle_{\mathcal{X}}$.

Let us identify $\mathbb{R}$ with $\langle z_1\rangle_{\mathcal X}$ via $\tau\mapsto\tau z_1$.
We now define $\bar\xi:(-R,R)\times\R(\simeq (-R,R)\times \langle z_1 \rangle_{\mathcal X})\to H^2(\Sd),
$ by
\[
\bar\xi(R,\tau)=\text{sgn}(\tau)\sqrt{R^{2}-\tau^{2}(z_2^2+z_3^2)}-R.
\]
It is now easy to see that $\bar\xi(R,\tau)\in(\langle z_1 \rangle_{\mathcal X})^{\perp}$, indeed $\bar\xi$ - and in turn its derivatives - are even in $z_1$ and thus integrating against $z_1$ gives the orthogonality.
Moreover, it is straightforward to see that $\Omega_{R+\tau z_1+\psi(R,\tau)}$ is a ball of radius $R$, translated of $\tau$ among the $x_1$-direction and then
\[
\Psi(R,\tau z_1+\bar\xi(R,\tau))=0.
\]

By uniqueness $\xi\equiv\bar\xi$ and then, taking into account~\eqref{eqridotta1} all the solutions close to $(\bar R,0)$ are given by $(R,\tau z_1+\psi(R,\tau))$. However, these correspond to balls, and are therefore trivial for the problem.\\

A similar argument shows that there are no nontrivial solutions in
\[
\{u\in H^2(\Sd):\,u(z_1,z_2,z_3)=u(-z_1,z_2,z_3)=u(z_1,z_2,-z_3)\},
\]
and in
\[
\{u\in H^2(\Sd):\,u(z_1,z_2,z_3)=u(-z_1,z_2,z_3)=u(z_1,-z_2,z_3)\},
\]
Finally, if we restrict ourselves to the orthogonal complement of the kernel, then by the Implicit Function Theorem the only solutions are given by $(R,0)$.
\end{proof}

\section{The case of dimension \texorpdfstring{$N=2$}{N=2} with logarithmic Coulomb interaction}
\label{sec_N=2}

In this section we focus on the case $N=2$ and when the Coulomb interaction is given by
\[
D (\Omega)={1\over 2\pi} \int_\Omega  \int_\Omega \log {1\over |x-y|} \, dx dy.
\]

The scheme of the proof is similar to that of Section~\ref{sec_N=3}. 
We write
\[
\Omega_\varphi = \left\{ x \in \R^2 : |x| < \varphi\!\left(\frac{x}{|x|}\right) \right\},
\]
for some function $\varphi : \Su \to \R$. We then derive the equation that $\varphi$ must satisfy in order for $\Omega_\varphi$ to be a critical point of the energy functional under the volume constraint. 
This is the statement of the next lemma.

\begin{lemma}\label{lemma1n}
Let $\varphi\in\mathcal C^{2,\alpha}(\Su)$, $\varphi\ge0$. Then
\[
{d \over dt}_{| t=0} \, E \left( |\Omega_\varphi |^{1\over 2} \, {\Omega_{\varphi +t u} \over |\Omega_{\varphi+tu}|^{1\over 2}} \right) =0,
\]
for all $u\in \mathcal C^{2,\alpha}(\Su)$, if and only if
\[
\begin{split}
-\mathrm{div}\left(\frac{\nabla\varphi}{\varphi\sqrt{ \varphi^2 + |\nabla \varphi|^2}}\right)+\frac{2}{\sqrt{ \varphi^2 + |\nabla \varphi|^2}}&-\frac{\sqrt{ \varphi^2 + |\nabla \varphi|^2}}{\varphi^2}+ \\
&+{1\over 2\pi} \int_{\Omega_\varphi}    \log {1 \over |x - \varphi (z)   z| }\, dz=\mu ,
\end{split}
\]
where
\[
\mu= {1 \over 2} \, {{\mathrm{Per}} (\Omega_{\varphi }) \over  |\Omega_\varphi |} + {2\over \pi}  \,\frac{D(\Omega)}{|\Omega_\varphi |} -{|\Omega_\varphi |\over 4\pi} .
\]
\end{lemma}

\begin{proof}
Using that for any $c>0$, ${\mbox {Per}} (c \Omega) = c
{\mbox {Per}} ( \Omega)$ and
\begin{align*}
 D(c\Omega)&=\frac{1}{2\pi}\int_{c\Omega}  \int_{c\Omega} \log {1\over |x-y|} \, dx dy \\
 &= \frac{c^4}{2\pi} \log {1\over c}  \, |\Omega|^2 + \frac{c^4}{2\pi} \int_{\Omega}  \int_{\Omega} \log {1\over |x-y|} \, dx dy= \frac{c^4}{2\pi} \log {1\over c} |\Omega|^2+c^4D(\Omega),
\end{align*}
we get
\begin{align*}
E \left( |\Omega_\varphi |^{1\over 2}  {\Omega_{\varphi +t u} \over |\Omega_{\varphi+tu}|^{1\over 2}} \right) &=
          { |\Omega_\varphi |^{1 \over 2} \over |\Omega_{\varphi+tu}|^{1\over 2}} {\mbox {Per}} (\Omega_{\varphi + t u})\\
          &\quad+ { |\Omega_\varphi |^{2} \over 2\pi}\log{|\Omega_{\varphi+tu}|\over |\Omega_{\varphi}|}+{ |\Omega_\varphi |^{2} \over |\Omega_{\varphi+tu}|^{2}}D(\Omega_{\varphi + t u}),
  \end{align*}
  for any $u\in\mathcal C^{2,\alpha}(\Su)$.
    Observe that
    \begin{align*}
        |\Omega_{\varphi} | &= \int_{\mathbb S^{1}} \int_0^{\varphi (z)} r\, dr d\sigma(z) = {1\over 2} \int_{\mathbb S^{1} }
        \varphi^2 (z)\, d\sigma(z).
    \end{align*}
    Thus
    \be
    \label{mis:2D}
  \begin{split}
        |\Omega_{\varphi +tu}|&=   {1\over 2} \int_{\mathbb S^{1} } \varphi^2 \,  d \sigma + t \int_{\mathbb S^{N-1}} \varphi  u  \,  d \sigma  +{t^2\over 2} \int_{\mathbb S^{1} } u^2 \,  d \sigma\\
        &=   |\Omega_{\varphi } | + t \int_{\mathbb S^{1}} \varphi u \,  d \sigma +{t^2\over 2} \int_{\mathbb S^{1} } u^2 \,  d \sigma.
    \end{split} 
    \ee
Taking into account that
\begin{align*}
  \frac{1}{2\pi} \int_{\Omega_{\varphi +tu}}& \int_{\Omega_{\varphi +t u}} \log{1\over |x-y|} \, dx \, dy\\
        &= \frac{1}{2\pi}\int_{\mathbb S^{1}} \int_{\mathbb S^{1}} \int_0^{\varphi + t u}   \int_0^{\varphi + t u}\log{1\over |\rho z - \rho' z'|} \rho  d \rho \rho' d \rho' d \sigma(z) d \sigma(z')\\
       &= \frac{1}{2\pi} \int_{\mathbb S^{1}} \int_{\mathbb S^{1}} \int_0^{\varphi }  \int_0^{\varphi}  \log{1\over |\rho z - \rho' z'|} \rho  d \rho \rho' d \rho' d \sigma(z) d \sigma(z')\\
    &\quad+\frac{1}{\pi}  \int_{\mathbb S^{1}} \int_{\mathbb S^{1}} \int_\varphi^{\varphi +t u }  \int_0^{\varphi} \log{1\over |\rho z - \rho' z'|} \rho  d \rho \rho' d \rho' d \sigma(z) d \sigma(z')\\
     &= \frac{1}{2\pi} \int_{\Omega_{\varphi }} \int_{\Omega_{\varphi }}\log {1\over |x-y| } \, dx \, dy\\
    &\quad+ \frac{t}{\pi} \int_{\mathbb S^{1}} \int_{\mathbb S^{1}}   \int_0^{\varphi (z)}  u(z')\log{1\over |\rho z - \varphi (z')   z'|}\rho d \rho \varphi (z')  d \sigma(z) d \sigma(z')  + O(t^2)
    \\
     &=D(\Omega_\varphi)+ {t\over\pi} \int_{\mathbb S^{1}}  u(z') \, \varphi (z') \,  \int_{\Omega_\varphi}   \log {  1\over |x - \varphi (z')   z'|^\lambda}\,dx d \sigma (z') + O(t^2),
\end{align*}
we obtain
\be \label{log1}
\begin{aligned}
  {|\Omega_{\varphi }|^2 \over |\Omega_{\varphi + tu} |^2 } D(\Omega_{\varphi + tu})\!&=D(\Omega_\varphi)-{ t\over\pi}  {D(\Omega_\varphi) \over |\Omega_\varphi |} \int_{\mathbb S^{1}} \varphi (z) u (z) d \sigma (z) \\
&\quad+\! {t\over2\pi}\! \int_{\mathbb {S}^1}\! u(z) \varphi (z)\! \int_{\Omega_\varphi}\! \log {1\over |x-\varphi (z) z|} \, dx  d \sigma (z)\! + O(t^2).
\end{aligned}
 \ee
 Moreover
\be
\label{logggg}
     {1\over 4\pi} |\Omega_\varphi |^2 \log \left({|\Omega_{\varphi + tu}| \over |\Omega_\varphi |} \right) = {|\Omega_\varphi | \over 4\pi} \, t \, \int_{\mathbb{S}^1} \varphi (z) u (z) d \sigma (z) + O(t^2).
 \ee
By Lemma~\ref{lemma:per}, we can compute
\begin{align*}
    {\mbox {Per}} (\Omega_{\varphi +tu})&= \int_{\mathbb S^{1} }\sqrt{ (\varphi + t u )^2 + |\nabla (\varphi + t u)|^2}  \, d \sigma\\
    &=\int_{\mathbb S^{1}}  \sqrt{ \varphi^2 + |\nabla \varphi|^2} d \sigma+ t  \int_{\mathbb S^{1}} { u \varphi + \nabla u \nabla \varphi\over  \sqrt{ \varphi^2 + |\nabla \varphi|^2}} \,d \sigma + O(t^2)\\
    &={\mbox {Per}} (\Omega_{\varphi})+ t  \int_{\mathbb S^{1}} { u \varphi + \nabla u \nabla \varphi\over  \sqrt{ \varphi^2 + |\nabla \varphi|^2}} \,d \sigma + O(t^2).
\end{align*}
Hence
\be\label{uno2D}
\begin{aligned}
     &{ |\Omega_\varphi |^{1 \over 2} \over |\Omega_{\varphi+tu}|^{1\over 2}}  {\mbox {Per}} (\Omega_{\varphi + t u}) \\
     &={\mbox {Per}} (\Omega_{\varphi })-  {t \over 2} \, {{\mbox {Per}} (\Omega_{\varphi }) \over  |\Omega_\varphi |} \int_{\mathbb S^{1}} \varphi u \,d\sigma + t  \int_{\mathbb S^{1}} { u \varphi + \nabla u \nabla \varphi\over  \sqrt{ \varphi^2 + |\nabla \varphi|^2}} \,d \sigma + O(t^2).
\end{aligned}
\ee
From~\eqref{log1},~\eqref{logggg} and~\eqref{uno2D}, we get 
\begin{align*}
 & E \left( |\Omega_\varphi |^{1\over 2} \, {\Omega_{\varphi +t u} \over |\Omega_{\varphi+tu}|^{1\over 2}} \right) \\
  &=E(\Omega_\varphi)+t\left[
 -\left( {1 \over 2} \, {{\mbox {Per}} (\Omega_{\varphi }) \over  |\Omega_\varphi |} +{1\over \pi}    {D(\Omega_\varphi) \over |\Omega_\varphi |} - {|\Omega_\varphi | \over 4\pi} \right)\int_{\mathbb S^{1}} \varphi u\,d\sigma +\right.\\
&\quad\left.
  + \int_{\mathbb S^{1}} \,  { u \varphi + \nabla u \nabla \varphi \over  \sqrt{ \varphi^2 + |\nabla \varphi|^2}}  d \sigma+  {1\over 2\pi} \int_{\mathbb {S}^1} u(z) \varphi (z) \int_{\Omega_\varphi} \log {1\over |x-\varphi (z) z|} \, dx  d \sigma (z) 
  \right]\\
  &\quad+O(t^2).
\end{align*}
Hence $\varphi$ is a solution if and only if for all $u\in\mathcal C^{2,\alpha}(\Su)$
\begin{align*}
 0= & -\left( {1 \over 2} \, {{\mbox {Per}} (\Omega_{\varphi }) \over  |\Omega_\varphi |} +{1\over \pi}    {D(\Omega_\varphi) \over |\Omega_\varphi |} - {|\Omega_\varphi | \over 4\pi} \right)\int_{\mathbb S^{1}} \varphi u\,d\sigma +\\
&+ \int_{\mathbb S^{1}} \,  { u \varphi + \nabla u \nabla \varphi \over  \sqrt{ \varphi^2 + |\nabla \varphi|^2}}  d \sigma+  {1\over 2\pi} \int_{\mathbb {S}^1} u(z) \varphi (z) \int_{\Omega_\varphi} \log {1\over |x-\varphi (z) z|} \, dx  d \sigma (z) ,
\end{align*}
that is $\varphi$ is a solution of the following equation on $\mathbb S^{1}$
\[
-\mathrm{div}\left(\frac{\nabla\varphi}{\sqrt{ \varphi^2 + |\nabla \varphi|^2}}\right)+\frac{\varphi}{\sqrt{ \varphi^2 + |\nabla \varphi|^2}}+{1\over2\pi}\varphi \,  \int_{\Omega_\varphi}   \log {  1\over |x - \varphi (z)   z|}\,dx=\mu \varphi,
\]
where $\mu$ is given by
\[
\mu= {1 \over 2} \, {{\mathrm{Per}} (\Omega_{\varphi }) \over  |\Omega_\varphi |} + {2\over \pi}  \,\frac{D(\Omega)}{|\Omega_\varphi |} -{|\Omega_\varphi |\over 4\pi} .
\]
Clearly, it can be rewritten as
\[
\begin{split}
-\mathrm{div}\left(\frac{\nabla\varphi}{\varphi\sqrt{ \varphi^2 + |\nabla \varphi|^2}}\right)+\frac{2}{\sqrt{ \varphi^2 + |\nabla \varphi|^2}}&-\frac{\sqrt{ \varphi^2 + |\nabla \varphi|^2}}{\varphi^2}+ \\
&+{1\over 2\pi} \int_{\Omega_\varphi}    \log {1 \over |x - \varphi (z)   z| }\, dz=\mu ,
\end{split}
\]
 on $\mathbb S^{1}$.
\end{proof}

\medskip
We now reformulate the bifurcation problem as
\[
\Phi(R,u)=0, \qquad \text{where } \varphi = R + u .
\]
Let 
\[
F(\varphi)= F_P (\varphi ) + F_L  (\varphi), \quad \varphi\in\mathcal C^{2,\alpha}(\Su),
\]
with
\[
 F_P (\varphi )=    -\mathrm{div}\left(\frac{\nabla\varphi}{\varphi\sqrt{ \varphi^2 + |\nabla \varphi|^2}}\right)+\frac{2}{\sqrt{ \varphi^2 + |\nabla \varphi|^2}}-\frac{\sqrt{ \varphi^2 + |\nabla \varphi|^2}}{\varphi^2} ,
 \]
 and
 \[
 F_L  (\varphi)=  {1\over 2\pi} \int_{\Omega_\varphi}    \log {1 \over |x - \varphi (z)   z| }\, dx.
\]
We set $\Phi:\mathcal A\to\mathcal C^{0,\alpha}(\Su)$ as
\be
\label{defPhi2}
\Phi (R, u) = R^2 \left( F(R+u) - F(R) \right).
\ee
where
\[
\mathcal A=\{(R,u)\in(0,+\infty)\times\mathcal C^{2,\alpha}(\Su):u>-R\}.
\]
We have
\begin{lemma}
\label{lemma:der2D}
For all $R>0$ and $u\in\mathcal C^{2,\alpha}(\mathbb S^{1})$ one has
\begin{align*}
 \lim_{t\to0}&\frac{  F(R+tu)-F(R)}{t}=\frac{1}{R^2}\left(-\Delta_{\Su} u-u\right) \\&+ R \left[ -{1\over 2} v + {1\over 2\pi} \log {1\over R} \, \int_{\mathbb{S}^1} v + {1\over 2\pi}  \int_{\mathbb{S}^1} v(z') \log {1\over |z-z'|} d\sigma (z')\right].
 \end{align*}
\end{lemma}

This is an immediate consequence the expansions of $F_P$ and $F_L$ that are computed in the following two lemmas.

\begin{lemma}
\label{FP2ordine}
Given $R>0$ and $u\in\mathcal C^{2,\alpha}(\Su)$, one has the following expansion of $F_P$, as $t\to0$
  \begin{align*}
    F_P(R+tu)&=\frac{1}{R}+\frac{t}{R^2}\left(-\Delta_{\Su} u-u\right)\\
    &\quad+\frac{2t^2}{R^3}\left(u\Delta_{\Su} u+\frac{1}{4}|\nabla u|^2+\frac{1}{2}u^2\right)+O(t^3).
    \end{align*}
\end{lemma}
\begin{proof}
Arguing as in~\cite[Lemma 14]{Fr19}, we have
\begin{align*}
    F_P(R+tu)&=\frac{1}{R}+\frac{t}{R^2}\left(-\Delta_{\Su} u-u\right)\\
    &+\frac{2t^2}{R^3}\left(\mathrm{div}(u\nabla u)-\frac{3}{4}|\nabla u|^2+\frac{1}{2}u^2\right)+O(t^3)
\end{align*}
    the claim follows from the identity $\mathrm{div}(u\nabla u)=|\nabla u|^2+u\Delta_{\Su} u$.
\end{proof}

\begin{lemma}
\label{FL2ordine}
Given $R>0$ and $\varphi\in\mathcal C^{2,\alpha}(\mathbb S^{1})$, one has the following expansion of $F_L $, as $t\to0$
\begin{align*}
    &F_L (R+tu)= {1 \over 2}  R^2 \log {1\over R}\\
    &+ t  \, R \, \left[ -{1\over 2} u + {1\over 2\pi} \log {1\over R} \, \int_{\mathbb{S}^1} u + {1\over 2\pi}  \int_{\mathbb{S}^1} u(z') \log {1\over |z-z'|} d\sigma (z')\right]\\
    &+t^2\left[-\frac{1}{4}u^2-\frac{1}{2\pi} u \int_{\mathbb{S}^1} u\,d\sigma+\frac{1}{4\pi} \log {1\over R}\int_{\mathbb{S}^1} u^2\,d\sigma+\frac{1}{4\pi}\int_{\mathbb{S}^1} (u(z'))^2 \log {1\over |z-z'|} d\sigma (z')\right.\\
&\left.\qquad\quad+\frac{1}{4\pi}u(z)\int_{\mathbb{S}^1} u(z')|z-z'| d\sigma (z')-\frac{1}{8\pi}\int_{\mathbb{S}^1} (u(z'))^2|z-z'| d\sigma (z')\right]+ O(t^3).
\end{align*}
\end{lemma}

\begin{proof}
Arguing as in~\cite[Lemma 15]{Fr19}, observe that
\begin{align*}
    F_L  (\varphi) &= {1\over 2\pi} \int_{\Omega_\varphi} \log {1\over |\varphi (z) |} dx + {1\over 2\pi} \int_{\Omega_\varphi} \log {1\over |{x\over \varphi (z)} - z |} dx\\
    &= {|\Omega_\varphi | \over 2\pi} \log {1\over |\varphi (z) |} + {\varphi^2 (z) \over 2\pi} \int_{\mathbb {S}^1} \int_0^{\varphi (z') \over \varphi (z)} \log {1 \over |r z' - z |} r d r d\sigma (z')\\
    &= {|\Omega_\varphi | \over 2\pi} \log {1\over |\varphi (z) |} + {\varphi^2 (z) \over 2\pi} \int_{\mathbb {S}^1}\int_0^{1+tg(t,z,z')}f(r,z,z')\,drd\sigma(z'),
\end{align*}
  where for $\varphi=R+tu$
    \[
    g(t,z,z')=\frac{u(z')-u(z)}{R+tu(z)}=\frac{u(z')-u(z)}{R}-u(z)\frac{u(z')-u(z)}{R^2}t+O(t^2),
    \]
    and
    \[
    f(r,z,z')=r\,\log\frac{1}{|z-rz'|}.
    \]
Note that
    \begin{align*}
\int_0^{1+tg(t,z,z')}&f(r,z,z')\,dr=\int_0^1f(r,z,z')\,dr+tf(1,z,z')g(t,z,z')\\
& +t^2g^2(t,z,z')\int_0^1\varsigma\int_0^1\partial_rf(1+t g(t,z,z')\varsigma\tau,z,z'))\,d\varsigma d\tau+ O(t^3).
    \end{align*}
Let us compute
\begin{align*}
    \partial_rf(r,z,z')&=\log\frac{1}{|z-rz'|}+ r\frac{z\cdot z'-r|z'|^2}{|z-rz'|}\\
&=\log\frac{1}{|z-rz'|}+\frac{1}{2}\frac{(1-r^2)}{|z-rz'|}-\frac{1}{2}|z-rz'|,
\end{align*}
where we used that $2r(z\cdot z'-r|z'|^2)=|z|^2-r^2|z'|^2-|z-rz'|^2=1-r^2-|z-rz'|^2$.\\
Thus, arguing as in~\cite[Lemma 15]{Fr19} and applying the dominated convergence theorem, we obtain
\begin{align*}
&\int_0^{1+tg(t,z,z')}f(r,z,z')\,dr=\int_0^1\log\frac{1}{|z-rz'|}\,dr+t\frac{u(z')-u(z)}{R}\log\frac{1}{|z-z'|}\\
&+t^2\left(-u(z)\frac{u(z')-u(z)}{R^2}\log\frac{1}{|z-z'|}+\frac{(u(z)-u(z'))^2}{2R^2}\left(\log\frac{1}{|z-z'|}-\frac12|z-z'|\right)\right)\\
&+O(t^3).
\end{align*}
We claim 
\be \label{integrals}
\int_{B_1} \log {1\over |x-z|} dx =0, \quad \int_{\mathbb{S}^1} \log {1\over |z-z'|} d\sigma (z') =0.
\ee
The proof of claim~\eqref{integrals}, is given in Appendix~\ref{app:int}. Then taking into account that
\[
\int_{\mathbb{S}^1} |z-z'| d\sigma (z')=4\pi,
\]
we have
\be
\begin{split}
&\int_{\mathbb {S}^1}\int_0^{1+tg(t,z,z')}f(r,z,z')\,drd\sigma(z')=\frac tR \int_{\mathbb{S}^1} u(z') \log {1\over |z-z'|} d\sigma (z')\\
&+\frac{t^2}{R^2}\left(-2u(z)\int_{\mathbb{S}^1} u(z') \log {1\over |z-z'|} d\sigma (z')+\frac12\int_{\mathbb{S}^1} (u(z'))^2 \log {1\over |z-z'|} d\sigma (z')\right.\\
&\left.\qquad\quad-\pi u^2+\frac12u(z)\int_{\mathbb{S}^1} u(z')|z-z'| d\sigma (z')-\frac14\int_{\mathbb{S}^1} (u(z'))^2|z-z'| d\sigma (z')\right)\\
&+O(t^3).
\end{split}
\ee
and then
\begin{align*}
  &  {(R+tu)^2 \over 2\pi} \int_{\mathbb {S}^1} \int_0^{1+ t g(t, R,z',z)} \log {1 \over |r z' - z |} r d r d\sigma (z')=\frac{ Rt}{2\pi}\int_{\mathbb{S}^1} u(z') \log {1\over |z-z'|} d\sigma (z')\\
&+\frac{t^2}{2\pi}\left(\frac12\int_{\mathbb{S}^1} (u(z'))^2 \log {1\over |z-z'|} d\sigma (z')-\pi u^2+\frac12u(z)\int_{\mathbb{S}^1} u(z')|z-z'| d\sigma (z')\right.\\
&\left.\quad\qquad-\frac14\int_{\mathbb{S}^1} (u(z'))^2|z-z'| d\sigma (z')\right)+O(t^3).
\end{align*}
Taking into account~\eqref{mis:2D}, we have
\begin{align*}
    {|\Omega_{R+tu} | \over 2\pi} \log {1\over |R+tu |} &= {R^2 \over 2}   \log {1\over R}+ t  \, R \, \left[ -{1\over 2} u + {1\over 2\pi} \log {1\over R} \, \int_{\mathbb{S}^1} u\,d\sigma \right]\\
    &\quad+t^2\left(\frac{1}{4}u^2-\frac{1}{2\pi} u \int_{\mathbb{S}^1} u\,d\sigma+\frac{1}{4\pi} \log {1\over R}\int_{\mathbb{S}^1} u^2\,d\sigma\right) +O(t^3).
\end{align*}
Thus, we conclude
\begin{align*}
    &F_L (R+tu)= {1 \over 2}  R^2 \log {1\over R}\\
    &+ t  \, R \, \left[ -{1\over 2} u + {1\over 2\pi} \log {1\over R} \, \int_{\mathbb{S}^1} u + {1\over 2\pi}  \int_{\mathbb{S}^1} u(z') \log {1\over |z-z'|} d\sigma (z')\right]\\
    &+t^2\left[-\frac{1}{4}u^2-\frac{1}{2\pi} u \int_{\mathbb{S}^1} u\,d\sigma+\frac{1}{4\pi} \log {1\over R}\int_{\mathbb{S}^1} u^2\,d\sigma+\frac{1}{4\pi}\int_{\mathbb{S}^1} (u(z'))^2 \log {1\over |z-z'|} d\sigma (z')\right.\\
&\left.\qquad\quad+\frac{1}{4\pi}u(z)\int_{\mathbb{S}^1} u(z')|z-z'| d\sigma (z')-\frac{1}{8\pi}\int_{\mathbb{S}^1} (u(z'))^2|z-z'| d\sigma (z')\right]+ O(t^3).\qedhere
\end{align*}
\end{proof}

In next proposition we study some properties of the map $\Phi$, defined in~\eqref{defPhi2}.

\begin{prop}
\label{prop-2d}
The operator $\Phi$ satisfies the following properties.
\begin{enumerate}
\item The map $\Phi:\mathcal A\to\mathcal C^{0,\alpha}(\Su)$ is of class $\mathcal C^\infty$.
\item $\Phi(R,0)=0$, for all $R>0$.
\item For all $R>0$ and $v\in\mathcal C^{2,\alpha}(\Su)$ the linearized operator at $u=0$ is given by
\[
D_u\Phi(R,0)[v]=L_R(v),
\]
where
\begin{align*} 
&L_R(v)[z]\\
&=- \Delta_{\mathbb{S}^1} v - v + R^3 \left( -{1\over 2} v + {1\over 2\pi} \log {1\over R} \, \int_{\mathbb{S}^1} v \,d\sigma+ {1\over 2\pi}  \int_{\mathbb{S}^1} v(z') \log {1\over |z-z'|} d\sigma (z')\right).
\end{align*}
\item $L_R v=\lambda v$ has a non trivial solution if and only if $v$ is any spherical harmonics of degree $n\in\N$ on $\Su$ and $\lambda=\lambda_n$, where
\[
\lambda_n=n^2-1-\frac{R^3}{2}\left(1-\frac{1}{n}\right).
\]
\end{enumerate}
\end{prop}

\begin{proof}
To prove (1), we argue as in~\cite[Proposition 7]{Fr19}. Since the map $F_P$ is clearly of class $\mathcal C^\infty$, it is enough to deal with $F_L $. Hence, let rewrite $F_L $ in the following way. Taking into account that $\mathrm{div}\left(\log\left(\frac{e^{\frac12}}{|x|}\right)x\right)=2\log\frac{1}{|x|}$, we have
\begin{align*}
2\int_{\Omega_\varphi}\log\frac{1}{|x-y|}\,dy&=\int_{\Omega_\varphi}\mathrm{div}\left(\log\left(\frac{e^{\frac12}}{|y-x|}\right)(y-x)\right)\,dy\\
&=\int_{\partial\Omega_\varphi}\log\left(\frac{e^{\frac12}}{|y-x|}\right)(y-x)\cdot\nu(y)\,d\sigma(y),
\end{align*}
where $\nu(y)$ denotes the outer unit normal which, taking into account Lemma~\ref{lemma:per}, is given by
\[
\nu(y)=\frac{\varphi(z')z'-\nabla\varphi(z')}{ \sqrt{\varphi^2(z') + |\nabla \varphi(z') |^2} },\quad z'\in\Su,
\]
while the surface measure is
\[
d\sigma(y)=\sqrt{\varphi^2(z') + |\nabla \varphi(z') |^2} \, d\sigma(z'),\quad z'\in\Su.
\]
Thus
\begin{align*}
-2F_L (\varphi)&=-2\int_{\Omega_\varphi}\log\frac{1}{|\varphi(z)z-y|}\,dy\\
&=\int_{\Su}\!(\varphi(z')z'-\nabla \varphi(z'))\cdot(\varphi(z')z'-\varphi(z)z)\log\!\left(\!\frac{e^{\frac12}}{|\varphi(z)z-\varphi(z')z'|}\!\right)\,d\sigma(z')\\
&=\int_{\Su}(\varphi(z')z'-\nabla \varphi(z'))\cdot(\varphi(z')z'-\varphi(z)z)K_L(\varphi,z,z')\,d\sigma(z')\\
&\quad-\int_{\Su}(\varphi(z')z'-\nabla \varphi(z'))\cdot(\varphi(z')z'-\varphi(z)z)\log\left(e^{\frac12}|z-z'|\right)\,d\sigma(z')\\
&=\varphi(z)\int_{\Su}[\varphi(z)-\varphi(z')-(z-z')\cdot\nabla \varphi(z')]K_L(\varphi,z,z')\,d\sigma(z')\\
&\quad-\int_{\Su}(\varphi(z)-\varphi(z'))^2K_L(\varphi,z,z')\,d\sigma(z')\\
&\quad-\frac{\varphi(z)}{2}\int_{\Su}|z-z'|^2K_L(\varphi,z,z')\varphi(z')\,d\sigma(z')\\
&\quad-\int_{\Su}(\varphi(z')z'-\nabla \varphi(z'))\cdot(\varphi(z')z'-\varphi(z)z)\log\left(e^{\frac12}|z-z'|\right)\,d\sigma(z'),
\end{align*}
where
\[
K_L(\varphi,z,z')=\log\frac{|z-z'|}{|\varphi(z)z-\varphi(z')z'|}=-\frac12\log\left(\frac{(\varphi(z)-\varphi(z'))^2}{|z-z'|^2}+\varphi(z)\varphi(z')\right).
\]
 Finally, it is enough to argue as in the last part of the proof of~\cite[Proposition 7]{Fr19}: indeed, it is not hard to check that the proof in~\cite[Section 4]{CFW} - see also~\cite{Fa18} - still works in this framework, and it is even simpler since the kernels we are considering are less singular than those considered there.

Claim (2) is trivial, while the remaining claims are an immediate consequence of Lemma~\ref{lemma:der2D} and the fact that for all $v_n$ spherical harmonic on $\Su$ of degree $n$ one has
\[
{1\over 2\pi}  \int_{\mathbb{S}^1} v_n (z') \log {1\over |z-z'|} d\sigma (z') ={1\over 2n} \,  \, v_n.
\]
Indeed, let us write $z=e^{i\phi}$, $z'=e^{i\phi'}$ and $v_n(z')=e^{in\phi'}$ or $v_n(z')=e^{-in\phi'}$
Taking into account that
\[
|z-z'|^2=2(1-\cos(\phi'-\phi)),
\]
and~\cite[Formula 1.441.2 page 46]{GRbook}
\[
-\frac{1}{2}\log[2(1-\cos\xi)]=\sum_{k=1}^{+\infty}\frac{\cos(k\xi)}{k},
\]
we can write
\[
\log\left(\frac{1}{|z-z'|}\right)=\sum_{k=1}^{+\infty}\frac{e^{ik(\phi'-\phi)}+e^{-ik(\phi'-\phi})}{2k}.
\]
Hence
\begin{align*}
    &{1\over 2\pi}  \int_{\mathbb{S}^1} v_n (z') \log {1\over |z-z'|} d\sigma (z')\\
    &=\sum_{k=1}^{+\infty}\frac{e^{-ik\phi}}{4\pi k}\int_0^{2\pi}e^{i(k\pm n)\phi'}\,d\phi'+\sum_{k=1}^{+\infty}\frac{e^{ik\phi}}{4\pi k}\int_0^{2\pi}e^{i(-k\pm n)\phi'}\,d\phi'=\frac{e^{\pm i n\phi}}{2n}=\frac1{2n}v_n (z).
\end{align*}
Therefore
\[
L v_n =\left[n^2-1-{R^3 \over 2}\left(1-\frac1n\right) \right]\, v_n.\qedhere
\]
\end{proof}

The following lemma is a trivial consequence of the preceding result.

\begin{lemma}
\label{lemma:autov-log}
One has
\begin{enumerate}
\item $\lambda_1=0$, for all $R>0$, and the corresponding eigenfunctions are the spherical harmonics of degree 1,
\item for all $n\in\N$, $n\ge2$, $\lambda_n=0$ if and only if $R=\mathfrak R_n=(2n(n+1))^{1/3}$,
\item $\mathfrak R_{n_1}<\mathfrak R_{n_2}$, for all  $n_1,n_2\in\N$, $2\leq n_1<n_2$,
\item $\mathfrak R_n\to+\infty$, as $n\to+\infty.$
\end{enumerate}
\end{lemma}

From now up to the end of the section, we use polar coordinate coordinates on $\R^2$ (and, by taking $\rho=1$, on $\Su$), that is
\[
\begin{aligned}
x_1 &=\rho\, \cos\phi, \\
x_2 &=\rho\, \sin\phi,\\
\end{aligned}
\]
with $\rho>0$ and $\phi \in [0,2\pi)$.
For all $m\in\N$, let us set
\[
\mathcal X_m=\{u\in\mathcal C^{2,\alpha}(\Su):\,u(\phi)=u(\phi+2\pi/m)=u(\pi-\phi)\},
\]
and
\[
\mathcal Y_m=\{u\in\mathcal C^{0,\alpha}(\Su):\,u(\phi)=u(\phi+2\pi/m)=u(\pi-\phi)\}.
\]

The operator $\Phi$ restricted to spaces introduced above satisfies the following properties. 

\begin{lemma}
\label{lemma:bif2d1}
Let $n\ge2$ and consider the restriction of $\Phi$ to $\mathcal A\cap\mathcal X_n$. Then $\Phi:\mathcal X_n\to \mathcal Y_n$,
\begin{gather*}
\mathrm{ker}( L_{\rn})=\mathrm{span}\{\sin(n\phi)\},\\
\quad\mathrm{ran}(L_{\rn})=\left\{u\in\mathcal Y_n:\int_{\Su}\sin(n\phi)u(\phi)\,d\phi=0\right\},
\end{gather*}
and
\[
\frac{d}{dR}_{|R=\rn}L_R\sin(n\phi)=-6n(n^2-1)\rn^{-1} \sin(n\phi)\not\in\mathrm{ran}(L_{\rn}).
\]
\end{lemma}
\begin{proof}
 It is enough to argue as in~\cite[Proposition 9]{Fr19}, once we observe that the unique spherical harmonic of degree $1$ or $n$ in $\mathcal X_n$ is $\sin(n\phi)$.

We just point out that, given $u \in \mathcal X_n$, then $\Phi(R,u) \in \mathcal Y_n$. This is obvious for the local terms in $\Phi$, while for the nonlocal one it easily follows by a suitable change of variables (with Jacobian equal to $1$).
Then, to compute the range, the inclusion $\subset$ is straightforward. To prove the reverse inclusion, we still follow~\cite{Fr19}: for any $\xi \in\mathcal Y_m$, with $\int_{\Su}\sin(n\phi)\xi(\phi)\,d\phi=0$, there exists $u \in H^2(\Su)$ such that $L_{\rn}u = \xi$. Since
\[
 \int_{\mathbb{S}^1} v(z')\,d\sigma(z'),\quad\int_{\mathbb S^1} v (z')\log\frac{1}{|z-z'|^\lambda }\,d\sigma(z') \in H^2(\Su),
\]
for $v \in H^2(\Su)$ (the second can be seen, for instance, by means of the regularity theory for the fractional Laplacian), it follows from the equation that
\[
- \Delta_{\mathbb{S}^1} u = u - R^3 \left( -{1\over 2} u + {1\over 2\pi} \log {1\over R} \, \int_{\mathbb{S}^1} u + {1\over 2\pi}  \int_{\mathbb{S}^1} u(z') \log {1\over |z-z'|} d\sigma (z')\right),
\]
belongs to $H^2(\Su)$, and by classical regularity theory for the Laplacian and Sobolev embedding, we can prove that $u \in \mathcal C^{2,\alpha}(\Su)$.
\end{proof}

\subsection{Proof of Theorem\texorpdfstring{~\ref{thm_N=2}}{ 2}}

Taking into account Lemma~\ref{lemma:bif2d1}, we can now apply the Crandall-Rabinowitz Theorem to prove the first part of the theorem, except from the fact that $\frac{d}{dt}{R_n^t}_{|t=0}=0$, which will be established in Lemma~\ref{lemma:derRag2D} below.

\begin{remark}
\label{rmk-bif2}
As in Remark~\ref{rmk-bif3}, we can restrict our operator from
\[
\mathbf X_n=\{u\in\mathcal C^{2,\alpha}(\Su):\,u(\phi,\theta)=u(\phi,\pi-\theta)=u(2\pi/m-\phi,\theta)\},
\]
to
\[
\mathbf Y_n=\{u\in\mathcal C^{0,\alpha}(\Su):\,u(\phi,\theta)=u(\phi,\pi-\theta)=u(2\pi/m-\phi,\theta)\},
\]
to deduce the existence of another branch of solutions whose first order expansion is given by $v_n^t(\phi)=t\, \mathfrak \cos(n\phi) +O(t^2)$ in $\mathcal C^{2,\alpha}(\Su)$. Again, if $u\in\mathcal X_n$, then $v (\phi) = u (\phi+\pi/n)\in\mathbf X_n$ and this means that the two solutions $v_n^t$ and $u_n^t$ are geometrically the same solution of Problem~\eqref{eq:2}.
\end{remark}

\begin{lemma}
\label{lemma:derRag2D}
We have
\[
\frac{d}{dt}{R_n^t}_{|t=0}=0.
\]
\end{lemma}
\begin{proof}
For convenience in the proof, we write $\mathrm{ker}( L_{\rn})=\mathrm{span}\{e^{in\phi}\}$, thus by~\cite[formula (I.6.3)]{KieBook}, we have
\[
\frac{d}{dt}{R_n^t}_{|t=0}=-\frac12\frac{\int_{\Su} D^2_{uu}\Phi(\rn,0)[e^{in\phi},e^{in\phi}]\cdot e^{-in\phi}\,d\sigma}{\int_{\Su} D^2_{Ru}\Phi(\rn,0)[e^{in\phi}]\cdot e^{-in\phi}\,d\sigma}.
\]
Taking into account Lemmas~\ref{FP2ordine} and~\ref{FL2ordine} and arguing as in~\cite[Proposition 10]{Fr19} we can show that
\[
\frac12 D^2_{uu}\Phi(R,0)[v,v]=R^2 Q_R(v),
\]
where
\begin{align*}
&Q_R(v)[z]\\
&=\frac{2}{R^3}\left(\!v(z)\Delta_{\Sd}v(z)+\frac{1}{4}|\nabla v(z)|^2+\frac{1}{2}(v(z))^2+2\!\right)\!-\frac{1}{4}(v(z))^2-\frac{1}{2\pi} v(z) \int_{\mathbb{S}^1}\! v(z')\,d\sigma\\
&\quad+\frac{1}{4\pi} \log {1\over R}\int_{\mathbb{S}^1} (v(z))^2\,d\sigma+\frac{1}{4\pi}\int_{\mathbb{S}^1} (v(z'))^2 \log {1\over |z-z'|} d\sigma (z')\\
&\quad+\frac{1}{4\pi}v(z)\int_{\mathbb{S}^1} v(z')|z-z'| d\sigma (z')-\frac{1}{8\pi}\int_{\mathbb{S}^1} (v(z'))^2|z-z'| d\sigma (z')
\end{align*}
We recall that for $v=e^{i\ell\phi}$, we have
\[
\Delta_{\Su}v=-\ell^2v,\qquad |\nabla v|^2=-\ell^2 v^2,
\]
and
\[
\int_{\mathbb{S}^1} v (z') d\sigma (z')=0,\qquad{1\over 2\pi}  \int_{\mathbb{S}^1} v (z') \log {1\over |z-z'|} d\sigma (z') ={1\over 2\ell} v.
\]
Moreover, we claim that
\[
\int_{\mathbb{S}^1} v(z')|z-z'| d\sigma (z')=-\frac{8}{(2\ell+1)(2\ell-1)}v.
\]
Indeed, writing $z=e^{i\phi}$, $z'=e^{i\phi'}$ and taking into account that 
\[
|z-z'|=2\sin\left(\frac{\phi'-\phi}{2}\right),
\]
we have
\begin{align*}
\int_{\mathbb{S}^1} v(z')|z-z'| d\sigma (z')&=\int_0^{2\pi}e^{i\ell\phi'} 2\sin\left(\frac{\phi'-\phi}{2}\right)\,d\phi'\\
&=e^{i\ell\phi}\int_0^{2\pi}e^{-i\ell\varsigma} 2\sin\left(\frac{\varsigma}{2}\right)\,d\varsigma\\
&=-\frac{8}{(2\ell+1)(2\ell-1)}e^{i\ell\phi}.
\end{align*}
Taking into account that $(e^{in\phi})^2=e^{i2n\phi}$, it is then immediate to see that $Q_R(e^{in\phi})$ is proportional to $e^{i2n\phi}$ and then
\[
\int_{\Su} D^2_{uu}\Phi(\rn,0)[e^{in\phi},e^{in\phi}]\cdot e^{-in\phi}\,d\sigma=0,
\]
concluding the proof of the lemma.
\end{proof}

To conclude the proof, we still need to show that there are no critical points of $E$ arbitrarily close to $B_R$ if $R \not= \rn$. This is established in the following proposition.

\begin{prop}
\label{prop:nobif2D}
For all $R\not=\rn$, $n\ge2$, there are no critical points of $E$ arbitrarily close to $B_R$ in $\mathcal C^{2,\alpha}(\Su)$.
\end{prop}
\begin{proof}
The proof is just a simplified version of the one of Proposition~\ref{prop:nobif3D}, so we only sketch it. Indeed, for $\bar R\not=\rn$, one has
\[
\mathrm{ker}(L_{\bar R})=\{z_1,z_2\},
\]
where $z=(z_1,z_2)\in\Su$ is given by $z_i=x_i/|x|$, for $x=(x_1,x_2)\in\R^2$. Calling $\Psi(R,u)=R^2(F(R+u)-F(R))$ our operator acting from $(0,\infty)\times\mathcal X\to\mathcal Y$, with
\begin{gather*}
\mathcal X=\{u\in H^2(\Su):\,u(z_1,z_2)=u(z_1,-z_2)\},\\
\mathcal Y=\{u\in L^2(\Su):\,u(z_1,z_2)=u(z_1,-z_2)\},
\end{gather*}
 we have
\[
\mathrm{ker}(D_u\Psi(\bar R,0))=\langle{z_1}\rangle_{\mathcal X},\quad\text{and}\quad \mathrm{ran}(D_u\Psi(\bar R,0))=(\langle{z_1}\rangle_{\mathcal Y})^\perp,
\]
where we denote by $\langle{\,\cdot\,}\rangle_{\mathcal X}$, or $\langle{\,\cdot\,}\rangle_{\mathcal Y}$, the span in $\mathcal X$ or $\mathcal Y$, respectively.

Hence, there exists a unique function $\xi:B_\delta(\bar R,0)\subset(0,\infty)\times\langle z_1\rangle_{\mathcal{X}}\to(\langle z_1\rangle_{\mathcal{X}})^\perp$ such that
\[
(1-Q)\Psi(R,v+w)=0\quad\text{if and only if}\quad w=\xi(R,v),
\]
and the problem $\Psi(R,u)=0$ reduces to $Q\Psi(R,v+\xi(R,v))=0$, where $(R,v)\in B_\delta(\bar R,0)\subset(0,\infty)\times\langle z_1\rangle_{\mathcal{X}}$, and where $P:\mathcal X\to  \langle{z_1}\rangle_{\mathcal X}$ and $Q:\mathcal Y\to \langle{z_1}\rangle_{\mathcal Y}$ are the canonical projections.

Here again we can write $\xi$ explicitly as
\[
\xi(R,\tau)=\text{sgn}(\tau)\sqrt{R^{2}-\tau^{2}z_2^2}-R\in(\langle z_1 \rangle_{\mathcal X})^{\perp},
\]
where we identify $\tau\in\R$ with $\tau z_1\in\langle{z_1}\rangle_{\mathcal X}$.
Taking into account the reduced equation, all the solutions close to $(\bar R,0)$ are given by $(R,\tau z_1+\psi(R,\tau))$. However, $\Omega_{R+\tau z_1+\psi(R,\tau)}$ is a ball of radius $R$, translated by $\tau$ in the $x_1$-direction, and is therefore trivial for the problem.

A similar argument shows that there are no nontrivial solutions in
\[
\{u\in H^2(\Su):\,u(z_1,z_2)=u(-z_1,z_2)\},
\]
and to conclude it is enough to observe that, if we restrict ourselves to the orthogonal complement of the kernel, then by the Implicit Function Theorem the only solutions are given by $(R,0)$.
\end{proof}

\section{The case of dimension \texorpdfstring{$N\geq4$}{N>=4} and Coulomb interaction with parameter \texorpdfstring{$0<\lambda <N-1$}{0<lambda<N-1}}
\label{sec_further}

To prove Theorem~\ref{thm_lambda}, we proceed as in the previous sections. As in the $3D$ case, after some common preliminary results, we consider the case $n$ odd and $n$ even separately.

\begin{lemma}
\label{lemma1_4D}
Let $\varphi\in\mathcal C^{2,\alpha}(\Sn)$, $\varphi\ge0$. Then 
\[
{d \over dt}_{| t=0} \, E_\lambda \left( |\Omega_\varphi |^{1\over N} \, {\Omega_{\varphi +t u} \over |\Omega_{\varphi+tu}|^{1\over N}} \right) =0,
\]
for all $u\in \mathcal C^{2,\alpha}(\Sn)$, if and only if
\[
-\mathrm{div}\left(\frac{\nabla\varphi}{\varphi\sqrt{ \varphi^2 + |\nabla \varphi|^2}}\right)+\frac{N}{\sqrt{ \varphi^2 + |\nabla \varphi|^2}}-\frac{\sqrt{ \varphi^2 + |\nabla \varphi|^2}}{\varphi^2}+ \,  \int_{\Omega_\varphi}    {  d x\over |x - \varphi (z)   z|^\lambda}=\mu ,
\]
where
\[
\mu= {N-1 \over N} \, {{\mathrm{Per}} (\Omega_{\varphi }) \over  |\Omega_\varphi |} +{2N-\lambda \over N} \,\frac{D_\lambda(\Omega_\varphi) }{|\Omega_\varphi |}.
\]
\end{lemma}
\begin{proof}

Using that for any $c>0$
\[
{\mbox {Per}} (c \Omega) = c^{N-1 }
{\mbox {Per}} ( \Omega), \quad \text{and}\quad D_\lambda(c\Omega)=c^{2N-\lambda}D_\lambda(\Omega),
\]
we get
\[
E_\lambda \left( |\Omega_\varphi |^{1\over N}  {\Omega_{\varphi +t u} \over |\Omega_{\varphi+tu}|^{1\over N}} \right) =
          { |\Omega_\varphi |^{N-1 \over N} \over |\Omega_{\varphi+tu}|^{N-1\over N}} {\mbox {Per}} (\Omega_{\varphi + t u})+ { |\Omega_\varphi |^{2N-\lambda \over N} \over |\Omega_{\varphi+tu}|^{2N-\lambda\over N}}D_\lambda(\Omega_{\varphi + t u}),
  \]
  for any $u\in\mathcal C^{2,\alpha}(\Sn)$.
    Observe that
    \begin{align*}
        |\Omega_{\varphi} | &= \int_{\mathbb S^{N-1}} \int_0^{\varphi (z)} r^{N-1}\, dr d\sigma(z) = {1\over N} \int_{\mathbb S^{N-1} }
        \varphi^N (z)\, d\sigma(z).
    \end{align*}
    Thus
  \begin{align*}
        |\Omega_{\varphi +tu} | &=   {1\over N} \int_{\mathbb S^{N-1} }
        (\varphi  + t u  )^N \,  d \sigma \\
        &=   {1\over N} \int_{\mathbb S^{N-1} }
        \varphi^N (z)\,  d \sigma(z) + t \int_{\mathbb S^{N-1}} \varphi^{N-1}  u  \,  d \sigma + O (t^2)\\
         &=   |\Omega_{\varphi } | + t \int_{\mathbb S^{N-1}} \varphi^{N-1}  u \,  d \sigma + O (t^2),
    \end{align*}    
    as $t \to 0$.
Taking into account that
\begin{equation}
\label{formula-perimetro}
     {\mbox {Per}} (\Omega_{\varphi })=\int_{\mathbb S^{N-1}} \varphi^{N-2} \sqrt{\varphi^2 + |\nabla \varphi |^2} \, d\sigma,
\end{equation}
see Lemma~\ref{lemma:per}, we can compute
\begin{align*}
    &{\mbox {Per}} (\Omega_{\varphi +tu})\\
    &= \int_{\mathbb S^{N-1} }\, (\varphi + t u )^{N-2} \sqrt{ (\varphi + t u )^2 + |\nabla (\varphi + t u)|^2}  \, d \sigma\\
    &={\mbox {Per}} (\Omega_{\varphi})\\
     &\quad+ t \! \int_{\mathbb S^{N-1}\!\!\!} \left(\! \!{\varphi^{N-2} \over  \sqrt{ \varphi^2 + |\nabla \varphi|^2}} (u \varphi +\! \nabla u \nabla \varphi) +\! (N\!-2) \varphi^{N-3} u  \sqrt{ \varphi^2 + |\nabla \varphi|^2}\right) \!d \sigma\! + O(t^2).
\end{align*}
Hence
\be\label{uno}
\begin{aligned}
     &{ |\Omega_\varphi |^{N-1 \over N} \over |\Omega_{\varphi+tu}|^{N-1\over N}}  {\mbox {Per}} (\Omega_{\varphi + t u}) \\
     &={\mbox {Per}} (\Omega_{\varphi })- t \,  {N-1 \over N} \, {{\mbox {Per}} (\Omega_{\varphi }) \over  |\Omega_\varphi |} \int_{\mathbb S^{N-1}} \varphi^{N-1} u \,d\sigma  \\
     &\quad+\!t\!\! \int_{\mathbb S^{N-1}} \!\! \left( {\varphi^{N-2} \over  \sqrt{ \varphi^2\! +\! |\nabla \varphi|^2}} (u \varphi\! +\! \nabla u \nabla \varphi)\! +\! (N\!-\!2) \varphi^{N-3} u  \sqrt{ \varphi^2 \!+\! |\nabla \varphi|^2}\right) d \sigma \\
    &\quad+ O(t^2) .
\end{aligned}
\ee
We now proceed with the term $D_\lambda$
\begin{align*}
   {1\over 2} \int_{\Omega_{\varphi +tu}}& \int_{\Omega_{\varphi +t u}} {1\over |x-y|^\lambda } \, dx \, dy\\
        &= {1\over 2}\int_{\mathbb S^{N-1}} \int_{\mathbb S^{N-1}} \int_0^{\varphi + t u}   \int_0^{\varphi + t u} {\rho^{N-1}  d \rho (\rho')^{N-1} d \rho' d \sigma(z) d \sigma(z')\over |\rho z - \rho' z'|^\lambda}\\
     &=D_\lambda(\Omega_\varphi)+ t \int_{\mathbb S^{N-1}}  u(z') \, (\varphi (z'))^{N-1} \,  \int_{\Omega_\varphi}    {  d x\over |x - \varphi (z')   z'|^\lambda} d \sigma (z') + O(t^2).
\end{align*}
Thus
\be\label{due}
\begin{aligned}
 &{ |\Omega_\varphi |^{2N-\lambda \over N} \over |\Omega_{\varphi+tu}|^{2N-\lambda\over N}}   D_\lambda(\Omega_{\varphi +tu}) \\
 &=D_\lambda(\Omega_\varphi)- t \,  {2N-\lambda \over N} \,\frac{D_\lambda(\Omega_\varphi) }{|\Omega_\varphi |}\int_{\mathbb S^{N-1}} \varphi^{N-1} u \,d\sigma  \\
     &\quad+t \, \int_{\mathbb S^{N-1}}  u(z') \, (\varphi (z'))^{N-1} \,  \int_{\Omega_\varphi}    {  d x\over |x - \varphi (z')   z'|^\lambda} d \sigma (z')+ O(t^2) .
\end{aligned}
\ee
From~\eqref{uno} and~\eqref{due}
\[
\begin{aligned}
  E_\lambda &\left( |\Omega_\varphi |^{1\over N} \, {\Omega_{\varphi +t u} \over |\Omega_{\varphi+tu}|^{1\over N}} \right) =E_\lambda(\Omega_\varphi)\\
  &+t\left[
 -\left( {N-1 \over N} \, {{\mbox {Per}} (\Omega_{\varphi }) \over  |\Omega_\varphi |} +{2N-\lambda \over N} \,\frac{D_\lambda(\Omega_\varphi) }{|\Omega_\varphi |}\right)\int_{\mathbb S^{N-1}} \varphi^{N-1} u\,d\sigma \right.\\
&\left.
  + \int_{\mathbb S^{N-1}} \, \left( {\varphi^{N-2} \over  \sqrt{ \varphi^2 + |\nabla \varphi|^2}} (u \varphi + \nabla u \nabla \varphi) + (N-2) \varphi^{N-3} u  \sqrt{ \varphi^2 + |\nabla \varphi|^2}\right) d \sigma\right.\\
&\left.+\int_{\mathbb S^{N-1}}  u(z') \, (\varphi (z'))^{N-1} \,  \int_{\Omega_\varphi}    {  d x\over |x - \varphi (z')   z'|^\lambda} d \sigma (z')  \right]+O(t^2).
\end{aligned}
\]
Hence $\varphi$ is a solution if and only if for all $u\in\mathcal C^{2,\alpha}(\Sn)$
\begin{align*}
 0= &-\left( {N-1 \over N} \, {{\mbox {Per}} (\Omega_{\varphi }) \over  |\Omega_\varphi |} +{2N-\lambda \over N} \,\frac{D_\lambda(\Omega_\varphi) }{|\Omega_\varphi |}\right)\int_{\mathbb S^{N-1}} \varphi^{N-1} u \,d\sigma\\
  &+
   \int_{\mathbb S^{N-1}} \, \left( {\varphi^{N-2} \over  \sqrt{ \varphi^2 + |\nabla \varphi|^2}} (u \varphi + \nabla u \nabla \varphi) + (N-2) \varphi^{N-3} u  \sqrt{ \varphi^2 + |\nabla \varphi|^2}\right) d \sigma\\
   &+\int_{\mathbb S^{N-1}}  u\,\varphi^{N-1} \,  \int_{\Omega_\varphi}    {  d x\over |x - \varphi  (z)  z|^\lambda} d \sigma,
\end{align*}
that is $\varphi$ is a solution of the following equation on $\mathbb S^{N-1}$
\[
\begin{split}
    &-\mathrm{div}\left(\frac{\varphi^{N-2}\nabla\varphi}{\sqrt{ \varphi^2 + |\nabla \varphi|^2}}\right)+\frac{\varphi^{N-1}}{\sqrt{ \varphi^2 + |\nabla \varphi|^2}}+\\
    &+(N-2)\varphi^{N-3}\sqrt{ \varphi^2 + |\nabla \varphi|^2}+\varphi^{N-1} \,  \int_{\Omega_\varphi}    {  d x\over |x - \varphi (z)   z|^\lambda}=\mu \varphi^{N-1},
\end{split}
\]
where $\mu$ is given by
\[
\mu= {N-1 \over N} \, {{\mbox {Per}} (\Omega_{\varphi }) \over  |\Omega_\varphi |} +{2N-\lambda \over N} \,\frac{D_\lambda(\Omega_\varphi) }{|\Omega_\varphi |}.
\]
Clearly, it can be rewritten as
\[
-\mathrm{div}\left(\frac{\nabla\varphi}{\varphi\sqrt{ \varphi^2 + |\nabla \varphi|^2}}\right)+\frac{N}{\sqrt{ \varphi^2 + |\nabla \varphi|^2}}-\frac{\sqrt{ \varphi^2 + |\nabla \varphi|^2}}{\varphi^2}+ \,  \int_{\Omega_\varphi}    {  d x\over |x - \varphi (z)   z|^\lambda}=\mu ,
\]
 on $\mathbb S^{N-1}$.
\end{proof}

For $\varphi\in\mathcal C^{2,\alpha}(\Sn)$, let us write
\[
F(\varphi)=\!-\mathrm{div}\left(\frac{\nabla\varphi}{\varphi\sqrt{ \varphi^2 + |\nabla \varphi|^2}}\right)+\frac{N}{\sqrt{ \varphi^2 + |\nabla \varphi|^2}}-\frac{\sqrt{ \varphi^2 + |\nabla \varphi|^2}}{\varphi^2}+ \int_{\Omega_\varphi}    \!{  d x\over |x - \varphi (z)   z|^\lambda},
\]
and  we split it into
\[
F_P(\varphi)=-\mathrm{div}\left(\frac{\nabla\varphi}{\varphi\sqrt{ \varphi^2 + |\nabla \varphi|^2}}\right)+\frac{N}{\sqrt{ \varphi^2 + |\nabla \varphi|^2}}-\frac{\sqrt{ \varphi^2 + |\nabla \varphi|^2}}{\varphi^2},
\]
and
\[
F_C(\varphi)=  \int_{\Omega_\varphi}    {  d x\over |x - \varphi (z)   z|^\lambda},
\]
so that $F=F_P+F_C$.
We have the following lemma.

\begin{lemma}
\label{lemma:der}
For all $R>0$ and $u\in\mathcal C^{2,\alpha}(\mathbb S^{N-1})$ one has
\begin{align*}
 \lim_{t\to0}&\frac{  F(R+tu)-F(R)}{t}=\frac{1}{R^2}\left(-\Delta_{\Sn} u-(N-1)u\right)
 \\&+R^{N-\lambda-1}\left[\int_{\mathbb S^{N-1}}\frac{u(z')}{|z-z'|^\lambda}\,d\sigma(z')-(C_{2,\lambda}-C_{1,\lambda}(N-\lambda))u(z)\right],
 \end{align*}
 where
\[
C_{1,\lambda} =\int_{B_1}\frac{1}{|z-y|^\lambda}\,dy,
\]
and
\[
C_{2, \lambda}=\int_{\mathbb S^{N-1}}\frac{1}{|z-z'|^\lambda}\,d\sigma(z').
\]
\end{lemma}

This is an immediate consequence the expansions of $F_P$ and $F_C$ that are computed in the following two lemmas.

\begin{lemma}
Given $R>0$ and $u\in\mathcal C^{2,\alpha}(\mathbb S^{N-1})$, one has the following expansion of $F_P$, as $t\to0$
  \begin{align*}
    F_P(R+tu)&=\frac{N-1}{R}+\frac{t}{R^2}\left(-\Delta_{\Sn} u-(N-1)u\right)+O(t^2).
    \end{align*}
\end{lemma}
\begin{proof}
It is enough to argue as in~\cite[Lemma 14]{Fr19}.
\end{proof}

\begin{lemma}
Given $R>0$ and $\varphi\in\mathcal C^{2,\alpha}(\mathbb S^{N-1})$, one has the following expansion of $F_C$, as $t\to0$
     \begin{align*}
    F_C&(R+t\varphi)=R^{N-\lambda}C_{1,\lambda}\\
    &+t\,R^{N-\lambda-1}\left[\int_{\mathbb S^{N-1}}\frac{u(z')}{|z-z'|^\lambda}\,d\sigma(z')-(C_{2,\lambda}-C_{1,\lambda}(N-\lambda))u(z)\right]+O(t^2).
    \end{align*}
\end{lemma}
\begin{proof}
Arguing as in~\cite[Lemma 15]{Fr19}, we have
\be
\begin{split}
\label{F_C}
    F_C(\varphi)&=\varphi(z)^{N-\lambda}\int_{\mathbb S^{N-1}}\int_0^{\varphi(z')/\varphi(z)}\frac{r^{N-1}}{|z-rz'|^\lambda}\,drd\sigma(z')\\
        &=\varphi(z)^{N-\lambda}\int_{\mathbb S^{N-1}}\int_0^{1+tg(t,z,z')}f(r,z,z')\,drd\sigma(z'),
    \end{split}
    \ee
    where for $\varphi=R+tu$
    \[
    g(t,z,z')=\frac{u(z')-u(z)}{R+tu(z)}=\frac{u(z')-u(z)}{R}-u(z)\frac{u(z')-u(z)}{R^2}t+O(t^2),
    \]
    and
    \[
    f(r,z,z')=\frac{r^{N-1}}{|z-rz'|^\lambda}.
    \]
Note that
\[
\int_0^{1+tg(t,z,z')}f(r,z,z')\,dr=\int_0^1f(r,z,z')\,dr+tf(1,z,z')g(t,z,z')+ O(t^2).
\]
Thus
\[
\int_0^{1+tg(t,z,z')}f(r,z,z')\,dr=\int_0^1\frac{r^{N-1}}{|z-rz'|^\lambda}\,dr+t\frac{u(z')-u(z)}{R{|z-z'|^\lambda}}+O(t^2).
\]
Recalling that $C_{1,\lambda} =\int_{B_1}\frac{1}{|z-y|^\lambda}\,dy$ and $C_{2, \lambda}=\int_{\mathbb S^{N-1}}\frac{1}{|z-z'|^\lambda}\,d\sigma(z')$, we can rewrite it as
\[
\int_0^{1+tg(t,z,z')}f(r,z,z')\,dr=C_{1,\lambda}+\frac{t}{R}\left(\int_{\mathbb{S}^{N-1}}\frac{u(z')}{|z-z'|^\lambda}\,d\sigma(z')-C_{2,\lambda}u(z)\right)+O(t^2).
\]
Together with
\[
\varphi(z)^{N-\lambda}=R^{N-\lambda}+(N-\lambda)R^{N-\lambda-1}u(z)t+O(t^2),
\]
and~\eqref{F_C} we have
    \begin{align*}
        F_C&(R+tu)=R^{N-\lambda}C_{1,\lambda}\\
        &+R^{N-\lambda-1}\left[\int_{\mathbb S^{N-1}}\frac{u(z')}{|z-z'|^\lambda}\,d\sigma(z')-(C_{2,\lambda}-C_{1,\lambda}(N-\lambda))u(z)\right]t+O(t^2).\qedhere
    \end{align*}
\end{proof}

To introduce our operator let us set
\[
\mathcal A=\{(R,u)\in(0,+\infty)\times\mathcal C^{2,\alpha}(\Sn):u>-R\}.
\]
Hence, we can define $\Phi:\mathcal A\to\mathcal C^{0,\alpha}(\Sn)$ by
\[
\Phi(R,u)= R^2 \left( F(R+u)-F(R) \right).
\]
Similarly to the two- and three-dimensional cases, the operator $\Phi$ satisfies the following properties, whose proof follows the same lines as~\cite[Proposition 7 and 8]{Fr19}.

\begin{prop}
The operator $\Phi$ satisfies the following properties.
\begin{enumerate}
\item The map $\Phi:\mathcal A\to\mathcal C^{0,\alpha}(\Sn)$ is of class $\mathcal C^\infty$.
\item $\Phi(R,0)=0$, for all $R>0$.
\item For all $R>0$ and $v\in\mathcal C^{2,\alpha}(\Sn)$ the linearized operator at $u=0$ is given by
\[
D_u\Phi(R,0)[v]=L_R(v),
\]
where
\[
L_R(v)[z]= -\Delta_{\mathbb S^{N-1}} v - (N-1) v +R^{N+1-\lambda} \left( \int_{\mathbb S^{N-1}} {v (z') \over |z-z'|^\lambda }\,d\sigma(z') + C (\lambda, N) v\right),
\]
where $C (\lambda, N)= (N-\lambda ) \, C_{1,\lambda} - C_{2,\lambda }=-\lambda\pi^{\frac N2}\frac{\Gamma(N-1-\lambda)}{ \Gamma (N -{\lambda \over 2})\Gamma(\frac{N-\lambda}{2})}$.
\item $L_R v=\lambda v$ has a non trivial solution if and only if $v$ is any spherical harmonics of degree $n\in\N$ and $\lambda=\lambda_n$, where
\[
\lambda_n=n (n +N-2) - (N-1) +R^{N+1-\lambda}( \mu_n + C (\lambda, N) ),
\]
with
\[
\mu_n = 2^{N-1-\lambda} \, \pi^{N-1 \over 2}\,\Gamma \left({N-1-\lambda \over 2}\right)\,\frac{\prod_{k=0}^{n-1}\left(\frac\lambda2+k\right)} {\Gamma (N-1 -{\lambda \over 2} + n)}.
\]
\end{enumerate}
\end{prop}
\begin{proof}
To prove (1), we argue as in~\cite[Proposition 7]{Fr19}. Since the map $F_P$ is clearly of class $\mathcal C^\infty$, it is enough to deal with $F_C$. Hence, let rewrite $F_C$ in the following way. Taking into account that $\mathrm{div}\left(\frac{x}{|x|^\lambda}\right)=(N-\lambda)\frac{1}{|x|^\lambda}$, we have
\[
(N-\lambda)\int_{\Omega_\varphi}\frac{1}{|x-y|^\lambda}\,dy=\int_{\Omega_\varphi}\mathrm{div}\left(\frac{y-x}{|y-x|^\lambda}\right)\,dy=\int_{\partial\Omega_\varphi}\nu(y)\cdot\frac{y-x}{|y-x|^\lambda}\,d\sigma(y),
\]
where, taking into account Lemma~\ref{lemma:per}, $\nu(y)$ is the outer unit normal given by
\[
\nu(y)=\frac{\varphi(z')z'-\nabla\varphi(z')}{ \sqrt{\varphi^2(z') + |\nabla \varphi(z') |^2} },\quad z'\in\Sn,
\]
and the surface measure is
\[
d\sigma(y)= \varphi^{N-2}(z') \sqrt{\varphi^2(z') + |\nabla \varphi(z') |^2} \, d\sigma(z'),\quad z'\in\Sn.
\]
Thus as in the proof of Proposition~\ref{prop-2d},
\begin{align*}
-&(N-\lambda)F_C(\varphi)=-(N-\lambda)\int_{\Omega_\varphi}\frac{1}{|\varphi(z)z-y|^\lambda}\,dy\\
&=\varphi(z)\int_{\Sn}\frac{\varphi(z)-\varphi(z')-(z-z')\cdot\nabla \varphi(z')}{|z-z'|^\lambda}K_\lambda(\varphi,z,z')(\varphi(z'))^{N-2}\,d\sigma(z')\\
&\quad-\int_{\Sn}\frac{(\varphi(z)-\varphi(z'))^2}{|z-z'|^\lambda}K_\lambda(\varphi,z,z')(\varphi(z'))^{N-2}\,d\sigma(z')\\
&\quad-\frac{\varphi(z)}{2}\int_{\Sn}|z-z'|^{2-\lambda}K_\lambda(\varphi,z,z')(\varphi(z'))^{N-1}\,d\sigma(z'),
\end{align*}
where
\[
K_\lambda(\varphi,z,z')=\frac{|z-z'|^\lambda}{|\varphi(z)z-\varphi(z')z'|^\lambda}=\left(\frac{(\varphi(z)-\varphi(z'))^2}{|z-z'|^2}+\varphi(z)\varphi(z')\right)^{-\lambda/2}.
\]
 Finally, it is enough to argue as in the last part of the proof of~\cite[Proposition 7]{Fr19}, taking into account that the exponent $N+\alpha$ in~\cite[equation (4.17)]{CFW} is replaced by $\lambda$.

Claim (2) is trivial, while the rest of the claims are an immediate consequence of Lemma~\ref{lemma:der} and the fact that for all $v_n$ spherical harmonic of degree $n$ one has
\[
\int_{\mathbb S^{N-1}} {v_n(z') \over |z-z'|^\lambda } \,d\sigma(z')= \mu_n v_n(z), 
\]
with
\begin{align*}
 \mu_n& = 2^{N-1-\lambda} \, \pi^{N-1 \over 2} \, {\lambda \over 2} \, \left( {\lambda \over 2} +1\right) \, \left( {\lambda \over 2} +2 \right) \, \cdots \, \left( {\lambda \over 2}  + n -1\right) \frac{\Gamma ({N-1-\lambda \over 2})} {\Gamma (N-1 -{\lambda \over 2} + n)}\\
 &=  2^{N-1-\lambda} \, \pi^{N-1 \over 2}\,\Gamma \left({N-1-\lambda \over 2}\right)\,\frac{\prod_{k=0}^{n-1}\left(\frac\lambda2+k\right)} {\Gamma (N-1 -{\lambda \over 2} + n)},
\end{align*}
 see for instance~\cite[formulas (13) and (17)]{HAZ12}. For the computation of the exact value of $C (\lambda, N)$ we refer to Lemma~\ref{lemmaC2}, in Appendix~\ref{app:const}.
\end{proof}

Next Lemma introduces the bifurcation radii as in~\eqref{Rnla}.

\begin{lemma}
\label{lemma:autov}
For all $N\ge2$ and $\lambda\in (0,N-1)$ one has
\begin{enumerate}
\item $\lambda_1=0$, for all $R>0$, and the corresponding eigenfunctions are the spherical harmonics of degree 1,
\item for all $n\in\N$, $n\ge2$, $\lambda_n=0$ if and only if $R=\mathfrak R_n$, where we recall that $\mathfrak R_n=\left(\frac{n(n+N-2)-N+1}{-(\mu_n (\lambda , N)+C({\lambda},N))}\right)^{\frac{1}{N+1-\lambda}}$,
\item $\mathfrak R_{n_2}<\mathfrak R_{n_2}$, for all  $n_1,n_2\in\N$, $2\le n_1<n_2$,
\item $\mathfrak R_n\to+\infty$, as $n\to+\infty.$
\end{enumerate}
\end{lemma}
\begin{proof}
To show (1) it is enough to apply the Legendre multiplication formula
\begin{equation}
\label{x}
    \Gamma \left({N-1 -\lambda \over 2} \right) \Gamma \left({N -\lambda \over 2} \right) = 2^{-N +\lambda +2} \, \sqrt{\pi} \, \Gamma (N-1-\lambda ),
 \end{equation}
 to see that
 \begin{align*}
    \mu_1+C (\lambda, N)&= 2^{N-1-\lambda} \, \pi^{N-1 \over 2} \, {\lambda \over 2} \,  {\Gamma ({N-1-\lambda \over 2}) \over \Gamma (N -{\lambda \over 2} )} -\lambda\pi^{\frac N2}\frac{\Gamma(N-1-\lambda)}{ \Gamma (N -{\lambda \over 2})\Gamma(\frac{N-\lambda}{2})} \\
    &= {\la \pi^{N \over 2} \over \Gamma (N-{\lambda \over 2})} \, \left( {2^{N-2-\lambda} \over \sqrt{\pi}} \Gamma \left({N-1-\lambda \over 2}\right) - \frac{\Gamma(N-1-\lambda)}{ \Gamma(\frac{N-\lambda}{2})} \right) =0.
\end{align*}
Taking into account that $\lambda<N-1$, one has
\[
\mu_{n+1}=\frac{\frac\lambda2+n}{N-1-\frac\lambda2+n}\,\mu_n<\mu_n.
\]
Hence, $\mu_n+ C (\lambda, N)<0$, for all $n\ge2$, and (2), (3) and (4) immediately follow. 
\end{proof}

\subsection{Proof of Theorem\texorpdfstring{~\ref{thm_lambda}}{ 3}: the case \texorpdfstring{$n$}{n} even}

Consider the spaces
\[
\mathcal X_k=\{u\in\mathcal C^{2,\alpha}(\Sn):\,\,u(g\,\cdot)=u(\cdot),\,\forall g\in\mathcal O(k)\times \mathcal O(N-k)\},
\]
and
\[
\mathcal Y_k=\{u\in\mathcal C^{0,\alpha}(\Sn):\,u(g\,\cdot)=u(\cdot),\,\forall g\in\mathcal O(k)\times \mathcal O(N-k)\}.
\]

The following lemma shows that the Crandall-Rabinowitz Theorem can be applied.

\begin{lemma}
\label{lemma-ker4d}
For $k=1,\dots, N-1$, consider the restriction of $\Phi$ to $\mathcal A\cap\mathcal X_k$, then there exists a unique spherical harmonic on $\Sn$, invariant under the action of $\mathcal O(k)\times \mathcal O(N-k)$, denoted by $\mathfrak L_k$, such that
\begin{gather*}
\mathrm{ker}( L_{\rn})=\mathrm{span}\{\mathfrak L_{n,k}\},\\
\mathrm{ran}(L_{\rn})=\left\{u\in\mathcal Y_k:\int_{\Sn}\mathfrak L_{n,k}\cdot\,u\,d\sigma=0\right\},
\end{gather*}
and
\[
\frac{d}{dR}_{|R=\rn}L_R\mathfrak L_{n,k}= -(N+1-\lambda)(n(n+N-2)-N+1)\rn^{-1}\mathfrak L_{n,k}\not\in\mathrm{ran}(L_{\rn}).
\]
\end{lemma}
\begin{proof}
Taking into account~\cite[Lemma 6.5]{SM90}, there exists a unique spherical harmonic on $\Sn$ of degree $n$ invariant under the action of $\mathcal O(k)\times \mathcal O(N-k)$, that we denote by $\mathfrak L_k$. Moreover, it is trivial to see that there are no spherical harmonics on $\Sn$ of degree $1$ invariant under the same action.

Also in this case, given $u \in \mathcal X_k$, we have $\Phi(R,u) \in \mathcal Y_k$. For the nonlocal term in $\Phi$, it is just a change of variables (with Jacobian equal to $1$), while for the local terms it is trivial.
Then, to compute the range, the inclusion $\subset$ is straightforward. To prove the reverse inclusion, we still follow~\cite{Fr19}: for any $\xi \in\mathcal Y_k$ with $\int_{\Sn}\mathfrak L_{n,k}\cdot\xi\,d\sigma=0$, there exists $u \in H^2(\Sn)$ such that $L_{\rn}u = \xi$. Since
\[
 \int_{\mathbb S^{N-1}} \frac{v (z')}{|z-z'|^\lambda }\,d\sigma(z') \in H^m(\Sn),
\]
for $v \in H^m(\Sn)$ (this can be seen, for instance, by means of the regularity theory for the fractional Laplacian), it follows from the equation that
\[
-\Delta_{\mathbb S^{N-1}} u = (N-1) u - R^{N+1-\lambda} \left( \int_{\mathbb S^{N-1}} \frac{u (z')}{|z-z'|^\lambda }\,d\sigma(z') + C (\lambda, N) u \right) \in H^2(\Sn),
\]
and by classical regularity theory for the Laplacian, we can prove (eventually up to some iteration depending on $N$) that $u \in \mathcal C^{2,\alpha}(\Sn)$.
\end{proof}

As before, we can now apply the Crandall-Rabinowitz Theorem, and the proof is complete if $n$ is even.

\subsection{Proof of Theorem\texorpdfstring{~\ref{thm_lambda}}{ 3}: the case \texorpdfstring{$n\ge3$}{n>=3} odd}
\label{devorichiamarla}

Here we use spherical coordinates on $\R^N$ (and, by taking $\rho=1$, on $\Sn$), that is
\[
\begin{aligned}
x_1 &=\rho\, \cos\phi \, \sin\theta_1 \cdots \sin\theta_{N-2}, \\
x_2 &=\rho\, \sin\phi \, \sin\theta_1 \cdots \sin\theta_{N-2}, \\
&\ \vdots \\
x_{N-1} &=\rho\, \sin\theta_{N-2} \cos\theta_{N-3}, \\
x_N &=\rho\, \cos\theta_{N-2},
\end{aligned}
\]
with $\rho>0$, $\phi \in [0,2\pi)$ and $\theta_1,\dots\theta_{N-2} \in [0,\pi]$.

Then let us consider the spaces
\begin{align*}
\mathcal X=\{u\in\mathcal C^{2,\alpha}(\Sn):\,&u(\phi,\theta_1,\dots,\theta_{N-2})=u(\pi/n-\phi,\theta_1,\dots,\theta_{N-2})\\
=&u(\phi,\pi-\theta_1,\dots,\theta_{N-2})=\dots=u(\phi,\theta_1,\dots,\pi-\theta_{N-2})\},
\end{align*}
and
\begin{align*}
\mathcal Y=\{u\in\mathcal C^{0,\alpha}(\Sn):\,&u(\phi,\theta_1,\dots,\theta_{N-2})=u(\pi/n-\phi,\theta_1,\dots,\theta_{N-2})\\
=&u(\phi,\pi-\theta_1,\dots,\theta_{N-2})=\dots=u(\phi,\theta_1,\dots,\pi-\theta_{N-2})\}.
\end{align*}

Let us recall, see for instance~\cite{Wnote}, that the spherical harmonics on $\Sn$ of order $n$, can be written as follows (in the following $m=1,\dots,n$ and $ m=i_0\le i_1\le\dots\le  i_{N-2}=n$)
\begin{gather*}
\mathfrak I_n^{i_0,\dots,i_{N-2}}(\theta_1,\dots,\theta_{N-2})=c_n^{i_0,\dots,i_{N-2}} \prod_{k=1}^{N-2}G^{i_{k-1}}_{i_k}(\theta_k,k-1),\\
\mathfrak L^{i_0,\dots,i_{N-2},1}_{n,m}(\phi,\theta_1,\dots,\theta_{N-2})=c_{n,m}^{i_0,\dots,i_{N-2}}\sin(m\phi)\prod_{k=1}^{N-2}G^{i_{k-1}}_{i_k}(\theta_k,k-1),\\
\mathfrak L^{i_0,\dots,i_{N-2},2}_{n,m}(\phi,\theta_1,\dots,\theta_{N-2})=c_{n,m}^{i_0,\dots,i_{N-2}}\cos(m\phi)\prod_{k=1}^{N-2}G^{i_{k-1}}_{i_k}(\theta_k,k-1),
\end{gather*}
where $G_i^0(t,k)$ are the Gegenbauer polynomials
\[
\sum_{i=0}^{+\infty}G_i^0(t,k)s^i=(1-2st+s^2)^{-(1+k)/2},
\]
and $G_i^j(t,k)$ denotes the associated Gegenbauer polynomial that can be expressed as
\[
G_i^j(t,k)=(1-t^2)^{j/2}\frac{d^j}{dt^j}G_i^0(t,k).
\]
Hence, we can show the following lemma.

\begin{lemma}
Consider the restriction of $\Phi$ to $\mathcal A\cap\mathcal X$, then
\begin{gather*}
\mathrm{ker}( L_{\rn})=\mathrm{span}\{\mathfrak L^{n,\dots,n,1}_{n,n}\},\\
\mathrm{ran}(L_{\rn})=\left\{u\in\mathcal Y:\int_{\Sd}\mathfrak L^{n,\dots,n,1}_{n,n}\cdot\,u\,d\sigma=0\right\},
\end{gather*}
and
\[
\frac{d}{dR}_{|R=\rn}L_R \mathfrak L^{n,\dots,n,1}_{n,n}= -(N+1-\lambda)(n(n+N-2)-N+1)\rn^{-1}\mathfrak L^{n,\dots,n,1}_{n,n}\not\in\mathrm{ran}(L_{\rn}).
\]
\end{lemma}
\begin{proof} 
Note that there are no spherical harmonics of degree $1$ in $\mathcal X$, while there is exactly one of degree $n$, which is not hard to prove, see~\cite[Proposition 5.1]{GGT18}, it is $\mathfrak L^{n,\dots,n,1}_{n,n}$, where
\[
\mathfrak L^{n,\dots,n,1}_{n,n}(\phi,\theta_1,\dots,\theta_{N-2})=\sin(n\phi)(\sin(\theta_1))^n\dots(\sin(\theta_{N-2}))^n.
\]
Thus, to conclude the proof, it is enough to argue exactly as in~\cite[Proposition 9]{Fr19}. See the end of the proof of Lemma~\ref{lemma-ker4d} for more details about $\mathrm{ran}(L_{\rn})$.
\end{proof}

Statement (2) in the theorem is then a direct consequence of the Crandall-Rabinowitz Theorem.

Finally, as in the $2$ and $3$ dimensional cases, we can prove that for all $R\not=\rn$, $n\ge2$, there are no critical points of $E$ arbitrarily close to $B_R$ in $\mathcal C^{2,\alpha}(\Sn)$.
This follows as in the proof of Propositions~\ref{prop:nobif3D} and~\ref{prop:nobif2D}, up to minor modifications.

\appendix
\section{}
\label{appendix}

\subsection{Proof of\texorpdfstring{~\eqref{integrals}}{ (3.6)}}
\label{app:int}
Let $B_1$ be the unit ball in $\R^2$.
Here we show the claim~\eqref{integrals}, that is
\[
\int_{B_1} \log {1\over |x-z|} dx = 0\quad\text{and}\quad \int_{\mathbb{S}^1} \log {1\over |z-z'|} d\sigma (z') =0.
\]

\begin{proof}[Proof of~\eqref{integrals}]
Since  $\int_{B_1} \log {1\over |x-z|} dx $ is constant on $\mathbb{S}^1$ as a function of $z$, we can fix $z=(1,0)$ to get
\[
\begin{aligned}
\int_{B_1} \log {1\over |x-(1,0)|} dx  &= \int_{0}^{2\pi}\int_{0}^{1} \log\!\big(r^{2}-2r\cos\theta+1\big)\,r\,dr\,d\theta\\
&= \int_{0}^{2\pi}\int_{0}^{1} 2\log\!\big|1-re^{i\theta}\big|\,r\,dr\,d\theta \\[6pt]
  &= 2\int_{0}^{1}\left(\int_{0}^{2\pi}\log\!\big|1-re^{i\theta}\big|\,d\theta\right) r\,dr.
\end{aligned}
\]
Let
\[
u(z)=\log|1-z|\qquad (z\in\{w:|w|<1\}).
\]
Since $z\mapsto 1-z$ is holomorphic and nonvanishing on the unit disk, $u$ is harmonic there (it is the real part of the holomorphic function $\log(1-z)$). Hence by the mean value property for harmonic functions, for every fixed $0\le r<1$
\[
\frac{1}{2\pi}\int_{0}^{2\pi} u\big(re^{i\theta}\big)\,d\theta = u(0)=\log|1-0|=0.
\]
Therefore $\displaystyle\int_{0}^{2\pi}\log\!\big|1-re^{i\theta}\big|\,d\theta=0$ for all $r\in[0,1)$, and consequently by the dominate convergence theorem
\[
\int_{B_1} \log {1\over |x-(1,0)|} dx =2\int_{0}^{1}0\cdot r\,dr=0.
\]

\medskip
Now we treat the integral on $\mathbb{S}^1$.
Also, the function $z \to \int_{\mathbb{S}^1} \log {1\over |z-z'|} d\sigma (z') $ is constant on $\mathbb{S}^1$, thus we fix $z=(1,0)$. So 
    \[
\int_{\mathbb{S}^1} \log {1\over |(1,0)-z'|} d\sigma (z') =\int_{0}^{2\pi} \log \bigl(2(1-\cos \theta)\bigr)\, d\theta
= \int_{0}^{2\pi}\! \left(\log 4 + 2 \log \left|\sin \tfrac{\theta}{2}\right|\right)\, d\theta.
\]
The first term gives
\[
\int_{0}^{2\pi} \log 4 \, d\theta = 2\pi \log 4.
\]
For the second term, substitute $\phi = \tfrac{\theta}{2}$, so that
\[
\int_{0}^{2\pi} 2 \log \left|\sin \tfrac{\theta}{2}\right|\, d\theta
= 4 \int_{0}^{\pi} \log(\sin \phi)\, d\phi.
\]
It is known that, see for instance~\cite[Formula 4.484.3 page 582]{GRbook}
\[
\int_{0}^{\pi} \log(\sin \phi)\, d\phi = -\pi \log 2.
\]
Hence,
\[
\int_{0}^{2\pi} \log \bigl(2(1-\cos \theta)\bigr)\, d\theta
= 2\pi \log 4 - 4\pi \log 2
= 0.\qedhere
\]
\end{proof}

\subsection{Some geometric quantities related to \texorpdfstring{$\Omega_\varphi$}{Omega phi}.}
\label{app:per}

Here we record some geometric quantities associated with $\Omega_\varphi$ and, in particular, derive formula~\eqref{formula-perimetro} for the perimeter. For $n\ge2$, remember that $\Omega_\varphi = \left\{ x \in \R^N \, : \, |x| < \varphi \left( {x \over |x|} \right) \right\}$, for some $\varphi \in C^{2, \alpha } (\mathbb S^{N-1})$, $\alpha \in (0,1)$, with $\varphi \geq 0$. Then we have the following lemma.
A proof can be found, for example, in~\cite[Proposition 4.1]{CFW}.

\begin{lemma}
\label{lemma:per}
 With the parametrization $x=\varphi(z)z$, $z\in\Sn$, the unit outer normal to $\partial\Omega_\varphi$ at $x$ is given by
 \[
 \nu(x)=\frac{\varphi(z)z-\nabla\varphi(z)}{ \sqrt{\varphi^2(z) + |\nabla \varphi(z) |^2} },\quad z\in\Sn.
 \]
 Moreover, the surface measure $d\sigma(x)$ on $\partial\Omega_\varphi$ is given by
 \[
 d\sigma(x)=\varphi^{N-2}(z) \sqrt{\varphi^2(z) + |\nabla \varphi(z) |^2} \, d\sigma(z),\quad z\in\Sn.
 \]
 In particular, the perimeter of the domain $\Omega_\varphi$ is given by
\[
{\mathrm{Per}} (\Omega_{\varphi })=\int_{\mathbb S^{N-1}} \varphi^{N-2} \sqrt{\varphi^2 + |\nabla \varphi |^2} \, d\sigma.
\]
\end{lemma}

\subsection{The constants \texorpdfstring{$C_{1,\lambda}, C_{2,\lambda}$ and $C(\lambda,N)$}{C1,lambda, C2,lambda and C(lambda,N)}}
\label{app:const}

In this section, we compute the constants $C_{1,\lambda}, C_{2,\lambda}$ and $C(\lambda,N) $. We recall they are defined as
\begin{align*}
C_{1,\lambda} &=\int_{B_1}\frac{1}{|z-y|^\lambda}\,dy,\\
C_{2, \lambda}&=\int_{\mathbb S^{N-1}}\frac{1}{|z-z'|^\lambda}\,d\sigma(z'),\\
 C(\lambda,N) &= (N-\lambda ) \, C_{1,\lambda} - C_{2,\lambda }.
\end{align*}

\begin{lemma}
\label{lemmaC1}
For all $N\ge2$ and $\lambda\in(0,N-1)$, one has
\[
 C_{1,\lambda}= |\mathbb S^{N-2}| \, \Gamma \left( {1\over 2}\right) \, {\Gamma ({N-1\over 2}) \Gamma (N-\lambda) \over 2 \Gamma ({N\over 2} + 1-{\lambda \over 2}) \Gamma (N-{\lambda \over 2}) }.
 \]
\end{lemma}
\begin{proof}
We observe that $C_{1, \lambda}$ does not depends on the choice of $z\in\mathbb S^{N-1}$, so that we can take $z$ as the north pole $e_N=(0,\dots,0,1)$ in the definition $C_{1, \lambda}$. For $y \in B_1$, we perform the following change of variables $y=r \omega$, $\omega \in \mathbb S^{N-1}$, $r \in (0,1)$. Then
\[
C_{1,\lambda}=\int_{B_1}\frac{1}{|z-y|^\lambda}\,dy = \int_{B_1}\frac{1}{|e_N-y|^\lambda}\,dy = \int_0^1 \int_{\mathbb S^{N-1}} {r^{N-1} \over |r \omega - e_N|^\lambda } dr d \omega.
\]
Let $\theta \in [0,\pi]$ be such that $\omega \cdot e_N= \cos \theta$. Then $|r \omega - e_N|^2 = r^2 - 2 r \cos \theta +1$. Besides
\[
\int_{\mathbb S^{N-1}} f (\omega \cdot e_N) d \omega = |\mathbb S^{N-2}| \, \int_0^\pi f (\cos \theta ) \sin^{N-2} \theta \, d\theta,
\]
for any $f$.
Hence
\[
 C_{1, \lambda}= |\mathbb S^{N-2}| \, \int_0^1 r^{N-1} dr \int_0^\pi {\sin^{N-2} \theta \over (r^2 - 2r \cos \theta +1)^{\lambda \over 2}} \, d \theta.
\]
By~\cite[formula 3.665.2, page 409]{GRbook} one has
\[
 \int_0^\pi {\sin^{2\mu -1} \theta \over (a^2 + 2a \cos \theta +1)^{\nu}} \, d \theta = B\left(\mu , {1\over 2}\right) \, F\left(\nu , \nu- \mu +{1\over 2}; \mu+{1\over 2}; a^2\right),
\]
for $Re (\mu )>0$ and $|a|<1$.
We take 
\[
a=-r, \quad \mu={N-1 \over 2}, \quad \nu={\lambda \over 2}, \quad \nu- \mu +{1\over 2}= {\lambda \over 2} -{N \over 2} +1, \quad \mu+{1\over 2} = {N \over 2},
\]
and get 
\begin{align*}
    \int_0^\pi {\sin^{N-2} \theta \over (r^2 - 2r \cos \theta +1)^{\lambda \over 2}} \, d \theta&= 
    B\left( {N-1 \over 2}, {1\over 2}\right) \, F\left({\lambda \over 2} , {\lambda \over 2} -{N \over 2} +1 ; {N\over 2}; r^2\right).
\end{align*}
Now,~\cite[formula 7.512.4, page 821]{GRbook} gives
\begin{align*}
    \int_0^1 x^{\gamma -1} (1-x)^{\rho -1} F(\alpha, \beta; \gamma; x) dx ={\Gamma (\gamma) \Gamma (\rho) \Gamma (\gamma + \rho - \alpha - \beta) \over \Gamma (\gamma + \rho - \alpha ) \Gamma (\gamma + \rho - \beta)},
\end{align*}
provided $Re (\rho )>0$, $Re (\gamma) >0$, $Re (\gamma +\rho -\alpha - \beta) >0.$

We choose
\[
\gamma ={N \over 2} , \quad \rho=1, \quad \alpha ={\lambda \over 2}, \quad \beta = {\lambda \over 2} -{N-2 \over 2},
\]
and check that
\[
\gamma +\rho -\alpha - \beta = N-\lambda >0,
\]
Thus, writing $r^2 = s$,
\begin{align*}
    \int_0^1 r^{N-1 }  \, F\left({\lambda \over 2} , {\lambda \over 2} -{N \over 2} +1 ; {N\over 2}; r^2\right) \, dr  &={1\over 2} \int_0^1 s^{N-2 \over 2} \, F\left({\lambda \over 2} , {\lambda \over 2} -{N \over 2} +1 ; {N\over 2}; s\right) \, ds\\
    &= {\Gamma ({N\over 2}) \Gamma (N-\lambda) \over 2 \Gamma ({N\over 2} + 1-{\lambda \over 2}) \Gamma (N-{\lambda \over 2}) },
\end{align*}
as $\Gamma (1) =1$.
Since $B(a,b)={\Gamma (a) \Gamma (b) \over \Gamma (a+b)}$, we conclude
\begin{align*}
    C_{1, \lambda}&= |\mathbb S^{N-2}| \, B\left({N-1 \over 2}, {1\over 2}\right) \, {\Gamma ({N\over 2}) \Gamma (N-\lambda) \over 2 \Gamma ({N\over 2} + 1-{\lambda \over 2}) \Gamma (N-{\lambda \over 2}) }\\
    &= |\mathbb S^{N-2}| \, \Gamma \left( {1\over 2}\right) \, {\Gamma ({N-1\over 2}) \Gamma (N-\lambda) \over 2 \Gamma ({N\over 2} + 1-{\lambda \over 2}) \Gamma (N-{\lambda \over 2}) }.\qedhere
\end{align*}
\end{proof}

\begin{lemma}
\label{lemmaC2}
For all $N\ge2$ and $\lambda\in(0,N-1)$, one has
\[
 C_{2,\lambda}=|\mathbb S^{N-2}| \, 2^{N-2 -\lambda } {\Gamma ({N-1 -\lambda \over 2}) \Gamma ({N-1 \over 2} ) \over \Gamma (N-1 -{\lambda \over 2})}.
 \]
\end{lemma}
\begin{proof}
We observe that $C_{2, \lambda}$ does not depends on the choice of $z\in\mathbb S^{N-1}$, so that we can take $z$ as the north pole $e_N=(0,\dots,0,1)$ in the definition $C_{2, \lambda}$. Thus
\begin{align*}
    C_{2,\lambda }&= \int_{\mathbb S^{N-1}}\frac{1}{|z-y|^\lambda}\,d\sigma(y)  \\
    &= \int_{\mathbb S^{N-1}}\frac{1}{|e_N-y|^\lambda}\,d\sigma(y) \\
    &= \int_{\mathbb S^{N-1}}\frac{1}{(2(1-y_N))^{\lambda \over 2}}
    \,d\sigma(y)\\
    &= \int_{-1}^1 \int_{ \sqrt{1-y_n^2} \, \mathbb S^{N-2} } \frac{1}{(2(1-y_N))^{\lambda \over 2}}
    \,d\sigma(y) dy_N\\
    &= |\mathbb S^{N-2}| \int_{-1}^1 { (1-y_N^2)^{N-3 \over 2} \over (2(1-y_N))^{\lambda \over 2}} \, dy_N \\
    &= {|\mathbb S^{N-2}| \over 2^{\lambda \over  2} }\int_{-1}^1 (1-y_N)^{N-3 -\lambda \over 2}  \, (1+y_N)^{N-3 \over 2}  \, dy_N,
 \end{align*}
Changing variables $2t= 1-y_N$, we get
\begin{align*}
    C_{2,\lambda }&=2 {|\mathbb S^{N-2}| \over 2^{\lambda \over  2} }\int_{0}^1 (2t)^{N-3 -\lambda \over 2}  \, (2(1-t))^{N-3 \over 2}  \, dt\\
    &=|\mathbb S^{N-2}| \, 2^{N-2 -\lambda }\int_{0}^1 t^{N-3 -\lambda \over 2}  \, (1-t)^{N-3 \over 2}  \, dt \\
    &= |\mathbb S^{N-2}| \, 2^{N-2 -\lambda } B\left({N-3 -\lambda \over 2} +1, {N-3 \over 2}  +1\right),
\end{align*}
where $B(a,b)= \int_0^1 t^{a-1} (1-t)^{b-1}\, dt$ is the Beta function. Since $B(a,b)={\Gamma (a) \Gamma (b) \over \Gamma (a+b)}$, we conclude
\[
    C_{2,\lambda }= |\mathbb S^{N-2}| \, 2^{N-2 -\lambda } {\Gamma ({N-1 -\lambda \over 2}) \Gamma ({N-1 \over 2} ) \over \Gamma (N-1 -{\lambda \over 2})}.\qedhere
    \]
\end{proof}

\begin{lemma}
\label{lemmaC3}
For all $N\ge2$ and $\lambda\in(0,N-1)$, one has
\[
C (\lambda, N)=-\lambda\pi^{\frac N2}\frac{\Gamma(N-1-\lambda)}{ \Gamma (N -{\lambda \over 2})\Gamma(\frac{N-\lambda}{2})}.
 \]
\end{lemma}
\begin{proof}
Taking into account Lemmas~\ref{lemmaC1} and~\ref{lemmaC2} and the Legendre multiplication formula~\eqref{x}, one has
\begin{align*}
     &C (\lambda, N)= (N-\lambda ) \, C_{1,\lambda} - C_{2,\lambda } \\
    & = |\mathbb S^{N-2}| \Gamma \left({N-1 \over 2} \right)\left((N-\lambda){\Gamma ({1\over 2}) \Gamma (N-\lambda) \over 2 \Gamma ({N\over 2} + 1-{\lambda \over 2}) \Gamma (N-{\lambda \over 2}) }-2^{N-2 -\lambda } {\Gamma ({N-1 -\lambda \over 2}) \over \Gamma (N-1 -{\lambda \over 2})}\right)\\
    &= |\mathbb S^{N-2}|{ \Gamma ({N-1\over 2}) \over \Gamma (N-1 -{\lambda \over 2})} \left(  \frac{N-\lambda}{2(N-1 -{\lambda \over 2})}{\Gamma ( {1\over 2}) \,  \Gamma (N-\lambda) \over \Gamma ({N\over 2} + 1-{\lambda \over 2})  } -  2^{N-2 -\lambda }  \Gamma \left({N-1 -\lambda \over 2}\right)   \right) \\
    &=|\mathbb S^{N-2}|\sqrt\pi{ \, \Gamma ({N-1\over 2}) \over \Gamma (N-1 -{\lambda \over 2})} \, \left( \frac{N-\lambda}{2(N-1 -{\lambda \over 2})}{\Gamma (N-\lambda) \over \Gamma ({N\over 2} + 1-{\lambda \over 2})  }-\frac{\Gamma(N-1-\lambda)}{\Gamma(\frac{N-\lambda}{2})}\right)\\
   & =|\mathbb S^{N-2}|\sqrt\pi\frac{\Gamma ({N-1\over 2})\Gamma(N-1-\lambda)}{ \Gamma (N-1 -{\lambda \over 2})\Gamma(\frac{N-\lambda}{2})} \, \left(\frac{N-\lambda-1}{N-1 -{\lambda \over 2}}-1\right)\\
   &=2\pi^{\frac N2}\frac{\Gamma(N-1-\lambda)}{ \Gamma (N-1 -{\lambda \over 2})\Gamma(\frac{N-\lambda}{2})} \, \left(\frac{N-\lambda-1}{N-1 -{\lambda \over 2}}-1\right)\\
   &=-\frac{\lambda\pi^{\frac N2}}{N-1-\lambda/2}\frac{\Gamma(N-1-\lambda)}{ \Gamma (N-1 -{\lambda \over 2})\Gamma(\frac{N-\lambda}{2})}\\
   &=-\lambda\pi^{\frac N2}\frac{\Gamma(N-1-\lambda)}{ \Gamma (N -{\lambda \over 2})\Gamma\left(\frac{N-\lambda}{2}\right)}.\qedhere
\end{align*}
\end{proof}

\bibliographystyle{abbrv}
\bibliography{LiquidDrop.bib}

@book {Wnote,
    AUTHOR = {Wheeler, N.},
     TITLE = { Algebraic theory of spherical harmonics},
   EDITION = {},
      NOTE = {},
 PUBLISHER = {Reed College Physics Department},
      YEAR = {1996},
     PAGES = {},
      ISBN = {},
   MRCLASS = {},
  MRNUMBER = {},
}

@book {GRbook,
    AUTHOR = {Gradshteyn, I. S. and Ryzhik, I. M.},
     TITLE = {Table of integrals, series, and products},
   EDITION = {Seventh},
      NOTE = {},
 PUBLISHER = {Elsevier/Academic Press, Amsterdam},
      YEAR = {2007},
     PAGES = {xlviii+1171},
      ISBN = {978-0-12-373637-6; 0-12-373637-4},
   MRCLASS = {00A22 (33-00 65-00 65A05)},
  MRNUMBER = {2360010},
}

@article {HAZ12,
    AUTHOR = {Han, Weimin and Atkinson, Kendall and Zheng, Hao},
     TITLE = {Some integral identities for spherical harmonics in an
              arbitrary dimension},
   JOURNAL = {J. Math. Chem.},
  FJOURNAL = {Journal of Mathematical Chemistry},
    VOLUME = {50},
      YEAR = {2012},
    NUMBER = {5},
     PAGES = {1126--1135},
      ISSN = {0259-9791,1572-8897},
   MRCLASS = {92E99},
  MRNUMBER = {2969949},
       DOI = {10.1007/s10910-011-9956-7},
       URL = {https://doi.org/10.1007/s10910-011-9956-7},
}

@article {GGT18,
    AUTHOR = {Gladiali, Francesca and Grossi, Massimo and Troestler,
              Christophe},
     TITLE = {Entire radial and nonradial solutions for systems with
              critical growth},
   JOURNAL = {Calc. Var. Partial Differential Equations},
  FJOURNAL = {Calculus of Variations and Partial Differential Equations},
    VOLUME = {57},
      YEAR = {2018},
    NUMBER = {2},
     PAGES = {Paper No. 53, 26},
      ISSN = {0944-2669,1432-0835},
   MRCLASS = {35J47 (35B08 35B09 35B32 35B33)},
  MRNUMBER = {3772879},
MRREVIEWER = {Daniele\ Cassani},
       DOI = {10.1007/s00526-018-1340-z},
       URL = {https://doi.org/10.1007/s00526-018-1340-z},
}

@article {SM90,
    AUTHOR = {Smoller, J. and Wasserman, Arthur G.},
     TITLE = {Bifurcation and symmetry-breaking},
   JOURNAL = {Invent. Math.},
  FJOURNAL = {Inventiones Mathematicae},
    VOLUME = {100},
      YEAR = {1990},
    NUMBER = {1},
     PAGES = {63--95},
      ISSN = {0020-9910,1432-1297},
   MRCLASS = {58E07 (35B32 35J65 58E40)},
  MRNUMBER = {1037143},
MRREVIEWER = {Reiner\ Lauterbach},
       DOI = {10.1007/BF01231181},
       URL = {https://doi.org/10.1007/BF01231181},
}

@article {Fr19,
    AUTHOR = {Frank, Rupert},
     TITLE = {Non-spherical equilibrium shapes in the liquid drop model},
   JOURNAL = {J. Math. Phys.},
  FJOURNAL = {Journal of Mathematical Physics},
    VOLUME = {60},
      YEAR = {2019},
    NUMBER = {7},
     PAGES = {071506, 19},
      ISSN = {0022-2488,1089-7658},
   MRCLASS = {76A02 (49Q20)},
  MRNUMBER = {3981098},
MRREVIEWER = {Georgios\ C.\ Georgiou},
       DOI = {10.1063/1.5095603},
       URL = {https://doi.org/10.1063/1.5095603},
}

@book {KieBook,
    AUTHOR = {Kielh\"ofer, Hansj\"org},
     TITLE = {Bifurcation theory},
    SERIES = {Applied Mathematical Sciences},
    VOLUME = {156},
   EDITION = {Second},
      NOTE = {},
 PUBLISHER = {Springer, New York},
      YEAR = {2012},
     PAGES = {viii+398},
      ISBN = {978-1-4614-0501-6},
   MRCLASS = {47J15 (35B32 37G15 37K50 58E07)},
  MRNUMBER = {2859263},
       DOI = {10.1007/978-1-4614-0502-3},
       URL = {https://doi.org/10.1007/978-1-4614-0502-3},
}

@article {CFW,
    AUTHOR = {Cabr\'e, Xavier and Fall, Mouhamed Moustapha and Weth, Tobias},
     TITLE = {Near-sphere lattices with constant nonlocal mean curvature},
   JOURNAL = {Math. Ann.},
  FJOURNAL = {Mathematische Annalen},
    VOLUME = {370},
      YEAR = {2018},
    NUMBER = {3-4},
     PAGES = {1513--1569},
      ISSN = {0025-5831,1432-1807},
   MRCLASS = {53A10},
  MRNUMBER = {3770173},
MRREVIEWER = {Atsushi\ Fujioka},
       DOI = {10.1007/s00208-017-1559-6},
       URL = {https://doi.org/10.1007/s00208-017-1559-6},
}

@article {Fa18,
    AUTHOR = {Fall, Mouhamed Moustapha},
     TITLE = {Periodic patterns for a model involving short-range and
              long-range interactions},
   JOURNAL = {Nonlinear Anal.},
  FJOURNAL = {Nonlinear Analysis. Theory, Methods \& Applications. An
              International Multidisciplinary Journal},
    VOLUME = {175},
      YEAR = {2018},
     PAGES = {73--107},
      ISSN = {0362-546X,1873-5215},
   MRCLASS = {46N20 (35B32 35B36 35Q40 82D25 82D60)},
  MRNUMBER = {3830723},
MRREVIEWER = {H.\ Hogreve},
       DOI = {10.1016/j.na.2018.05.009},
       URL = {https://doi.org/10.1016/j.na.2018.05.009},
}

@article {CR73,
    AUTHOR = {Crandall, Michael G. and Rabinowitz, Paul H.},
     TITLE = {Bifurcation from simple eigenvalues},
   JOURNAL = {J. Functional Analysis},
  FJOURNAL = {Journal of Functional Analysis},
    VOLUME = {8},
      YEAR = {1971},
     PAGES = {321--340},
      ISSN = {0022-1236},
   MRCLASS = {47.80 (35.00)},
  MRNUMBER = {288640},
MRREVIEWER = {B.\ V.\ Loginov},
       DOI = {10.1016/0022-1236(71)90015-2},
       URL = {https://doi.org/10.1016/0022-1236(71)90015-2},
}

@article{bohr,
    AUTHOR = {Bohr, Niels},
     TITLE = {Neutron capture and nuclear constitution},
   JOURNAL = {Nature},
  FJOURNAL = {Nature},
    VOLUME = {137},
      YEAR = {1936},
     PAGES = {344--348},
       DOI = {10.1038/137344a0},
       URL = {https://doi.org/10.1038/137344a0},
}

@article{BohrWheeler1939,
    AUTHOR = {Bohr, N. and Wheeler, J. A.},
    TITLE = {The mechanism of nuclear fission},
    JOURNAL = {Physical Review},
    FJOURNAL = {Physical Review},
    VOLUME = {56},
    YEAR = {1939},
    NUMBER = {5},
    PAGES = {426--450},
    DOI = {10.1103/PhysRev.56.426},
    URL = {https://doi.org/10.1103/PhysRev.56.426},
}

@article{BonaciniCristoferi2014,
    AUTHOR = {Bonacini, M. and Cristoferi, R.},
    TITLE = {Local and global minimality results for a nonlocal isoperimetric problem on $\mathbb{R}^N$},
    JOURNAL = {SIAM J. Math. Anal.},
    FJOURNAL = {SIAM Journal on Mathematical Analysis},
    VOLUME = {46},
    YEAR = {2014},
    NUMBER = {4},
    PAGES = {2310--2349},
    DOI = {10.1137/13092697X},
    URL = {https://doi.org/10.1137/13092697X},
}

@article{ChoksiPeletier2011,
    AUTHOR = {Choksi, R. and Peletier, M.A.},
    TITLE = {Small volume-fraction limit of the diblock copolymer problem: {II}. Diffuse-interface functional},
    JOURNAL = {SIAM J. Math. Anal.},
    FJOURNAL = {SIAM Journal on Mathematical Analysis},
    VOLUME = {43},
    YEAR = {2011},
    NUMBER = {2},
    PAGES = {739--763},
    DOI = {10.1137/100788492},
    URL = {https://doi.org/10.1137/100788492},
}

@article{ChodoshRuohoniemi2025,
    AUTHOR = {Chodosh, O. and Ruohoniemi, I.},
    TITLE = {On minimizers in the liquid drop model},
    JOURNAL = {Comm. Pure Appl. Math.},
    FJOURNAL = {Communications on Pure and Applied Mathematics},
    VOLUME = {78},
    YEAR = {2025},
    PAGES = {366--381},
    DOI = {10.1002/cpa.21893},
    URL = {https://doi.org/10.1002/cpa.21893},
}

@article{ChoksiMuratovTopaloglu2017,
    AUTHOR = {Choksi, R. and Muratov, C. and Topaloglu, I.},
    TITLE = {An old problem resurfaces nonlocally: Gamow's liquid drops inspire today's research and applications},
    JOURNAL = {Notices Amer. Math. Soc.},
    FJOURNAL = {Notices of the American Mathematical Society},
    VOLUME = {64},
    YEAR = {2017},
    NUMBER = {11},
    PAGES = {1275--1283},
    URL = {https://www.ams.org/notices/201711/201711-full-issue.pdf},
}

@article{delPinoMussoZuniga2025,
    AUTHOR = {del Pino, M. and Musso, M. and Zuniga, A.},
    TITLE = {Delaunay-like compact equilibria in the liquid drop model},
    JOURNAL = {Arch. Ration. Mech. Anal.},
    FJOURNAL = {Archive for Rational Mechanics and Analysis},
    VOLUME = {249},
    YEAR = {2025},
    NUMBER = {6},
    PAGES = {Paper No. 74},
    DOI = {10.1007/s00205-024-02255-2},
    URL = {https://doi.org/10.1007/s00205-024-02255-2},
}

@article{FigalliFuscoMaggiMillotMorini2015,
    AUTHOR = {Figalli, A. and Fusco, N. and Maggi, F. and Millot, V. and Morini, M.},
    TITLE = {Isoperimetry and stability properties of balls with respect to nonlocal energies},
    JOURNAL = {Comm. Math. Phys.},
    FJOURNAL = {Communications in Mathematical Physics},
    VOLUME = {336},
    YEAR = {2015},
    NUMBER = {1},
    PAGES = {441--507},
    DOI = {10.1007/s00220-014-2223-1},
    URL = {https://doi.org/10.1007/s00220-014-2223-1},
}

@article{FrankKillipNam2016,
    AUTHOR = {Frank, R. and Killip, R. and Nam, P.T.},
    TITLE = {Nonexistence of large nuclei in the liquid drop model},
    JOURNAL = {Lett. Math. Phys.},
    FJOURNAL = {Letters in Mathematical Physics},
    VOLUME = {106},
    YEAR = {2016},
    NUMBER = {8},
    PAGES = {1033--1036},
    DOI = {10.1007/s11005-016-0907-1},
    URL = {https://doi.org/10.1007/s11005-016-0907-1},
}

@article{FrankNam2021,
    AUTHOR = {Frank, R. and Nam, P.T.},
    TITLE = {Existence and nonexistence in the liquid drop model},
    JOURNAL = {Calc. Var. Partial Differential Equations},
    FJOURNAL = {Calculus of Variations and Partial Differential Equations},
    VOLUME = {60},
    YEAR = {2021},
    NUMBER = {6},
    PAGES = {Paper No. 223, 12 pp.},
    DOI = {10.1007/s00526-021-02066-5},
    URL = {https://doi.org/10.1007/s00526-021-02066-5},
}

@article{Gamow1930,
    AUTHOR = {Gamow, G.},
    TITLE = {Mass defect curve and nuclear constitution},
    JOURNAL = {Proc. R. Soc. Lond. A},
    FJOURNAL = {Proceedings of the Royal Society of London. Series A},
    VOLUME = {126},
    YEAR = {1930},
    NUMBER = {803},
    PAGES = {632--644},
    DOI = {10.1098/rspa.1930.0140},
    URL = {https://doi.org/10.1098/rspa.1930.0140},
}

@article{Julin2014,
    AUTHOR = {Julin, V.},
    TITLE = {Isoperimetric problem with a Coulomb repulsive term},
    JOURNAL = {Indiana Univ. Math. J.},
    FJOURNAL = {Indiana University Mathematics Journal},
    VOLUME = {63},
    YEAR = {2014},
    NUMBER = {1},
    PAGES = {77--89},
    DOI = {10.1512/iumj.2014.63.5347},
    URL = {https://doi.org/10.1512/iumj.2014.63.5347},
}

@article{KnupferMuratov2014,
    AUTHOR = {Kn\"upfer, H. and Muratov, C.},
    TITLE = {On an isoperimetric problem with a competing nonlocal term {II}: The general case},
    JOURNAL = {Comm. Pure Appl. Math.},
    FJOURNAL = {Communications on Pure and Applied Mathematics},
    VOLUME = {67},
    YEAR = {2014},
    NUMBER = {12},
    PAGES = {1974--1994},
    DOI = {10.1002/cpa.21445},
    URL = {https://doi.org/10.1002/cpa.21445},
}

@article {KM14,
    AUTHOR = {Kn\"upfer, Hans and Muratov, Cyrill B.},
     TITLE = {On an isoperimetric problem with a competing nonlocal term
              {I}: {T}he planar case},
   JOURNAL = {Comm. Pure Appl. Math.},
  FJOURNAL = {Communications on Pure and Applied Mathematics},
    VOLUME = {66},
      YEAR = {2013},
    NUMBER = {7},
     PAGES = {1129--1162},
      ISSN = {0010-3640,1097-0312},
   MRCLASS = {49Q05 (35J20)},
  MRNUMBER = {3055587},
MRREVIEWER = {Otis\ Chodosh},
       DOI = {10.1002/cpa.21451},
       URL = {https://doi.org/10.1002/cpa.21451},
}

@book{LiebLoss1997,
    AUTHOR = {Lieb, E. and Loss, M.},
    TITLE = {Analysis},
    SERIES = {Graduate Studies in Mathematics},
    VOLUME = {14},
    PUBLISHER = {American Mathematical Society},
    ADDRESS = {Providence, RI},
    YEAR = {1997},
}

@article{LuOtto2014,
    AUTHOR = {Lu, J. and Otto, F.},
    TITLE = {Nonexistence of a minimizer for {T}homas-{F}ermi-{D}irac-von {W}eizs\"acker model},
    JOURNAL = {Comm. Pure Appl. Math.},
    FJOURNAL = {Communications on Pure and Applied Mathematics},
    VOLUME = {67},
    YEAR = {2014},
    NUMBER = {10},
    PAGES = {1605--1617},
    DOI = {10.1002/cpa.21430},
    URL = {https://doi.org/10.1002/cpa.21430},
}

@book{Maggi2012,
    AUTHOR = {Maggi, F.},
    TITLE = {Sets of finite perimeter and geometric variational problems: An introduction to geometric measure theory},
    SERIES = {Cambridge Studies in Advanced Mathematics},
    VOLUME = {135},
    PUBLISHER = {Cambridge University Press},
    ADDRESS = {Cambridge},
    YEAR = {2012},
}

@article{RenWei1,
    AUTHOR = {Ren, X. and Wei, J.},
    TITLE = {A toroidal tube solution to a problem involving mean curvature and Newtonian potential},
    JOURNAL = {Interfaces Free Bound.},
    FJOURNAL = {Interfaces and Free Boundaries},
    VOLUME = {13},
    YEAR = {2011},
    NUMBER = {1},
    PAGES = {127--154},
    DOI = {10.4171/IFB/230},
    URL = {https://doi.org/10.4171/IFB/230},
}

@article{RenWei2,
    AUTHOR = {Ren, X. and Wei, J.},
    TITLE = {Double tori solution to an equation of mean curvature and Newtonian potential},
    JOURNAL = {Calc. Var. Partial Differential Equations},
    FJOURNAL = {Calculus of Variations and Partial Differential Equations},
    VOLUME = {49},
    YEAR = {2014},
    NUMBER = {3-4},
    PAGES = {987--1018},
    DOI = {10.1007/s00526-013-0634-1},
    URL = {https://doi.org/10.1007/s00526-013-0634-1},
}

@article{XuDu2023,
    AUTHOR = {Xu, Z. and Du, Q.},
    TITLE = {Bifurcation and fission in the liquid drop model: a phase-field approach},
    JOURNAL = {J. Math. Phys.},
    FJOURNAL = {Journal of Mathematical Physics},
    VOLUME = {64},
    YEAR = {2023},
    NUMBER = {7},
    PAGES = {Paper 071701, 21 pp.},
    DOI = {10.1063/5.0144521},
    URL = {https://doi.org/10.1063/5.0144521},
}

@article {RWW25,
    AUTHOR = {Ren, Xiaofeng and Wang, Chong and Wei, Juncheng},
     TITLE = {A geometric variational problem with logarithmic-quadratic
              interaction},
   JOURNAL = {SIAM J. Math. Anal.},
  FJOURNAL = {SIAM Journal on Mathematical Analysis},
    VOLUME = {57},
      YEAR = {2025},
    NUMBER = {3},
     PAGES = {2497--2532},
      ISSN = {0036-1410,1095-7154},
   MRCLASS = {49Q20 (82B24 82D60 92C15)},
  MRNUMBER = {4905063},
       DOI = {10.1137/24M1677605},
       URL = {https://doi.org/10.1137/24M1677605},
}

@article{RZ25,
author = {Ren, Xiaofeng and Zhang, Guanning},
title = {Splintering and coarsening in a growth and inhibition system},
journal = {Communications in Contemporary Mathematics},
volume = {0},
number = {0},
pages = {2550084},
year = {0},
doi = {10.1142/S0219199725500841},
URL = {https://doi.org/10.1142/S0219199725500841},
}

@article {FL15,
    AUTHOR = {Frank, Rupert L. and Lieb, Elliott H.},
     TITLE = {A compactness lemma and its application to the existence of
              minimizers for the liquid drop model},
   JOURNAL = {SIAM J. Math. Anal.},
  FJOURNAL = {SIAM Journal on Mathematical Analysis},
    VOLUME = {47},
      YEAR = {2015},
    NUMBER = {6},
     PAGES = {4436--4450},
      ISSN = {0036-1410,1095-7154},
   MRCLASS = {49Q10 (46T99 49J40 49Q20 81V35)},
  MRNUMBER = {3425373},
MRREVIEWER = {Donghui\ Yang},
       DOI = {10.1137/15M1010658},
       URL = {https://doi.org/10.1137/15M1010658},
}

\end{document}